\documentclass[12pt, reqno]{amsart}
\usepackage{amssymb,latexsym,amsmath,amscd,amsthm,graphicx, color}
\usepackage[all]{xy}
\usepackage{pgf,tikz}
\usepackage{mathrsfs}
\usetikzlibrary{arrows}
\usepackage[left=0.6 in, top=0.6 in, right=0.6 in, bottom=0.3 in]{geometry}
\allowdisplaybreaks
\makeatletter
\newcommand{\setword}[2]{%
	\phantomsection
	#1\def\@currentlabel{\unexpanded{#1}}\label{#2}%
}
\makeatother

\usepackage[linkcolor=blue, urlcolor=blue, citecolor=blue,
colorlinks, bookmarks]{hyperref}
\definecolor{uuuuuu}{rgb}{0.26666666666666666,0.26666666666666666,0.26666666666666666}
\definecolor{xdxdff}{rgb}{0.49019607843137253,0.49019607843137253,1.}
\definecolor{ffqqqq}{rgb}{1.,0.,0.}
\definecolor{ffqqqq}{rgb}{1.,0.,0.}
\definecolor{ffxfqq}{rgb}{1.,0.4980392156862745,0.}

\definecolor{uuuuuu}{rgb}{0.26666666666666666,0.26666666666666666,0.26666666666666666}
\definecolor{qqwuqq}{rgb}{0.,0.39215686274509803,0.}
\definecolor{zzttqq}{rgb}{0.6,0.2,0.}
\definecolor{xdxdff}{rgb}{0.49019607843137253,0.49019607843137253,1.}
\definecolor{qqqqff}{rgb}{0.,0.,1.}
\definecolor{cqcqcq}{rgb}{0.7529411764705882,0.7529411764705882,0.7529411764705882}
\definecolor{sqsqsq}{rgb}{0.12549019607843137,0.12549019607843137,0.12549019607843137}

\theoremstyle{plain}

\newtheorem{theorem}[subsection]{Theorem}

\newtheorem{corollary}[subsection]{Corollary}
\newtheorem{lemma}[subsection]{Lemma}
\newtheorem{defi}[subsection]{Definition}
\newtheorem{prop}[subsection]{Proposition}

\theoremstyle{definition}

\newtheorem{exam}[subsection]{Example}

\newtheorem{remark}[subsection]{Remark}

\newtheorem{note}[subsection]{Note}

\newtheorem{method}[subsection]{Methodology}

\newtheorem{main}[subsection]{Main Result}

\newcommand{\uu}{\cup}% union
\newcommand{\ii}{\cap}% intersection
\newcommand{\UU}{\bigcup}% big union
\newcommand{\II}{\bigcap}% big intersection
\newcommand{\ci}{\subseteq}% contained in with equality
\newcommand{\sci}{\subset}% strictly contained in
\newcommand{\es}{\emptyset}% the empty set
\newcommand{\set}[1]{\{#1\}}% set
\newcommand{\ga}{\alpha}
\newcommand{\gb}{\beta}

\renewcommand{\gg}{\gamma}% old use >>

\newcommand{\gk}{\kappa}

\newcommand{\gq}{\theta}

\newcommand{\tit}{\textit}% text italic

\newcommand{\D}[1]{\mathbb{#1}}% Doubled - blackboard bold - only caps, uas as \D{A}
\newcommand{\te}{\text}% same as \mathrm command.

\newcommand{\nd}{\noindent}

\newcommand{\pa}{\partial}

\newcommand{\tri}{\triangle}
\begin{document}

\nd To appear, in the journal `Real Analysis Exchange'
	\title{Conditional constrained and unconstrained quantization for uniform distributions on regular polygons}
	
	\address{$^{1}$Department of Mathematics\\
		Concordia University, Austin, Texas 78726, USA.}
	\address{$^{2}$Russ College of Engineering and Technology\\ 
Department of Industrial and System Engineering, Ohio University, Athens, Ohio 45701, USA.}

	\address{$^{3}$School of Mathematics \\
		Northwest University Xi'an, Shaanxi Province, 710069, China}
	 
	\address{$^{4}$School of Mathematical and Statistical Sciences\\
		University of Texas Rio Grande Valley\\
		1201 West University Drive\\
		Edinburg, TX 78539-2999, USA.}

	\author{$^1$Christina Hamilton}
	\author{$^2$Evans Nyanney}
	\author{$^{3}$Megha Pandey}
	\author{$^4$Mrinal K. Roychowdhury}

	\email{$^1$christina.hamilton@concordia.edu, $^2$en596624@ohio.edu, $^{3}$meghapandey1071996@gmail.com}
	\email{$^4$mrinal.roychowdhury@utrgv.edu}

	% \thanks{A part of the results in this paper was used in partial fulfillment of the first author's Master's thesis at the University of Texas Rio Grande Valley under the direction of the last author.}

	
	\date{}
	\maketitle
	\pagestyle{myheadings}\markboth{C. Hamilton, E. Nyanney, M. Pandey, and M.K. Roychowdhury }
	{Conditional constrained and unconstrained quantization for uniform distributions on regular polygons}
	
	\begin{abstract}
		In this paper, we have considered a uniform distribution on a regular polygon with $k$-sides for some $k\geq 3$ and the set of all its $k$ vertices as a conditional set. For the uniform distribution under the conditional set first, for all positive integers $n\geq k$, we obtain the conditional optimal sets of $n$-points and the $n$th conditional quantization errors, and then we calculate the conditional quantization dimension and the conditional quantization coefficient in the unconstrained scenario. Then, for the uniform distribution on the polygon taking the same conditional set, we investigate the conditional constrained optimal sets of $n$-points and the conditional constrained quantization errors for all $n \geq 6$, taking the constraint as the circumcircle, incircle, and then the different diagonals of the polygon.
	\end{abstract}

	\section{Introduction}
	Constrained quantization has recently been introduced by Pandey and Roychowdhury (see \cite{PR2, PR1}). They have also introduced conditional quantization in both constrained and unconstrained quantization (see \cite{PR4}). For some follow up papers in the direction of constrained quantization and conditional quantization one can see \cite{BCDRV,BCDR,PR3,PR5}. With the introduction of constrained quantization, quantization now has two classifications: Constrained Quantization and Unconstrained Quantization. Unconstrained quantization is traditionally known as quantization.
	For unconstrained quantization and its applications one can see \cite{DR,DFG,GG,GL2, GL3,GL,GN,GL1, KNZ, P, P1, R1, R2, R3, Z1, Z2}. Constrained quantization has much more interdisciplinary applications in information theory, machine learning and data compression, signal processing, and national security (see Subsection \ref{MotSig}). 
	Let $P$ be a Borel probability measure on $\D R^2$ equipped with a Euclidean metric $d$ induced by the Euclidean norm $\|\cdot\|$.

	\begin{defi}  \label{defi0} 
		Let $S\ci \D R^2$ be nonempty and closed. Let $\gb\ci \D R^2$ be given with $\te{card}(\gb)=\ell$ for some $\ell\in \D N$. 
		Then, for $n\in \D N$ with $n\geq \ell$, the \tit {$n$th conditional constrained quantization
			error} for $P$ with respect to the constraint $S$ and the set $\gb$, is defined as
		\begin{equation}  \label{eq0} 
			V_{n}:=V_{n}(P)=\inf_{\ga} \Big\{\int \mathop{\min}\limits_{a\in\ga\uu\gb} d(x, a)^2 dP(x) : \ga \ci S, ~ 0\leq  \text{card}(\ga) \leq n-\ell \Big\},
		\end{equation} 
		where $\te{card}(A)$ represents the cardinality of the set $A$. 
	\end{defi} 
	If $S=\D R^2$, i.e., if there is no constraint, then the $n$th conditional constrained quantization error, given in Definition~\ref{defi0}, reduces to the definition of the $n$th conditional unconstrained quantization error, which is also refereed to as \tit{$n$th conditional quantization error}, i.e., the \tit{$n$th conditional quantization error} is defined as 
	
	\begin{equation}  \label{eq1} 
		V_{n}:=V_{n}(P)=\inf_{\ga} \Big\{\int \mathop{\min}\limits_{a\in\ga\uu\gb} d(x, a)^2 dP(x) : \ga \ci \D R^2, ~ 0\leq  \text{card}(\ga) \leq n-\ell \Big\}.
	\end{equation} 
	For any $\gg\ci \D R^2$, the number 
	\begin{equation*}
		V(P; \gg):= \int \mathop{\min}\limits_{a\in\gg} d(x, a)^2 dP(x)
	\end{equation*}
	is called the \tit{distortion error} for $P$ with respect to the set $\gg$.  
	We assume that $\int d(x, 0)^2 dP(x)<\infty$ to make sure that the infimums in \eqref{eq0} and \eqref{eq1} exist (see \cite{PR1}). For a finite set $\gg \sci \D R^2$ and $a\in \gg$, by $M(a|\gg)$ we denote the set of all elements in $\D R^2$ which are nearest to $a$ among all the elements in $\gg$, i.e.,
	$M(a|\gg)=\set{x \in \D R^2 : d(x, a)=\mathop{\min}\limits_{b \in \gg}d(x, b)}.$
	$M(a|\gg)$ is called the \tit{Voronoi region} in $\D R^2$ generated by $a\in \gg$. 
	
	\begin{defi} \label{defiM1}
		A set $ \ga\uu\gb$, where $P(M(b|\ga\uu \gb))>0$ for $b\in \gb$, for which the infimum in $V_n$ exists and contains no less than $\ell$ elements, and no more than $n$ elements is called an \tit{optimal set of $n$-points} for $P$, or more specifically a  \tit{conditional optimal set of $n$-points} for $P$. Elements of an optimal set are called \tit{optimal elements}. 
	\end{defi}
	Let $V_{n}(P)$ be a strictly decreasing sequence, and write $V_{\infty}(P):=\mathop{\lim}\limits_{n\to \infty} V_{n}(P)$.  Then, the number 
	\begin{align*}
		D(P):=\mathop{\lim}\limits_{n\to \infty}  \frac{2\log n}{-\log (V_{n}(P)-V_{\infty}(P))}
	\end{align*} 
	if it exists, 
	is called the \tit{conditional constrained quantization dimension} (or \tit{conditional quantization dimension}) of $P$ if $V_n(P)$ represents the $n$th conditional constrained quantization error (or conditional quantization error), and is denoted by $D(P)$. The quantization dimension measures the speed how fast the specified measure of the quantization error converges as $n$ tends to infinity.
	For any $\gk>0$, the number 
	\[\lim_n n^{\frac 2 \gk}  (V_{n}(P)-V_{\infty}(P)),\]
	if it exists, is called the $\gk$-dimensional \tit{conditional constrained quantization coefficient} (or \tit{conditional quantization coefficient}) for $P$ if $V_n(P)$ represents the $n$th conditional constrained quantization error (or conditional quantization error). 
	For $k\in \D N$ with $k\geq 3$, let us consider a regular polygon  with $k$-sides inscribed in a unit circle with center $(0, 0)$, known as \tit{regular $k$-gon}, with vertices $A_j$, where 
	\[A_j := \Big(\cos (\frac {3\pi}2 +\frac{\pi  (2 j-3)}{k}), \sin(\frac {3\pi}2 +\frac{\pi  (2 j-3)}{k})\Big) =\Big(\sin (\frac{\pi  (2 j-3)}{k}), -\cos(\frac{\pi  (2 j-3)}{k})\Big)\te{ for } 1\leq j\leq k.\]
	Notice that the length of each side of such a regular polygon inscribed in a unit circle is $2\sin \frac{\pi}k$. 
	
	\begin{remark}
		The vertices of a regular polygon inscribed in a unit circle can be chosen in an infinitely many different ways. The vertices of the regular polygon in this paper are chosen in such a way so that the side $A_1A_2$ always remains below and parallel to the horizontal axis.
	\end{remark}  
	If $k=6$, it becomes a regular hexagon with vertices $A_1(-\frac 12, -\frac{\sqrt 3}{2})$, $A_2(\frac 12, -\frac{\sqrt 3}{2})$, $A_3(1, 0)$, $A_4(\frac 12, \frac{\sqrt 3}{2})$, $A_5(-\frac 12, \frac{\sqrt 3}{2})$, and $A_6(-1, 0)$ as shown in Figure~\ref{Fig1}. Let us take the conditional set $\gb$ as the set of all the vertices of the regular $k$-gon. Let $P$ be a uniform probability distribution on the regular $k$-gon. In this paper, in Section~\ref{sec2} for the uniform distribution $P$, first we calculate the conditional optimal sets of $n$-points and the $n$th conditional quantization errors for all $n\in \D N$ with $n\geq k$, and then we calculate the conditional quantization dimension and the conditional quantization coefficient. Then, for the uniform distribution on the regular $k$-gon, for $n\geq k$, we calculate the conditional constrained optimal sets $\ga_n$ of $n$-points and the corresponding $n$th conditional constrained quantization errors $V_n$ for the same conditional set $\gb$ and taking different constraints: in Section~\ref{sec3} taking the constraint $S$ as the circumcircle, and in Section~\ref{sec4} taking the constraint as the incircle of the polygon. Finally, in Section~\ref{sec5}, taking $k=6$, i.e., for the uniform distribution on a regular hexagon inscribed in a unit circle taking the conditional set as all of its six vertices we calculate the conditional constrained optimal sets $\ga_n$ of $n$-points and the corresponding $n$th conditional constrained quantization errors $V_n$ for the constraint as a diagonal of the hexagon: in Subsection~\ref{sub1} the constraint is taken as one of its larger diagonals, and in Subsection~\ref{sub2} the constraint is taken as one of its smaller diagonals.

	\subsection{Motivation and Significance.} \label{MotSig} 
	To keep our calculation simple and to avoid many technical difficulties in this paper, we have investigated the conditional quantization in both constrained and unconstrained scenarios for a uniform distribution on a regular polygon. By using the similar technique or with a major overhaul, one can investigate the conditional quantization for any probability distribution defined on a polygon with respect to any conditional set. Conditional quantization has significant interdisciplinary applications: for example, in radiation therapy of cancer treatment to find the optimal locations of $n$ centers of radiation, where $k$ centers for some $k<n$ of radiation are preselected, the conditional quantization technique can be used. As an another application, we can say that the amount of agricultural water usage needs to be controlled by placing a minimal number of water sprinkles (or any other resources) in a way that they cover the whole cropland as required. Suppose for some positive integer $k$, there are already $k$ sprinklers installed. Now, to install more sprinklers not removing the previously installed sprinklers in an optimal way, the conditional quantization technique can be used. Conditional quantization has recently been introduced by Pandey-Roychowdhury in \cite{PR4}. There are many interesting open problems that can be investigated in the area of Mathematics and Statistics.  The work in this paper is an advancement in this direction.

	\section{Preliminaries} \label{preliminary}  \label{secPre}
	For any two elements $(a, b)$ and $(c, d)$ in $\D R^2$, we write 
	\[\rho((a, b), (c, d)):=(a-c)^2 +(b-d)^2,\] which gives the squared Euclidean distance between the two elements $(a, b)$ and $(c, d)$. Let $p$ and $q$ be two elements that belong to an optimal set of $n$-points for some positive integer $n$. Then, $p$ and $q$  are called \tit{adjacent elements} if they have a common boundary in their own Voronoi regions. Let $e$ be an element on the common boundary of the Voronoi regions of the adjacent elements $p$ and $q$.
	Since the common boundary of the Voronoi regions of any two adjacent elements is the perpendicular bisector of the line segment joining the elements, we have
	\[\rho(p, e)-\rho(q, e)=0. \]
	We call such an equation a \tit{canonical equation}. Let $L$ be the boundary of the regular polygon $A_1A_2\cdots A_k$ with $k$ sides inscribed in a unit circle with center $(0, 0)$  as described in the previous section. Let $L_j$ represent the side $A_jA_{j+1}$ for $1\leq j\leq k$, where the vertex $A_{k+1}$ is identified with the vertex $A_1$. Then, we have 
	\begin{equation} \label{eqMM0} L=\UU_{j=1}^k L_j.
	\end{equation} 
	Let $s$ represent the distance of any point on the regular polygon from the vertex $A_1$ tracing along the boundary $L$ in the counterclockwise direction. Then, the vertices $A_1, A_2, \cdots, A_k$ are, respectively, represented by $s=0, s=2\sin \frac{\pi}k, \cdots, s=2(k-1)\sin \frac{\pi}k$. 
	The probability density function (pdf) $f$ of the uniform distribution $P$ is given by $f(s):=f(x, y)=\frac 1{k}$ for all $(x, y)\in L$, and zero otherwise. The sides $L_j$ are represented by the parametric equations as follows:
	\begin{align*}
		L_j&=tA_{j+1}+(1-t)A_j \\
		&=\Big\{\Big(t \sin\frac{\pi(2j-1)}{k}+(1-t)\sin \frac{\pi(2j-3)}{k}, -t \cos\frac{\pi(2j-1)}{k}-(1-t)\cos \frac{\pi(2j-3)}{k}\Big) : 0\leq t\leq 1\Big\}. 
	\end{align*}
	Thus, in the case of a hexagon, the sides $L_1, L_2, L_3, L_4, L_5, L_6$ are represented by the parametric equations as follows:
	\begin{align*}
		L_1=\set {(t-\frac{1}{2},-\frac{\sqrt{3}}{2})},  &\quad L_2 =\{(\frac{t+1}{2},\frac{1}{2} \sqrt{3} (t-1))\}, \quad L_3=\{(1-\frac{t}{2},\frac{\sqrt{3} t}{2})\}, \quad L_4=\{(\frac{1}{2}-t,\frac{\sqrt{3}}{2})\}, \\
		& L_5=\{(\frac{1}{2} (-t-1),-\frac{1}{2} \sqrt{3} (t-1))\}, \quad L_6=\{(\frac{t-2}{2},-\frac{\sqrt{3} t}{2})\},
	\end{align*}
	where $0\leq t\leq 1$. 
	Again, $dP(s)=P(ds)=f(x_1, x_2) ds$, where $d$ stands for differential. Notice that on each of $L_j$, for $1\leq j\leq k$, we have $ds=\sqrt{(dx)^2+(dy)^2}=2 \sin \frac {\pi} k\, dt$.
	The conditional set $\gb$ considered in this paper is the set  
	\[\gb=\Big\{A_j  : 1\leq j\leq k \Big\}.\]
	Let $T: \D R^2 \to \D R^2$ be an affine transformations such that $T(x, y)=(ax+by, cx+dy)$ for $(x, y) \in \D R^2$, where 
	\[a=\cos  \frac{2 \pi }{k}, \, b= -\sin  \frac{2 \pi }{k}, \, c= \sin  \frac{2 \pi }{k}, \te{ and } d= \cos  \frac{2 \pi }{k}.\] Notice that $T$ is a rotation map with angle $\frac{2\pi}{k}$ in $\mathbb{R}^2$.
	Also, for any $j\in \D N$, by $T^j$ it is meant the composition mapping 
	$T^j=T\circ T\circ T\circ\cdots j\te{-times}.$ If $j=0$, i.e., by $T^0$ it is meant the identity mapping on $\D R^2$. Then, notice that 
	\begin{align*} T(A_j)=A_{j+1} \te{ and }  T (L_j)=L_{j+1} \te{ for any } j\in \D N, 
	\end{align*}
	where $A_{k+1} \te{ is identified as } A_1 \te{ and } L_{k+1} \te{ is identified as } L_1$.

The following proposition is a modified version of Proposition~2.1 in \cite{PR3}.
	\begin{prop}\emph{(see \cite{PR3})} \label{Me0} 
		Let $\ga_n$ be an optimal set of $n$-points for $P$ such that $\ga_n$ contains $m$ elements, for some $m\leq n$, from the side $L_1$ including the endpoints $A_1$ and $A_2$. Then,
		\[\ga_n\ii L_1=\Big\{A_1+\frac{j-1}{m-1}(A_2-A_1): 1\leq j\leq m\Big\}=\Big\{\Big(\frac{(2 j-m-1) \sin \frac{\pi }{k}}{m-1},-\cos\frac{\pi }{k}\Big): 1\leq j\leq m\Big\}\] 
		and the distortion error contributed by these $m$ elements from the side $L_1$ is given by 
		\[V(P; \ga_n\ii L_1)=\frac{2\sin ^3 \frac{\pi }{k}}{3 k (m-1)^2}.\]
	\end{prop}

	\begin{corollary}   
		Let $\ga_n$ be an optimal set of $n$-points for $P$ such that $\ga_n$ contains $m$ elements, for some $m\leq n$, from the side $L_1$ including the endpoints $A_1$ and $A_2$. Then, if $k=6$, i.e., in the case of a uniform distribution on a  regular hexagon, we have  
		\[\ga_n\ii L_1=\Big\{\Big(-\frac 1 {2}+\frac{j-1}{m-1}, \, -\frac {\sqrt 3}{2}\Big) : 1\leq j\leq m\Big\},\] 
		and the distortion error contributed by these $m$ elements from the side $L_1$ is given by 
		\[V(P; \ga_n\ii L_1)=\frac {1} {72 (m-1)^2}.\]
	\end{corollary}

	In the following sections, we give the main results of the paper.
	
	\begin{figure}
		\vspace{-0.3 in} 
		\begin{tikzpicture}[line cap=round,line join=round,>=triangle 45,x=0.5 cm,y=0.5 cm]
			\clip(-7.5,-6.5) rectangle (6.9,6.9);
			\draw [line width=0.8 pt] (-2.,3.4641016151377544)-- (-4.,0.)-- (-2.,-3.4641016151377544)-- (2.,-3.4641016151377544)-- (4.,0.)-- (2.,3.4641016151377544);
			\draw [line width=0.8 pt] (-2.,3.4641016151377544)-- (2.,3.4641016151377544);
			\draw [line width=0.8 pt] (0.,0.) circle (2.cm);
			\draw [line width=0.8 pt] (0.,0.) circle (3.4641);
			\draw (-6.66,-3.10) node[anchor=north west] {$A_1(-\frac 12, -\frac{\sqrt{3}}{2})$};
			\draw (1.60,-3.05) node[anchor=north west] {$A_2(\frac 12, -\frac{\sqrt{3}}{2})$};
			\draw (1.8,4.66) node[anchor=north west] {$A_4(\frac 12, \frac{\sqrt{3}}{2})$};
			\draw (-6.5,4.66) node[anchor=north west] {$A_5(-\frac 12, \frac{\sqrt{3}}{2})$};
			\draw (3.8,0.65) node[anchor=north west] {$A_3(1, 0)$};
			\draw (-7.9,0.65) node[anchor=north west] {$A_6(-1, 0)$};
			\draw [line width=0.2 pt,dash pattern=on 1pt off 1pt] (0.,0.)-- (-2.,-3.4641016151377544);
			\draw [line width=0.2,dash pattern=on 1pt off 1pt] (0.,0.)-- (2.,-3.4641016151377544);
			\draw [line width=0.2,dash pattern=on 1pt off 1pt] (0.,0.)-- (2.,3.4641016151377544);
			\draw [line width=0.2,dash pattern=on 1pt off 1pt] (0.,0.)-- (-4.,0.);
			\draw [line width=0.2,dash pattern=on 1pt off 1pt] (0.,0.)-- (4.,0.);
			\draw [line width=0.2,dash pattern=on 1pt off 1pt] (0.,-0.02)-- (-2.,3.4441016151377544);
			\draw [line width=0.2 pt,dash pattern=on 1pt off 1pt] (-2., -3.4641)-- (4, 0);
			\draw [line width=0.2 pt,dash pattern=on 1pt off 1pt] (2., -3.4641)-- (2.,3.4641016151377544);
			\draw [line width=0.2 pt,dash pattern=on 1pt off 1pt] (-4, 0)-- (2.,3.4641016151377544);
			\draw [line width=0.2 pt,dash pattern=on 1pt off 1pt] (-4, 0)-- (2.,-3.4641016151377544);
			\draw [line width=0.2 pt,dash pattern=on 1pt off 1pt] (4, 0)-- (-2.,3.4641016151377544);
			\draw [line width=0.2 pt,dash pattern=on 1pt off 1pt] (-2.,3.4641016151377544)-- (-2.,-3.4641016151377544);
			\begin{scriptsize}
				\draw [fill=ffqqqq] (-2.,-3.4641016151377544) circle (1.0pt);
				\draw [fill=ffqqqq] (2.,-3.4641016151377544) circle (1.0pt);
				\draw [fill=ffqqqq] (4.,0.) circle (1.0pt);
				\draw [fill=ffqqqq] (2.,3.4641016151377544) circle (1.0pt);
				\draw [fill=ffqqqq] (-2.,3.4641016151377544) circle (1.0pt);
				\draw [fill=ffqqqq] (-4.,0.) circle (1.0pt);
				\draw [fill=ffqqqq] (0.,0.) circle (1.0pt);
			\end{scriptsize}
		\end{tikzpicture}
		\vspace{-0.5 in}
		\caption{The regular hexagon with the circumcircle and the incircle.} \label{Fig1}
	\end{figure}
	\section{Conditional quantization for the uniform distribution $P$ on a regular $k$-gon for some $k\geq 3$} \label{sec2} 
	
	In this section, we calculate the conditional optimal sets of $n$-points and the conditional quantization errors for all $n\geq k$ for a uniform distribution on a regular polygon with $k$ sides inscribed in a unit circle with center $(0, 0)$. 
	
	The following proposition which is similar to a well-known proposition that occurs in unconstrained quantization (see \cite{GL}) is also true here.  
	\begin{prop} \label{prop0}
		Let $\ga$ be a conditional optimal set of $n$-points with respect to a conditional set $\gb$ for a probability distribution $P$, and $a\in \ga\setminus \gb$. Then,
		
		$(i)$ $P(M(a|\ga))>0$, $(ii)$ $ P(\partial M(a|\ga))=0$, $(iii)$ $a=E(X : X \in M(a|\ga))$,
		where $M(a|\ga)$ is the Voronoi region of $a\in \ga , $ i.e.,  $M(a|\ga)$ is the set of all elements $x$ in $\D R^d$ which are closest to $a$ among all the elements in $\ga$.
	\end{prop}
	
	Recall that the conditional optimal set of $k$-points is the set $\gb$ which consists of all the vertices of the regular $k$-gon. We now investigate the conditional optimal sets of $n$-points and the conditional quantization errors for all $n>k$.
	\begin{prop}\label{prop1001}
		If $k=3$, i.e., if the regular polygon is an equilateral triangle, then the conditional optimal set of $n$-points for $n=4$ for the uniform distribution $P$ is given by 
		\[\ga_4=\gb\uu\set{(0, 0)} \te{ with quantization error } V_4=\frac 12.\]
	\end{prop} 
	
	\begin{proof}
		In this case, the equilateral triangle has vertices given by $A_1(-\frac{\sqrt{3}}{2},-\frac{1}{2})$, $A_2(\frac{\sqrt{3}}{2},-\frac{1}{2})$, and $A_3(0, 1)$. Then, the sides $A_1A_2$, $A_2A_3$, and $A_3A_1$ are, respectively, denoted by $L_1$, $L_2$, and $L_3$, the parametric representations of which are given by 
		\[L_1=\{(\frac{1}{2} \sqrt{3} (2 t-1),-\frac{1}{2})\}, L_2=\{(-\frac{1}{2} \sqrt{3} (t-1),\frac{1}{2} (3 t-1))\}, \te{ and } L_3=\{(\frac{1}{2} \sqrt{3} (1-2 t),-\frac{1}{2})\},\]
		for $0\leq t\leq 1$.
		Let $\gg:=\gb\uu \set{(a, b)}$ be a conditional optimal set of four-points. By Proposition~\ref{prop0}, the following two cases can happen.
		
		\tit{Case~1.  
			$(a, b)$ is the midpoint of one of the sides $L_1$, $L_2$, and $L_3$.}
		
		Without any loss of generality, we can take $(a, b)$ as the midpoint of the side $L_1$, i.e., $(a, b)=(0, -\frac 12)$. The midpoint of $(a, b)$ and $A_1$ is given by $(-\frac{\sqrt 3}{4}, -\frac 12)$ which is represented by the parameter $t=\frac 14$ on $L_1$. Hence, the distortion error in this case, due to symmetry, can be calculated as follows: 
		\begin{align*}
			V(P; \gg)&=4 \int_0^{\frac{1}{2}} \rho((\frac{1}{2} \sqrt{3} (2 t-1), -\frac{1}{2}), (-\frac{\sqrt 3}{2}, -\frac 12))\,dt+4 \int_0^{\frac{1}{4}} \rho((\frac{1}{2} \sqrt{3} (2 t-1),-\frac{1}{2}), (-\frac{\sqrt 3}{2}, -\frac 12))\,dt\\
			&=\frac 9 {16}.
		\end{align*}
		
		\tit{Case~2. $(a, b)$ is the center $(0, 0)$ of the unit circle.}
		
		In this case, by Figure~\ref{Fig00}, we see that the Voronoi region $M((0,0)|\gg)$ contains a segment $p_1q_1$ from the side $L_1$. Similarly, it also contains segments from $L_2$, and $L_3$, and so the Voronoi region $M((0,0)|\gg)$ has positive probability. Also, $(0,0)$ is the conditional expectation in its own Voronoi region. Thus, the elements in $\gg$ in this case satisfy the conditions given in  Proposition~\ref{prop0}. Notice that by a simple calculation it can be seen that the points $p_1$ and $q_1$ on $L_1$ can be obtained from the parametric representations of $L_1$ by putting $t=\frac 13$ and $t=\frac 23$, respectively. Hence, the distortion error in this case, due to symmetry, can be calculated as follows:
		\begin{align*}
			V(P; \gg)&=6 \int_0^{\frac{1}{3}} \rho((\frac{1}{2} \sqrt{3} (2 t-1), -\frac{1}{2}), (-\frac{\sqrt 3}{2}, -\frac 12))\,dt+6 \int_{\frac 13}^{\frac{1}{2}} \rho((\frac{1}{2} \sqrt{3} (2 t-1),-\frac{1}{2}), (0, 0))\,dt\\
			&=\frac 12.
		\end{align*}
		Comparing the two distortion errors, we see that the distortion error in Case~2 is smaller than the distortion error obtained in Case~1. Hence, $\ga_4=\gb\uu\set{(0, 0)}$ forms a conditional optimal set of four-points with quantization error $V_4=\frac 12$. Thus, the proof of the proposition is complete. 
	\end{proof} 
	
	\begin{figure} 
		\begin{tikzpicture}[line cap=round,line join=round,>=triangle 45,x=0.35 cm,y=0.35 cm]
			\clip(-5.472797830477084,-5.026156858833838) rectangle (5.25308689949056,5.274903060882107);
			\draw [line width=0.8 pt] (0.,0.) circle (1.41 cm);
			\draw [line width=0.8 pt] (0.,4.)-- (-3.4641016151377544,-2.);
			\draw [line width=0.8 pt] (3.4641016151377544,-2.)-- (0.,4.);
			\draw [line width=0.8 pt] (-3.4641016151377544,-2.)-- (3.4641016151377544,-2.);
			\draw [line width=0.4 pt,dash pattern=on 1pt off 1pt] (0.,0.) circle (0.814064 cm);
			\draw [line width=0.4 pt,dash pattern=on 1pt off 1pt] (0.,0.)-- (-1.1547005383792517,-2.);
			\draw [line width=0.4 pt,dash pattern=on 1pt off 1pt] (0.,0.)-- (1.1547005383792517,-2.);
			\draw (-4.99,-1.762) node[anchor=north west] {$A_1$};
			\draw (2.89,-1.762) node[anchor=north west] {$A_2$};
			\draw (-0.89,5.66) node[anchor=north west] {$A_3$};
			\draw (0.0,0.812) node[anchor=north west] {$O$};
			\draw (-1.71,-1.872) node[anchor=north west] {$p_1$};
			\draw (0.9,-1.74) node[anchor=north west] {$q_1$};
			%\draw (2.384,0.438) node[anchor=north west] {$p_2$};
			%\draw (1.24,2.572) node[anchor=north west] {$q_2$};
			%\draw (-1.84,2.638) node[anchor=north west] {$p_3$};
			%\draw (-2.896,0.35) node[anchor=north west] {$q_3$};
			\begin{scriptsize}
				\draw [fill=ffqqqq] (0.,0.) circle (1.5 pt);
				\draw [fill=ffqqqq] (-3.4641016151377544,-2.) circle (1.5pt );
				\draw [fill=ffqqqq] (3.4641016151377544,-2.) circle (1.5pt );
				\draw [fill=ffqqqq] (0.,4.) circle (1.5pt);
				\draw [fill=ffqqqq] (-1.1547005383792517,-2.) circle (1.5pt);
				\draw [fill=ffqqqq] (1.1547005383792517,-2.) circle (1.5pt);
				\draw [fill=ffqqqq] (1.1547005383792512,2.) circle (1.5pt);
				\draw [fill=ffqqqq] (2.3094010767585034,0.) circle (1.5pt);
				\draw [fill=ffqqqq] (-1.1547005383792512,2.) circle (1.5pt);
				\draw [fill=ffqqqq] (-2.3094010767585034,0.) circle (1.5pt);
			\end{scriptsize}
		\end{tikzpicture}
		\caption{$A_1, A_2, A_3$ and $O$ form the conditional optimal set of four-points in an equilateral triangle.} \label{Fig00}
	\end{figure}
	
	Let us now give the following two lemmas. Due to technicality not giving a detailed proof, we just give an outline of the proofs. 
	\begin{lemma}\label{imp1}
		If $k=3$, i.e., in the case of an equilateral triangle inscribed in a unit circle, for any $n\geq 5$, the elements in the conditional optimal set $\ga_n$ of $n$-points must lie on the sides of the equilateral triangle.
	\end{lemma}
	
	\begin{proof}
		Notice that the equilateral triangle is geometrically symmetric with respect to any of its medians. Hence, if $\ga_n$ is a conditional optimal set of $n$-points, and if all elements in $\ga_n$ do not lie on the sides of the equilateral triangle, then at least one element $a\in\ga_n$ must line on a median with $P(M(a|\ga_n))>0$. In that case it can be seen that the distortion error is larger than the distortion error when all elements in $\ga_n$ are on the sides of the equilateral triangle. Thus, we deduce the lemma. 
	\end{proof}
	
	\begin{lemma}\label{imp1}
		Let $k\geq 4$. Then, for any $n\geq k$, the elements in the conditional optimal set $\ga_n$ of $n$-points must lie on the sides of the regular $k$-gon.
	\end{lemma}
	
	\begin{proof}
		In this case, it can be seen that for any optimal set $\ga_n$ of $n$-points, if all the elements in $\ga_n$ are not on the sides of the polygon, then at least one of them will lie on any of its diagonals. Then, one can see that for the elements of $\ga_n$ which are on the diagonals either the Voronoi regions have zero probability, or the distortion error is larger than the distortion error if all elements in $\ga_n$ are on the sides of the polygon. Thus, we deduce the lemma. 
	\end{proof} 
	
	%%\begin{proof} Let $k\geq 4$. The conditional set $\gb$ consists of the vertices, and so the elements of $\gb$ lie on the sides of the regular polygon. By\eqref{eqMM0}, we know that 
	%%\[L=\UU_{j=1}^k L_j.\]
	%%For any element $a\in L$, we see that 
	%%\[\sqrt{\rho(a, \gb)}\leq 2 \sin \frac{\pi }{k}\leq 2\cos\frac{\pi}k\leq \sqrt{\rho(a, (0, 0))},\]
	%%where $\sqrt{\rho(a, \gb)}$ represents the distance of the element $a$ from the set $\gb$, and $\sqrt{\rho(a, (0, 0))}$ represents the distance of the element $a$ from the center of the unit circle. Hence, by Proposition~\ref{prop0}, we can deduce that the elements in the conditional optimal set $\ga_n$ of $n$-points must lie on the sides of the regular polygon. Thus, the proof of the lemma is complete. 
	%%\end{proof} 
	Let $\ga_n$ contain $n_i$ elements from a side $L_j$. Then, the set of all these $n_i$ elements from $L_j$, except the rightmost element, is denoted by $\ga_n\ii(L_j(n_i-1))$, and the corresponding distortion error due to these $n_i$ elements on $L_j$ is denoted by $V(P; \ga_n\ii(L_j(n_i)))$, where $1\leq i, j\leq k$. Notice that $n_i\geq 2$ for all $1\leq i\leq k$.  Then, notice that  
	by Proposition~\ref{Me0}, we have
	\begin{align}\label{Me42000}  
		\ga_n\ii L_1(n_i-1)=\Big\{\Big(\frac{(2 j-n_i-1) \sin \frac{\pi }{k}}{n_i-1},-\cos\frac{\pi }{k}\Big): 1\leq j\leq n_i-1\Big\}
	\end{align} 
	and due to rotational symmetry 
	\begin{align} \label{Me42100}
		V(P; \ga_n\ii(L_j(n_i)))=V(P; \ga_n\ii(L_1(n_i)))=\frac{2\sin ^3 \frac{\pi }{k}}{3 k (n_i-1)^2}.
	\end{align}

	Let us now give the following proposition which plays an important role in this section.  
	
	\begin{prop} \label{Me1} 
		Let $\ga_n$ be a conditional optimal set of $n$-points for $P$ for some $n\geq k$. Let $\te{card}(\ga_n\ii L_j)=n_j$ for $1\leq j\leq k$ and $T$ be the affine transformation. Then, 
		\[\ga_n\ii L_1(n_1-1)=\Big\{\Big(\frac{(2 j-n_1-1) \sin \frac{\pi }{k}}{n_1-1},-\cos\frac{\pi }{k}\Big): 1\leq j\leq n_1-1\Big\},\]
		and 
		\[\ga_n\ii L_j(n_j-1)=T^j\Big(\ga_n\ii L_1(n_j-1)\Big) \te{ for } 2\leq j\leq k,\]
		with the conditional quantization error  
		\[V_n:=V_{n_1, n_2, \cdots, n_k}(P)=\sum_{j=1}^k \frac{2 \sin ^3 \frac{\pi }{k}}{3 k (n_j-1)^2}.\]
	\end{prop}
	\begin{proof}
		The proof of the expression for $\ga_n\ii L_1(n_1-1)$  follows from Proposition~\ref{Me0}. As stated in the relation \eqref{Me42000}, the proof of the expressions for $\ga_n\ii L_j(n_j-1)$ are the consequence of Proposition~\ref{Me0} and the affine transformation $T$. Due to rotational symmetry, as stated in relation \eqref{Me42100}, using Proposition~\ref{Me0}, we have 
		\[V(P; \ga_n\ii L_j(n_j))=V(P; \ga_n\ii L_1(n_j))=\frac{2 \sin ^3 \frac{\pi }{k}}{3 k (n_j-1)^2} \te{ for all } 1\leq j\leq k.\]
		Hence, 
		\[V_n:=V_{n_1, n_2, \cdots, n_k}(P)=\sum_{j=1}^k V(P; \ga_n\ii L_j(n_j))=\sum_{j=1}^k V(P; \ga_n\ii(L_1(n_j)))=\sum_{j=1}^k \frac{2 \sin ^3 \frac{\pi }{k}}{3 k (n_j-1)^2}.\]
		Thus, the proof of the proposition is complete. 
	\end{proof}

	\begin{lemma} \label{lemmaMe1} 
		Let $n_j\in \D N$ for $1\leq j\leq k$ be the numbers as defined in Proposition~\ref{Me1}. Then, $|n_i-n_j|\in \set{0,1}$ for $1\leq i\neq j\leq k$. 
	\end{lemma}
	
	\begin{proof} We first show that $|n_1-n_2|\in \set{0,1}$. 
		Write $m:=(n_1-1)+(n_2-1)$. Now, the distortion error contributed by the $m$ elements in $\ga_n(L_1)\uu \ga_n(L_2)$ is given by 
		\[\frac{2 \sin ^3 \frac{\pi }{k}}{3k} \Big(\frac 1{(n_1-1)^2}+\frac 1{(n_2-1)^2}\Big).\]
		The above expression is minimum if $n_1-1\approx \frac m 2$ and $n_2-1\approx \frac m 2$. Thus, we see that if $m=2\ell$ for some positive integer $\ell$, then $n_1-1=n_2-1=\ell$, and if $m=2\ell+1$ for some positive integer $\ell$, then either $(n_1-1=\ell+1$ and $n_2-1=\ell)$ or $(n_1-1=\ell$ and $n_2-1=\ell+1)$ which yields the fact that $|n_1-n_2|\in \set{0,1}$. Similarly, we can show that  $|n_i-n_j|\in \set{0,1}$ for any $1\leq i\neq j\leq k$. Thus, the proof of the lemma is complete. 
	\end{proof}

	\begin{theorem}\label{theo0}
		Let $\ga_n$ be a conditional optimal set of $n$-points for $P$ for some $n\geq k$. Then, 
		\begin{align*}
			\ga_n= \Big(\ga_n\ii L_1(n_1-1)\Big)\UU \Big(\UU_{j=1}^k T^j(\ga_n\ii L_1(n_j-1))\Big)
		\end{align*} 
		with the conditional quantization error 
		\[V_n=\sum_{j=1}^k \frac{2 \sin ^3 \frac{\pi }{k}}{3 k (n_j-1)^2},\]
		where $n_j$ are chosen in such a way that $n_j\geq 2$ and $\sum\limits_{j=1}^k (n_j-1)=n$, and $|n_i-n_j|\in \set{0,1}$ for $1\leq i\neq j\leq k$. 
	\end{theorem}
	\begin{proof}
		The proof of the theorem follows due to Proposition~\ref{Me1} and Lemma~\ref{lemmaMe1}.  
	\end{proof}
	\begin{remark}
		Let $\ga_n$ be a conditional optimal set of $n$-points for $P$ given by Theorem~\ref{theo0}. If $n=k\ell$ for some $\ell\in \D N$, then $\ga_n$ is unique. If $n=k\ell+j$ for some $j, \ell\in \D N$ with $1\leq j<k$, then the number of such sets is given by $\binom{k}{j}$, where $\binom{k}{j}=\frac{k!}{j!(k-j)!}$ for $1\leq j<k$. 
	\end{remark} 
	Using Theorem~\ref{theo0}, the following example can be obtained. 
	\begin{exam} Let $\ga_n$ be a conditional optimal set of $n$-points given by Theorem~\ref{theo0} for a uniform distribution on a regular hexagon, i.e., when $k=6$. Then, we obtain the optimal sets as follows (see Figure~\ref{Fig2}): 
		\begin{align*}
			\ga_6&=\{(-\frac 12, -\frac{\sqrt 3}{2}), (\frac 12, -\frac{\sqrt 3}{2}), (1, 0), (\frac 12, \frac{\sqrt 3}{2}), (-\frac 12, \frac{\sqrt 3}{2}), (-1, 0)\}  \\
			&\qquad \te{ with quantization error } V_6=\frac  1{12},\\
			\ga_7&=\{(-\frac 12, -\frac{\sqrt 3}{2}), (0, -\frac{\sqrt 3}{2}), (\frac 12, -\frac{\sqrt 3}{2}), (1, 0), (\frac 12, \frac{\sqrt 3}{2}), (-\frac 12, \frac{\sqrt 3}{2}), (-1, 0)\}  \\
			&\qquad \te{ with quantization error } V_7=\frac  7{96},\\
			\ga_8&=\{(-\frac 12, -\frac{\sqrt 3}{2}), (0, -\frac{\sqrt 3}{2}), (\frac 12, -\frac{\sqrt 3}{2}), (\frac 34, -\frac{\sqrt 3} 4), (1, 0), (\frac 12, \frac{\sqrt 3}{2}), (-\frac 12, \frac{\sqrt 3}{2}), (-1, 0)\} \\
			&\qquad \te{ with quantization error } V_8=\frac  1{16},\\
			\ga_9&=\{(-\frac 12, -\frac{\sqrt 3}{2}), (0, -\frac{\sqrt 3}{2}), (\frac 12, -\frac{\sqrt 3}{2}), (\frac 34, -\frac{\sqrt 3} 4), (1, 0), (\frac 34,\frac{\sqrt 3} 4),  (\frac 12, \frac{\sqrt 3}{2}), (-\frac 12, \frac{\sqrt 3}{2}), (-1, 0)\} \\
			&\qquad \te{ with quantization error } V_9=\frac  5{96},\\
			&\te{ and so on.}
		\end{align*}  
	\end{exam}

	\begin{figure}
		\begin{tikzpicture}[line cap=round,line join=round,>=triangle 45,x=0.35 cm,y=0.35 cm]
			\clip(-4.472797830477084,-4.626156858833838) rectangle (4.25308689949056,4.274903060882107);
			\draw [line width=0.8 pt] (-2.,3.4641016151377544)-- (-4.,0.)-- (-2.,-3.4641016151377544)-- (2.,-3.4641016151377544)-- (4.,0.)-- (2.,3.4641016151377544);
			\draw [line width=0.8 pt] (-2.,3.4641016151377544)-- (2.,3.4641016151377544);
			\draw [line width=0.8 pt] (0.,0.) circle (1.41 cm);
			\draw [line width=0.4 pt,dash pattern=on 1pt off 1pt] (0.,0.)-- (-2.,-3.4641016151377544);
			\draw [line width=0.4 pt,dash pattern=on 1pt off 1pt] (0.,0.)-- (2.,-3.4641016151377544);
			\draw [line width=0.4 pt,dash pattern=on 1pt off 1pt] (0.,0.)-- (2.,3.4641016151377544);
			\draw [line width=0.4 pt,dash pattern=on 1pt off 1pt] (0.,0.)-- (-4.,0.);
			\draw [line width=0.4 pt,dash pattern=on 1pt off 1pt] (0.,0.)-- (4.,0.);
			\draw [line width=0.4 pt,dash pattern=on 1pt off 1pt] (0.,-0.02)-- (-2.,3.4441016151377544);
			\begin{scriptsize}
				\draw [fill=ffqqqq] (-2.,-3.4641016151377544) circle (1.5pt);
				\draw [fill=ffqqqq] (2.,-3.4641016151377544) circle (1.5pt);
				\draw [fill=ffqqqq] (4.,0.) circle (1.5pt);
				\draw [fill=ffqqqq] (2.,3.4641016151377544) circle (1.5pt);
				\draw [fill=ffqqqq] (-2.,3.4641016151377544) circle (1.5pt);
				\draw [fill=ffqqqq] (-4.,0.) circle (1.5pt);
			\end{scriptsize}
		\end{tikzpicture}
		\begin{tikzpicture}[line cap=round,line join=round,>=triangle 45,x=0.35 cm,y=0.35 cm]
			\clip(-4.472797830477084,-4.626156858833838) rectangle (4.25308689949056,4.274903060882107);
			\draw [line width=0.8 pt] (-2.,3.4641016151377544)-- (-4.,0.)-- (-2.,-3.4641016151377544)-- (2.,-3.4641016151377544)-- (4.,0.)-- (2.,3.4641016151377544);
			\draw [line width=0.8 pt] (-2.,3.4641016151377544)-- (2.,3.4641016151377544);
			\draw [line width=0.8 pt] (0.,0.) circle (1.41 cm);
			\draw [line width=0.4 pt,dash pattern=on 1pt off 1pt] (0.,0.)-- (-2.,-3.4641016151377544);
			\draw [line width=0.4 pt,dash pattern=on 1pt off 1pt] (0.,0.)-- (2.,-3.4641016151377544);
			\draw [line width=0.4 pt,dash pattern=on 1pt off 1pt] (0.,0.)-- (2.,3.4641016151377544);
			\draw [line width=0.4 pt,dash pattern=on 1pt off 1pt] (0.,0.)-- (-4.,0.);
			\draw [line width=0.4 pt,dash pattern=on 1pt off 1pt] (0.,0.)-- (4.,0.);
			\draw [line width=0.4 pt,dash pattern=on 1pt off 1pt] (0.,-0.02)-- (-2.,3.4441016151377544);
			\begin{scriptsize}
				\draw [fill=ffqqqq] (-2.,-3.4641016151377544) circle (1.5pt);
				\draw [fill=ffqqqq] (0.,-3.4641016151377544) circle (1.5pt);
				\draw [fill=ffqqqq] (2.,-3.4641016151377544) circle (1.5pt);
				\draw [fill=ffqqqq] (4.,0.) circle (1.5pt);
				\draw [fill=ffqqqq] (2.,3.4641016151377544) circle (1.5pt);
				\draw [fill=ffqqqq] (-2.,3.4641016151377544) circle (1.5pt);
				\draw [fill=ffqqqq] (-4.,0.) circle (1.5pt);
			\end{scriptsize}
		\end{tikzpicture} 
		\begin{tikzpicture}[line cap=round,line join=round,>=triangle 45,x=0.35 cm,y=0.35 cm]
			\clip(-4.472797830477084,-4.626156858833838) rectangle (4.25308689949056,4.274903060882107);
			\draw [line width=0.8 pt] (-2.,3.4641016151377544)-- (-4.,0.)-- (-2.,-3.4641016151377544)-- (2.,-3.4641016151377544)-- (4.,0.)-- (2.,3.4641016151377544);
			\draw [line width=0.8 pt] (-2.,3.4641016151377544)-- (2.,3.4641016151377544);
			\draw [line width=0.8 pt] (0.,0.) circle (1.41 cm);
			\draw [line width=0.4 pt,dash pattern=on 1pt off 1pt] (0.,0.)-- (-2.,-3.4641016151377544);
			\draw [line width=0.4 pt,dash pattern=on 1pt off 1pt] (0.,0.)-- (2.,-3.4641016151377544);
			\draw [line width=0.4 pt,dash pattern=on 1pt off 1pt] (0.,0.)-- (2.,3.4641016151377544);
			\draw [line width=0.4 pt,dash pattern=on 1pt off 1pt] (0.,0.)-- (-4.,0.);
			\draw [line width=0.4 pt,dash pattern=on 1pt off 1pt] (0.,0.)-- (4.,0.);
			\draw [line width=0.4 pt,dash pattern=on 1pt off 1pt] (0.,-0.02)-- (-2.,3.4441016151377544);
			\begin{scriptsize}
				\draw [fill=ffqqqq] (-2.,-3.4641016151377544) circle (1.5pt);
				\draw [fill=ffqqqq] (2.,-3.4641016151377544) circle (1.5pt);
				\draw [fill=ffqqqq] (0.,-3.4641016151377544) circle (1.5pt);
				\draw [fill=ffqqqq] (3., -1.73205) circle (1.5pt);
				\draw [fill=ffqqqq] (4.,0.) circle (1.5pt);
				\draw [fill=ffqqqq] (2.,3.4641016151377544) circle (1.5pt);
				\draw [fill=ffqqqq] (-2.,3.4641016151377544) circle (1.5pt);
				\draw [fill=ffqqqq] (-4.,0.) circle (1.5pt);
			\end{scriptsize}
		\end{tikzpicture}
		\begin{tikzpicture}[line cap=round,line join=round,>=triangle 45,x=0.35 cm,y=0.35 cm]
			\clip(-4.472797830477084,-4.626156858833838) rectangle (4.25308689949056,4.274903060882107);
			\draw [line width=0.8 pt] (-2.,3.4641016151377544)-- (-4.,0.)-- (-2.,-3.4641016151377544)-- (2.,-3.4641016151377544)-- (4.,0.)-- (2.,3.4641016151377544);
			\draw [line width=0.8 pt] (-2.,3.4641016151377544)-- (2.,3.4641016151377544);
			\draw [line width=0.8 pt] (0.,0.) circle (1.41 cm);
			\draw [line width=0.4 pt,dash pattern=on 1pt off 1pt] (0.,0.)-- (-2.,-3.4641016151377544);
			\draw [line width=0.4 pt,dash pattern=on 1pt off 1pt] (0.,0.)-- (2.,-3.4641016151377544);
			\draw [line width=0.4 pt,dash pattern=on 1pt off 1pt] (0.,0.)-- (2.,3.4641016151377544);
			\draw [line width=0.4 pt,dash pattern=on 1pt off 1pt] (0.,0.)-- (-4.,0.);
			\draw [line width=0.4 pt,dash pattern=on 1pt off 1pt] (0.,0.)-- (4.,0.);
			\draw [line width=0.4 pt,dash pattern=on 1pt off 1pt] (0.,-0.02)-- (-2.,3.4441016151377544);
			\begin{scriptsize}
				\draw [fill=ffqqqq] (-2.,-3.4641016151377544) circle (1.5pt);
				\draw [fill=ffqqqq] (2.,-3.4641016151377544) circle (1.5pt);
				\draw [fill=ffqqqq] (0.,-3.4641016151377544) circle (1.5pt);
				\draw [fill=ffqqqq] (3., -1.73205) circle (1.5pt);
				\draw [fill=ffqqqq] (3., 1.73205) circle (1.5pt);
				\draw [fill=ffqqqq] (4.,0.) circle (1.5pt);
				\draw [fill=ffqqqq] (2.,3.4641016151377544) circle (1.5pt);
				\draw [fill=ffqqqq] (-2.,3.4641016151377544) circle (1.5pt);
				\draw [fill=ffqqqq] (-4.,0.) circle (1.5pt);
			\end{scriptsize}
		\end{tikzpicture}\\
		\begin{tikzpicture}[line cap=round,line join=round,>=triangle 45,x=0.35 cm,y=0.35 cm]
			\clip(-4.472797830477084,-4.626156858833838) rectangle (4.25308689949056,4.274903060882107);
			\draw [line width=0.8 pt] (-2.,3.4641016151377544)-- (-4.,0.)-- (-2.,-3.4641016151377544)-- (2.,-3.4641016151377544)-- (4.,0.)-- (2.,3.4641016151377544);
			\draw [line width=0.8 pt] (-2.,3.4641016151377544)-- (2.,3.4641016151377544);
			\draw [line width=0.8 pt] (0.,0.) circle (1.41 cm);
			\draw [line width=0.4 pt,dash pattern=on 1pt off 1pt] (0.,0.)-- (-2.,-3.4641016151377544);
			\draw [line width=0.4 pt,dash pattern=on 1pt off 1pt] (0.,0.)-- (2.,-3.4641016151377544);
			\draw [line width=0.4 pt,dash pattern=on 1pt off 1pt] (0.,0.)-- (2.,3.4641016151377544);
			\draw [line width=0.4 pt,dash pattern=on 1pt off 1pt] (0.,0.)-- (-4.,0.);
			\draw [line width=0.4 pt,dash pattern=on 1pt off 1pt] (0.,0.)-- (4.,0.);
			\draw [line width=0.4 pt,dash pattern=on 1pt off 1pt] (0.,-0.02)-- (-2.,3.4441016151377544);
			\begin{scriptsize}
				\draw [fill=ffqqqq] (-2.,-3.4641016151377544) circle (1.5pt);
				\draw [fill=ffqqqq] (2.,-3.4641016151377544) circle (1.5pt);
				\draw [fill=ffqqqq] (0.,-3.4641016151377544) circle (1.5pt);
				\draw [fill=ffqqqq] (3., -1.73205) circle (1.5pt);
				\draw [fill=ffqqqq] (3., 1.73205) circle (1.5pt);
				\draw [fill=ffqqqq] (0.,3.4641016151377544) circle (1.5pt);
				\draw [fill=ffqqqq] (4.,0.) circle (1.5pt);
				\draw [fill=ffqqqq] (2.,3.4641016151377544) circle (1.5pt);
				\draw [fill=ffqqqq] (-2.,3.4641016151377544) circle (1.5pt);
				\draw [fill=ffqqqq] (-4.,0.) circle (1.5pt);
			\end{scriptsize}
		\end{tikzpicture}
		\begin{tikzpicture}[line cap=round,line join=round,>=triangle 45,x=0.35 cm,y=0.35 cm]
			\clip(-4.472797830477084,-4.626156858833838) rectangle (4.25308689949056,4.274903060882107);
			\draw [line width=0.8 pt] (-2.,3.4641016151377544)-- (-4.,0.)-- (-2.,-3.4641016151377544)-- (2.,-3.4641016151377544)-- (4.,0.)-- (2.,3.4641016151377544);
			\draw [line width=0.8 pt] (-2.,3.4641016151377544)-- (2.,3.4641016151377544);
			\draw [line width=0.8 pt] (0.,0.) circle (1.41 cm);
			\draw [line width=0.4 pt,dash pattern=on 1pt off 1pt] (0.,0.)-- (-2.,-3.4641016151377544);
			\draw [line width=0.4 pt,dash pattern=on 1pt off 1pt] (0.,0.)-- (2.,-3.4641016151377544);
			\draw [line width=0.4 pt,dash pattern=on 1pt off 1pt] (0.,0.)-- (2.,3.4641016151377544);
			\draw [line width=0.4 pt,dash pattern=on 1pt off 1pt] (0.,0.)-- (-4.,0.);
			\draw [line width=0.4 pt,dash pattern=on 1pt off 1pt] (0.,0.)-- (4.,0.);
			\draw [line width=0.4 pt,dash pattern=on 1pt off 1pt] (0.,-0.02)-- (-2.,3.4441016151377544);
			\begin{scriptsize}
				\draw [fill=ffqqqq] (-2.,-3.4641016151377544) circle (1.5pt);
				\draw [fill=ffqqqq] (2.,-3.4641016151377544) circle (1.5pt);
				\draw [fill=ffqqqq] (0.,-3.4641016151377544) circle (1.5pt);
				\draw [fill=ffqqqq] (3., -1.73205) circle (1.5pt);
				\draw [fill=ffqqqq] (3., 1.73205) circle (1.5pt);
				\draw [fill=ffqqqq] (-3., 1.73205) circle (1.5pt);
				\draw [fill=ffqqqq] (0.,3.4641016151377544) circle (1.5pt);
				\draw [fill=ffqqqq] (4.,0.) circle (1.5pt);
				\draw [fill=ffqqqq] (2.,3.4641016151377544) circle (1.5pt);
				\draw [fill=ffqqqq] (-2.,3.4641016151377544) circle (1.5pt);
				\draw [fill=ffqqqq] (-4.,0.) circle (1.5pt);
			\end{scriptsize}
		\end{tikzpicture}
		\begin{tikzpicture}[line cap=round,line join=round,>=triangle 45,x=0.35 cm,y=0.35 cm]
			\clip(-4.472797830477084,-4.626156858833838) rectangle (4.25308689949056,4.274903060882107);
			\draw [line width=0.8 pt] (-2.,3.4641016151377544)-- (-4.,0.)-- (-2.,-3.4641016151377544)-- (2.,-3.4641016151377544)-- (4.,0.)-- (2.,3.4641016151377544);
			\draw [line width=0.8 pt] (-2.,3.4641016151377544)-- (2.,3.4641016151377544);
			\draw [line width=0.8 pt] (0.,0.) circle (1.41 cm);
			\draw [line width=0.4 pt,dash pattern=on 1pt off 1pt] (0.,0.)-- (-2.,-3.4641016151377544);
			\draw [line width=0.4 pt,dash pattern=on 1pt off 1pt] (0.,0.)-- (2.,-3.4641016151377544);
			\draw [line width=0.4 pt,dash pattern=on 1pt off 1pt] (0.,0.)-- (2.,3.4641016151377544);
			\draw [line width=0.4 pt,dash pattern=on 1pt off 1pt] (0.,0.)-- (-4.,0.);
			\draw [line width=0.4 pt,dash pattern=on 1pt off 1pt] (0.,0.)-- (4.,0.);
			\draw [line width=0.4 pt,dash pattern=on 1pt off 1pt] (0.,-0.02)-- (-2.,3.4441016151377544);
			\begin{scriptsize}
				\draw [fill=ffqqqq] (-2.,-3.4641016151377544) circle (1.5pt);
				\draw [fill=ffqqqq] (2.,-3.4641016151377544) circle (1.5pt);
				\draw [fill=ffqqqq] (0.,-3.4641016151377544) circle (1.5pt);
				\draw [fill=ffqqqq] (3., -1.73205) circle (1.5pt);
				\draw [fill=ffqqqq] (3., 1.73205) circle (1.5pt);
				\draw [fill=ffqqqq] (-3., 1.73205) circle (1.5pt);
				\draw [fill=ffqqqq] (-3., -1.73205) circle (1.5pt);
				\draw [fill=ffqqqq] (0.,3.4641016151377544) circle (1.5pt);
				\draw [fill=ffqqqq] (4.,0.) circle (1.5pt);
				\draw [fill=ffqqqq] (2.,3.4641016151377544) circle (1.5pt);
				\draw [fill=ffqqqq] (-2.,3.4641016151377544) circle (1.5pt);
				\draw [fill=ffqqqq] (-4.,0.) circle (1.5pt);
			\end{scriptsize}
		\end{tikzpicture}
		\begin{tikzpicture}[line cap=round,line join=round,>=triangle 45,x=0.35 cm,y=0.35 cm]
			\clip(-4.472797830477084,-4.626156858833838) rectangle (4.25308689949056,4.274903060882107);
			\draw [line width=0.8 pt] (-2.,3.4641016151377544)-- (-4.,0.)-- (-2.,-3.4641016151377544)-- (2.,-3.4641016151377544)-- (4.,0.)-- (2.,3.4641016151377544);
			\draw [line width=0.8 pt] (-2.,3.4641016151377544)-- (2.,3.4641016151377544);
			\draw [line width=0.8 pt] (0.,0.) circle (1.41 cm);
			\draw [line width=0.4 pt,dash pattern=on 1pt off 1pt] (0.,0.)-- (-2.,-3.4641016151377544);
			\draw [line width=0.4 pt,dash pattern=on 1pt off 1pt] (0.,0.)-- (2.,-3.4641016151377544);
			\draw [line width=0.4 pt,dash pattern=on 1pt off 1pt] (0.,0.)-- (2.,3.4641016151377544);
			\draw [line width=0.4 pt,dash pattern=on 1pt off 1pt] (0.,0.)-- (-4.,0.);
			\draw [line width=0.4 pt,dash pattern=on 1pt off 1pt] (0.,0.)-- (4.,0.);
			\draw [line width=0.4 pt,dash pattern=on 1pt off 1pt] (0.,-0.02)-- (-2.,3.4441016151377544);
			\begin{scriptsize}
				\draw [fill=ffqqqq] (-2.,-3.4641016151377544) circle (1.5pt);
				\draw [fill=ffqqqq] (2.,-3.4641016151377544) circle (1.5pt);
				\draw [fill=ffqqqq] (-0.666667,-3.4641016151377544) circle (1.5pt);
				\draw [fill=ffqqqq] (0.666667,-3.4641016151377544) circle (1.5pt);
				\draw [fill=ffqqqq] (3., -1.73205) circle (1.5pt);
				\draw [fill=ffqqqq] (3., 1.73205) circle (1.5pt);
				\draw [fill=ffqqqq] (-3., 1.73205) circle (1.5pt);
				\draw [fill=ffqqqq] (-3., -1.73205) circle (1.5pt);
				\draw [fill=ffqqqq] (0.,3.4641016151377544) circle (1.5pt);
				\draw [fill=ffqqqq] (4.,0.) circle (1.5pt);
				\draw [fill=ffqqqq] (2.,3.4641016151377544) circle (1.5pt);
				\draw [fill=ffqqqq] (-2.,3.4641016151377544) circle (1.5pt);
				\draw [fill=ffqqqq] (-4.,0.) circle (1.5pt);
			\end{scriptsize}
		\end{tikzpicture}
		\caption{In conditional quantization the dots represent the elements in an optimal set of $n$-points for $6\leq n\leq 13$.} \label{Fig2}
	\end{figure}

	\begin{theorem}\label{theo3} Let $P$ be a uniform distribution on a regular polygon with $k$ sides. Then, 
		the conditional quantization dimension $D(P)$ of the probability distribution $P$ exists, and $D(P)=1$, and the $D(P)$-dimensional conditional quantization coefficient for $P$ exists as a finite positive number and equals $\frac 23 k^2\sin^3\frac {\pi}{k}$. 
	\end{theorem}
	
	\begin{proof}
		For $n\in \D N$ with $n\geq k$, let $\ell(n)$ be the unique natural number such that $k{\ell(n)}\leq n<k({\ell(n)+1})$. Then,
		$V_{k(\ell(n)+1)}\leq V_n\leq V_{k {\ell(n)}}$, i.e., $\frac {2\sin^3\frac{\pi}k} {3(\ell(n))^2}\leq V_n\leq  \frac {2\sin^3\frac{\pi}k} {3 (\ell(n)-1)^2}$ yielding 
		$V_\infty=\mathop{\lim}\limits_{n\to \infty} V_n=0$. Take $n$ large enough so that $ V_{k{\ell(n)}}-V_\infty<1$. Then,
		\[0<-\log (V_{k{\ell(n)}}-V_\infty)\leq -\log (V_n-V_\infty) \leq -\log (V_{k({\ell(n)+1})}-V_\infty)\]
		yielding
		\[\frac{2 \log k\ell(n)}{-\log (V_{k({\ell(n)+1})}-V_\infty)}\leq \frac{2\log n}{-\log (V_n-V_\infty)}\leq \frac{2 \log k(\ell(n)+1)}{-\log (V_{k{\ell(n)}}-V_\infty)}.\]
		Notice that
		\begin{align*}
			\lim_{n\to \infty}\frac{2 \log k\ell(n)}{-\log (V_{k({\ell(n)+1})}-V_\infty)}&=\lim_{n\to \infty} \frac{2 \log k\ell(n)}{-\log \frac {2\sin^3\frac{\pi}k} {3(\ell(n))^2}}=1, \te{ and }\\
			\lim_{n\to \infty} \frac{2 \log k(\ell(n)+1)}{-\log V_{k{\ell(n)}}}&=\lim_{n\to \infty} \frac{2 \log k(\ell(n)+1)}{-\log \frac {2\sin^3\frac{\pi}k} {3(\ell(n)-1)^2}}=1.
		\end{align*}
		Hence, $\lim_{n\to \infty}  \frac{2\log n}{-\log (V_n-V_\infty)}=1$, i.e., the conditional quantization dimension $D(P)$ of the probability measure $P$ exists and $D(P)=1$.
		Since
		\begin{align*}
			&\lim_{n\to \infty} n^2 (V_n-V_\infty)\geq \lim_{n\to \infty} (k\ell(n))^2 (V_{k(\ell(n)+1)}-V_\infty)=\lim_{n\to\infty}(k\ell(n))^2\frac {2\sin^3\frac {\pi}{k}}{3(\ell(n))^2}=\frac 23 k^2\sin^3\frac {\pi}{k}, \te{ and } \\
			&\lim_{n\to \infty} n^2 (V_n-V_\infty)\leq \lim_{n\to \infty} (k(\ell(n)+1))^2 (V_{k\ell(n)}-V_\infty)=\lim_{n\to\infty}(k(\ell(n)+1))^2\frac {2\sin^3\frac {\pi}{k}}{3(\ell(n)-1)^2}=\frac 23 k^2\sin^3\frac {\pi}{k},
		\end{align*}
		by the squeeze theorem, we have $\lim_{n\to \infty} n^2 (V_n-V_\infty)=\frac 23 k^2\sin^3\frac {\pi}{k}$, i.e., the $D(P)$-dimensional conditional quantization coefficient for $P$ exists as a finite positive number and equals $\frac 23 k^2\sin^3\frac {\pi}{k}$.  Thus, the proof of the theorem is complete. 
	\end{proof}

	\begin{remark}
		By Theorem~\ref{theo3}, we see that the conditional quantization dimension $D(P)$ of the uniform distribution $P$ on a regular polygon with $k$ sides is a constant, it does not depend on $k$ and equals one which is the dimension of the underlying space where the support of the probability measure is defined. On the other hand, the conditional quantization coefficient is not constant, it exists as a finite positive number $\frac 23 k^2\sin^3\frac {\pi}{k}$ that depends on the number of sides $k$ of the regular polygon. Moreover, notice that the conditional quantization coefficient $\frac 23 k^2\sin^3\frac {\pi}{k}$ is a strictly decreasing sequence and converges to zero as $k$ approaches to infinity. 
	\end{remark} 
	
	\section{Conditional constrained quantization for the uniform distribution $P$ on the regular $k$-gon when the constraint is the circumcircle} \label{sec3} 
	In this section, for all $n\geq k$, we calculate the conditional constrained optimal sets $\ga_n$ of $n$-points and the corresponding $n$th conditional constrained quantization errors $V_n$ for the uniform distribution $P$ on the regular polygon with $k$ sides with respect to the conditional set $\gb$ and the constraint $S$ which is the circumcircle of the polygon. Let $O$ be the center of the circumcircle. Let $S_j$ be the circular arcs subtended by the chords $A_jA_{j+1}$ for $1\leq j\leq k$, where $A_{j+1}$ is identified as $A_1$,  (see Figure~\ref{Fig1} for $k=6$). Then, 
	\[S=\UU_{j=1}^k S_j.\]
	The arcs $S_j$ are represented by the parametric equations as follows:
	\[S_j=\Big\{(\cos \gq, \sin \gq) : \frac {3\pi}{2}+\frac {(2j-3)\pi}{k}\leq  \gq\leq \frac {3\pi}{2}+\frac {(2j-1)\pi}{k}\Big\} \]
	for $1\leq j\leq k$. Also notice that for the affine transformation $T$, defined in Section~\ref{secPre}, we have 
	\[T(S_j)=S_{j+1} \te{ for } 1\leq j\leq k,\]
	where $S_{k+1}$ is identified as $S_1$. 
	Let $\ga_n$ be an optimal set of $n$-points for $P$ for some $n\geq k$. Notice that if $n=k$, then $\ga_n=\gb$, where $\gb$ is the conditional set. Hence, in the sequel, we assume that $n\geq k+1$. 
	Let $\ga_n$ contain $n_i$ elements from $S_i$ including the end elements, i.e., $\te{card}(\ga_n\ii S_i)=n_i$, where $n_i \geq 2$ for $1\leq i\leq k$. Notice that $n_1+n_2+\cdots+n_k=n+k$.  To obtain $\ga_n\ii S_i$ and the corresponding distortion error $V(P; \ga_n\ii S_i)$ contributed by the elements in $\ga_n\ii S_i$ on $L_i$, where $\te{card}(\ga_n\ii S_i)=n_i$, it is enough to obtain $\ga_n\ii S_1$ and the corresponding distortion error $V(P; \ga_n\ii S_1)$ with $\te{card}(\ga_n\ii S_1)=n_i$, and then to use the affine transformation $T$, as we know
	\[\ga_n\ii S_i=T^{i-1}(\ga_n\ii S_1), \te{ where } \te{card}(\ga_n\ii S_1)=n_i, \te{ for } 2\leq i\leq k.\]
	Moreover, notice that $V(P; \ga_n\ii S_i)=V(P; \ga_n\ii S_1), \te{ where } \te{card}(\ga_n\ii S_1)=n_i$ and $V(P; \ga_n\ii S_1)$ denote the distortion error contributed by the $n_i$ elements on $L_1$. 
	Again, notice that if $\te{card}(\ga_n\ii S_i)=n_i=2$, then 
	\[\ga_n\ii S_i=T^{i-1}\Big(\set{A_1, A_2}\Big) \te{ for } 1\leq i\leq k.\]
	Let $\te{card}(\ga_n\ii S_1)=\ell$ with $\ell\geq 3$. 
	Let the elements in $\ga_n\ii S_1$ be represented by the parameters $\gq_1<\gq_1<\gq_2<\dots<\gq_\ell$, including the end elements $\gq_1=\frac{3 \pi }{2}-\frac{\pi }{k}$ and $\gq_\ell=\frac{3 \pi }{2}+\frac{\pi }{k}$, where $\ell\geq 3$. Let the boundary of the Voronoi regions of the elements $\gq_j$ and $\gq_{j+1}$ intersect the side $L_1$ at the elements $(a_j, -\frac {\sqrt 3}{2})$ for $1\leq j\leq \ell-1$. 
	Then, the canonical equations are given by 
	\[\rho((\cos\gq_j, \sin\gq_j), (a_j, -\frac {\sqrt 3}{2}))-\rho((\cos\gq_{j+1}, \sin\gq_{j+1}), (a_{j}, -\frac {\sqrt 3}{2}))=0\]
	for $1\leq j\leq \ell-1$. Solving the canonical equations, we have 
	\begin{align} \label{Megha122} 
		a_j= \frac{-\sin ^2 \theta _j +\sin ^2 \theta _{j+1} -\sqrt{3} \sin  \theta _j +\sqrt{3} \sin  \theta _{j+1} -\cos ^2 \theta _j +\cos ^2 \theta _{j+1} }{2 \cos \theta _{j+1} -2 \cos  \theta _j }
	\end{align}
	for $1\leq j\leq \ell-1$. Hence, if $V(P; \ga_n\ii S_1)$ is the distortion error contributed by the elements in $\ga_n\ii S_1$ on $L_1$, then we have 
	\begin{align}\label{Megha123} 
		V(P; \ga_n\ii S_1)&=\frac 1 k\Big(\int_{-\frac{1}{2}}^{a_1} \rho((x,-\frac{\sqrt{3}}{2}), (\cos\theta _1 ,\sin \theta _1)) \, dx+\sum_{j=1}^{\ell-2}\int_{a_j}^{a_{j+1}} \rho((x,-\frac{\sqrt{3}}{2}), (\cos\theta_{j+1} ,\sin \theta_{j+1})) \, dx\notag\\
		&\qquad +\int_{a_\ell}^{\frac 12} \rho((x,-\frac{\sqrt{3}}{2}), (\cos\theta _\ell ,\sin \theta _\ell)) \, dx \Big).
	\end{align} 
	\begin{method}\label{method101}
		To obtain the conditional constrained optimal sets $\ga_n$ of $n$-points and the $n$th conditional constrained quantization errors for the uniform distribution $P$ on the regular $k$-gon with respect to the conditional set $\gb$ and the constraint $S$ for any $n\geq k+1$, assume that $\te{card}(\ga_n\ii S_1)=\ell$, i.e., 
		\[\ga_n\ii S_1=\set{(\cos \gq_j, \sin \gq_j) : 1\leq j\leq \ell}\]
		for $\ell\geq 3$. Then, solving the canonical equations we obtain the boundary of the Voronoi regions of $(a_j, -\frac{\sqrt 3}2)$, where $a_j$ are given as in \eqref{Megha122}. Putting the values of $a_j$,  $\gq_1=\frac{3 \pi }{2}-\frac{\pi }{k}$, and $\gq_\ell=\frac{3 \pi }{2}-\frac{\pi }{k}$ in \eqref{Megha123}, we obtain the distortion error $V(P; \ga_n\ii S_1)$, which is a function of $\gq_j$ for $2\leq j\leq \ell-1$. Now, using calculus, we obtain the values of $\gq_j$ for $2\leq j\leq \ell-1$ for which $V(P; \ga_n\ii S_1)$ is minimum. Upon obtained the values of $\gq_j$, we can easily obtain
		$\ga_n\ii S_1$ and the corresponding distortion error $V(P; \ga_n\ii S_1)$. As described before due to the affine transformation, we can also obtain $\ga_n\ii S_i$ and the corresponding distortion error $V(P; \ga_n\ii S_i)$ for $2\leq i\leq k$. Thus, we have 
		\[\ga_n=\UU_{i=1}^k(\ga_n\ii S_i) \te{ with the $n$th constrained quantization error } V_n=\sum_{i=1}^k V(P; \ga_n\ii S_i).\]
	\end{method}
	
	Recall that $S=\mathop{\uu}\limits_{j=1}^k S_j$ and $S_j$ are rotationally symmetric, and $P$ is a uniform distribution. Hence, the following lemma, which is similar to Lemma~\ref{lemmaMe1}, is also true here.    
	\begin{lemma} \label{lemmaMe112} 
		Let $\ga_n$ be an optimal set of $n$-points with $\te{card}(\ga_n\ii S_i)=n_i$ for $1\leq i\leq k$. Then, $|n_i-n_j|\in \set{0,1}$ for $1\leq i\neq j\leq k$. 
	\end{lemma}
	
	\begin{figure}
		\begin{tikzpicture}[line cap=round,line join=round,>=triangle 45,x=0.35 cm,y=0.35 cm]
			\clip(-4.472797830477084,-4.626156858833838) rectangle (4.25308689949056,4.274903060882107);
			\draw [line width=0.8 pt] (-2.,3.4641016151377544)-- (-4.,0.)-- (-2.,-3.4641016151377544)-- (2.,-3.4641016151377544)-- (4.,0.)-- (2.,3.4641016151377544);
			\draw [line width=0.8 pt] (-2.,3.4641016151377544)-- (2.,3.4641016151377544);
			\draw [line width=0.8 pt] (0.,0.) circle (1.41 cm);
			\draw [line width=0.4 pt,dash pattern=on 1pt off 1pt] (0.,0.)-- (-2.,-3.4641016151377544);
			\draw [line width=0.4 pt,dash pattern=on 1pt off 1pt] (0.,0.)-- (2.,-3.4641016151377544);
			\draw [line width=0.4 pt,dash pattern=on 1pt off 1pt] (0.,0.)-- (2.,3.4641016151377544);
			\draw [line width=0.4 pt,dash pattern=on 1pt off 1pt] (0.,0.)-- (-4.,0.);
			\draw [line width=0.4 pt,dash pattern=on 1pt off 1pt] (0.,0.)-- (4.,0.);
			\draw [line width=0.4 pt,dash pattern=on 1pt off 1pt] (0.,-0.02)-- (-2.,3.4441016151377544);
			\begin{scriptsize}
				\draw [fill=ffqqqq] (-2.,-3.4641016151377544) circle (1.5pt);
				\draw [fill=ffqqqq] (2.,-3.4641016151377544) circle (1.5pt);
				\draw [fill=ffqqqq] (4.,0.) circle (1.5pt);
				\draw [fill=ffqqqq] (2.,3.4641016151377544) circle (1.5pt);
				\draw [fill=ffqqqq] (-2.,3.4641016151377544) circle (1.5pt);
				\draw [fill=ffqqqq] (-4.,0.) circle (1.5pt);
			\end{scriptsize}
		\end{tikzpicture}
		\begin{tikzpicture}[line cap=round,line join=round,>=triangle 45,x=0.35 cm,y=0.35 cm]
			\clip(-4.472797830477084,-4.626156858833838) rectangle (4.25308689949056,4.274903060882107);
			\draw [line width=0.8 pt] (-2.,3.4641016151377544)-- (-4.,0.)-- (-2.,-3.4641016151377544)-- (2.,-3.4641016151377544)-- (4.,0.)-- (2.,3.4641016151377544);
			\draw [line width=0.8 pt] (-2.,3.4641016151377544)-- (2.,3.4641016151377544);
			\draw [line width=0.8 pt] (0.,0.) circle (1.41 cm);
			\draw [line width=0.4 pt,dash pattern=on 1pt off 1pt] (0.,0.)-- (-2.,-3.4641016151377544);
			\draw [line width=0.4 pt,dash pattern=on 1pt off 1pt] (0.,0.)-- (2.,-3.4641016151377544);
			\draw [line width=0.4 pt,dash pattern=on 1pt off 1pt] (0.,0.)-- (2.,3.4641016151377544);
			\draw [line width=0.4 pt,dash pattern=on 1pt off 1pt] (0.,0.)-- (-4.,0.);
			\draw [line width=0.4 pt,dash pattern=on 1pt off 1pt] (0.,0.)-- (4.,0.);
			\draw [line width=0.4 pt,dash pattern=on 1pt off 1pt] (0.,-0.02)-- (-2.,3.4441016151377544);
			\begin{scriptsize}
				\draw [fill=ffqqqq] (-2.,-3.4641016151377544) circle (1.5pt);
				\draw [fill=ffqqqq] (2.,-3.4641016151377544) circle (1.5pt);
				\draw [fill=ffqqqq] (4.,0.) circle (1.5pt);
				\draw [fill=ffqqqq] (2.,3.4641016151377544) circle (1.5pt);
				\draw [fill=ffqqqq] (-2.,3.4641016151377544) circle (1.5pt);
				\draw [fill=ffqqqq] (-4.,0.) circle (1.5pt);
				\draw [fill=ffqqqq] (0., -4) circle (1.5pt);
			\end{scriptsize}
		\end{tikzpicture}
		\begin{tikzpicture}[line cap=round,line join=round,>=triangle 45,x=0.35 cm,y=0.35 cm]
			\clip(-4.472797830477084,-4.626156858833838) rectangle (4.25308689949056,4.274903060882107);
			\draw [line width=0.8 pt] (-2.,3.4641016151377544)-- (-4.,0.)-- (-2.,-3.4641016151377544)-- (2.,-3.4641016151377544)-- (4.,0.)-- (2.,3.4641016151377544);
			\draw [line width=0.8 pt] (-2.,3.4641016151377544)-- (2.,3.4641016151377544);
			\draw [line width=0.8 pt] (0.,0.) circle (1.41 cm);
			\draw [line width=0.4 pt,dash pattern=on 1pt off 1pt] (0.,0.)-- (-2.,-3.4641016151377544);
			\draw [line width=0.4 pt,dash pattern=on 1pt off 1pt] (0.,0.)-- (2.,-3.4641016151377544);
			\draw [line width=0.4 pt,dash pattern=on 1pt off 1pt] (0.,0.)-- (2.,3.4641016151377544);
			\draw [line width=0.4 pt,dash pattern=on 1pt off 1pt] (0.,0.)-- (-4.,0.);
			\draw [line width=0.4 pt,dash pattern=on 1pt off 1pt] (0.,0.)-- (4.,0.);
			\draw [line width=0.4 pt,dash pattern=on 1pt off 1pt] (0.,-0.02)-- (-2.,3.4441016151377544);
			\begin{scriptsize}
				\draw [fill=ffqqqq] (-2.,-3.4641016151377544) circle (1.5pt);
				\draw [fill=ffqqqq] (2.,-3.4641016151377544) circle (1.5pt);
				\draw [fill=ffqqqq] (4.,0.) circle (1.5pt);
				\draw [fill=ffqqqq] (2.,3.4641016151377544) circle (1.5pt);
				\draw [fill=ffqqqq] (-2.,3.4641016151377544) circle (1.5pt);
				\draw [fill=ffqqqq] (-4.,0.) circle (1.5pt);
				\draw [fill=ffqqqq] (0., -4) circle (1.5pt);
				\draw [fill=ffqqqq] (3.4641, -2.) circle (1.5pt);
			\end{scriptsize}
		\end{tikzpicture} 
		\begin{tikzpicture}[line cap=round,line join=round,>=triangle 45,x=0.35 cm,y=0.35 cm]
			\clip(-4.472797830477084,-4.626156858833838) rectangle (4.25308689949056,4.274903060882107);
			\draw [line width=0.8 pt] (-2.,3.4641016151377544)-- (-4.,0.)-- (-2.,-3.4641016151377544)-- (2.,-3.4641016151377544)-- (4.,0.)-- (2.,3.4641016151377544);
			\draw [line width=0.8 pt] (-2.,3.4641016151377544)-- (2.,3.4641016151377544);
			\draw [line width=0.8 pt] (0.,0.) circle (1.41 cm);
			\draw [line width=0.4 pt,dash pattern=on 1pt off 1pt] (0.,0.)-- (-2.,-3.4641016151377544);
			\draw [line width=0.4 pt,dash pattern=on 1pt off 1pt] (0.,0.)-- (2.,-3.4641016151377544);
			\draw [line width=0.4 pt,dash pattern=on 1pt off 1pt] (0.,0.)-- (2.,3.4641016151377544);
			\draw [line width=0.4 pt,dash pattern=on 1pt off 1pt] (0.,0.)-- (-4.,0.);
			\draw [line width=0.4 pt,dash pattern=on 1pt off 1pt] (0.,0.)-- (4.,0.);
			\draw [line width=0.4 pt,dash pattern=on 1pt off 1pt] (0.,-0.02)-- (-2.,3.4441016151377544);
			\begin{scriptsize}
				\draw [fill=ffqqqq] (-2.,-3.4641016151377544) circle (1.5pt);
				\draw [fill=ffqqqq] (2.,-3.4641016151377544) circle (1.5pt);
				\draw [fill=ffqqqq] (4.,0.) circle (1.5pt);
				\draw [fill=ffqqqq] (2.,3.4641016151377544) circle (1.5pt);
				\draw [fill=ffqqqq] (-2.,3.4641016151377544) circle (1.5pt);
				\draw [fill=ffqqqq] (-4.,0.) circle (1.5pt);
				\draw [fill=ffqqqq] (0., -4) circle (1.5pt);
				\draw [fill=ffqqqq] (3.4641, -2.) circle (1.5pt);
				\draw [fill=ffqqqq] (3.4641, 2.) circle (1.5pt);
			\end{scriptsize}
		\end{tikzpicture}\\
		\begin{tikzpicture}[line cap=round,line join=round,>=triangle 45,x=0.35 cm,y=0.35 cm]
			\clip(-4.472797830477084,-4.626156858833838) rectangle (4.25308689949056,4.274903060882107);
			\draw [line width=0.8 pt] (-2.,3.4641016151377544)-- (-4.,0.)-- (-2.,-3.4641016151377544)-- (2.,-3.4641016151377544)-- (4.,0.)-- (2.,3.4641016151377544);
			\draw [line width=0.8 pt] (-2.,3.4641016151377544)-- (2.,3.4641016151377544);
			\draw [line width=0.8 pt] (0.,0.) circle (1.41 cm);
			\draw [line width=0.4 pt,dash pattern=on 1pt off 1pt] (0.,0.)-- (-2.,-3.4641016151377544);
			\draw [line width=0.4 pt,dash pattern=on 1pt off 1pt] (0.,0.)-- (2.,-3.4641016151377544);
			\draw [line width=0.4 pt,dash pattern=on 1pt off 1pt] (0.,0.)-- (2.,3.4641016151377544);
			\draw [line width=0.4 pt,dash pattern=on 1pt off 1pt] (0.,0.)-- (-4.,0.);
			\draw [line width=0.4 pt,dash pattern=on 1pt off 1pt] (0.,0.)-- (4.,0.);
			\draw [line width=0.4 pt,dash pattern=on 1pt off 1pt] (0.,-0.02)-- (-2.,3.4441016151377544);
			\begin{scriptsize}
				\draw [fill=ffqqqq] (-2.,-3.4641016151377544) circle (1.5pt);
				\draw [fill=ffqqqq] (2.,-3.4641016151377544) circle (1.5pt);
				\draw [fill=ffqqqq] (4.,0.) circle (1.5pt);
				\draw [fill=ffqqqq] (2.,3.4641016151377544) circle (1.5pt);
				\draw [fill=ffqqqq] (-2.,3.4641016151377544) circle (1.5pt);
				\draw [fill=ffqqqq] (-4.,0.) circle (1.5pt);
				\draw [fill=ffqqqq] (0., -4) circle (1.5pt);
				\draw [fill=ffqqqq] (3.4641, -2.) circle (1.5pt);
				\draw [fill=ffqqqq] (3.4641, 2.) circle (1.5pt);
				\draw [fill=ffqqqq] (0., 4) circle (1.5pt);
			\end{scriptsize}
		\end{tikzpicture} 
		\begin{tikzpicture}[line cap=round,line join=round,>=triangle 45,x=0.35 cm,y=0.35 cm]
			\clip(-4.472797830477084,-4.626156858833838) rectangle (4.25308689949056,4.274903060882107);
			\draw [line width=0.8 pt] (-2.,3.4641016151377544)-- (-4.,0.)-- (-2.,-3.4641016151377544)-- (2.,-3.4641016151377544)-- (4.,0.)-- (2.,3.4641016151377544);
			\draw [line width=0.8 pt] (-2.,3.4641016151377544)-- (2.,3.4641016151377544);
			\draw [line width=0.8 pt] (0.,0.) circle (1.41 cm);
			\draw [line width=0.4 pt,dash pattern=on 1pt off 1pt] (0.,0.)-- (-2.,-3.4641016151377544);
			\draw [line width=0.4 pt,dash pattern=on 1pt off 1pt] (0.,0.)-- (2.,-3.4641016151377544);
			\draw [line width=0.4 pt,dash pattern=on 1pt off 1pt] (0.,0.)-- (2.,3.4641016151377544);
			\draw [line width=0.4 pt,dash pattern=on 1pt off 1pt] (0.,0.)-- (-4.,0.);
			\draw [line width=0.4 pt,dash pattern=on 1pt off 1pt] (0.,0.)-- (4.,0.);
			\draw [line width=0.4 pt,dash pattern=on 1pt off 1pt] (0.,-0.02)-- (-2.,3.4441016151377544);
			\begin{scriptsize}
				\draw [fill=ffqqqq] (-2.,-3.4641016151377544) circle (1.5pt);
				\draw [fill=ffqqqq] (2.,-3.4641016151377544) circle (1.5pt);
				\draw [fill=ffqqqq] (4.,0.) circle (1.5pt);
				\draw [fill=ffqqqq] (2.,3.4641016151377544) circle (1.5pt);
				\draw [fill=ffqqqq] (-2.,3.4641016151377544) circle (1.5pt);
				\draw [fill=ffqqqq] (-4.,0.) circle (1.5pt);
				\draw [fill=ffqqqq] (0., -4) circle (1.5pt);
				\draw [fill=ffqqqq] (3.4641, -2.) circle (1.5pt);
				\draw [fill=ffqqqq] (3.4641, 2.) circle (1.5pt);
				\draw [fill=ffqqqq] (0., 4) circle (1.5pt);
				\draw [fill=ffqqqq] (-3.4641, 2.) circle (1.5pt);
			\end{scriptsize}
		\end{tikzpicture}
		\begin{tikzpicture}[line cap=round,line join=round,>=triangle 45,x=0.35 cm,y=0.35 cm]
			\clip(-4.472797830477084,-4.626156858833838) rectangle (4.25308689949056,4.274903060882107);
			\draw [line width=0.8 pt] (-2.,3.4641016151377544)-- (-4.,0.)-- (-2.,-3.4641016151377544)-- (2.,-3.4641016151377544)-- (4.,0.)-- (2.,3.4641016151377544);
			\draw [line width=0.8 pt] (-2.,3.4641016151377544)-- (2.,3.4641016151377544);
			\draw [line width=0.8 pt] (0.,0.) circle (1.41 cm);
			\draw [line width=0.4 pt,dash pattern=on 1pt off 1pt] (0.,0.)-- (-2.,-3.4641016151377544);
			\draw [line width=0.4 pt,dash pattern=on 1pt off 1pt] (0.,0.)-- (2.,-3.4641016151377544);
			\draw [line width=0.4 pt,dash pattern=on 1pt off 1pt] (0.,0.)-- (2.,3.4641016151377544);
			\draw [line width=0.4 pt,dash pattern=on 1pt off 1pt] (0.,0.)-- (-4.,0.);
			\draw [line width=0.4 pt,dash pattern=on 1pt off 1pt] (0.,0.)-- (4.,0.);
			\draw [line width=0.4 pt,dash pattern=on 1pt off 1pt] (0.,-0.02)-- (-2.,3.4441016151377544);
			\begin{scriptsize}
				\draw [fill=ffqqqq] (-2.,-3.4641016151377544) circle (1.5pt);
				\draw [fill=ffqqqq] (2.,-3.4641016151377544) circle (1.5pt);
				\draw [fill=ffqqqq] (4.,0.) circle (1.5pt);
				\draw [fill=ffqqqq] (2.,3.4641016151377544) circle (1.5pt);
				\draw [fill=ffqqqq] (-2.,3.4641016151377544) circle (1.5pt);
				\draw [fill=ffqqqq] (-4.,0.) circle (1.5pt);
				\draw [fill=ffqqqq] (0., -4) circle (1.5pt);
				\draw [fill=ffqqqq] (3.4641, -2.) circle (1.5pt);
				\draw [fill=ffqqqq] (3.4641, 2.) circle (1.5pt);
				\draw [fill=ffqqqq] (0., 4) circle (1.5pt);
				\draw [fill=ffqqqq] (-3.4641, 2.) circle (1.5pt);
				\draw [fill=ffqqqq] (-3.4641, -2.) circle (1.5pt);
			\end{scriptsize}
		\end{tikzpicture}
		\begin{tikzpicture}[line cap=round,line join=round,>=triangle 45,x=0.35 cm,y=0.35 cm]
			\clip(-4.472797830477084,-4.626156858833838) rectangle (4.25308689949056,4.274903060882107);
			\draw [line width=0.8 pt] (-2.,3.4641016151377544)-- (-4.,0.)-- (-2.,-3.4641016151377544)-- (2.,-3.4641016151377544)-- (4.,0.)-- (2.,3.4641016151377544);
			\draw [line width=0.8 pt] (-2.,3.4641016151377544)-- (2.,3.4641016151377544);
			\draw [line width=0.8 pt] (0.,0.) circle (1.41 cm);
			\draw [line width=0.4 pt,dash pattern=on 1pt off 1pt] (0.,0.)-- (-2.,-3.4641016151377544);
			\draw [line width=0.4 pt,dash pattern=on 1pt off 1pt] (0.,0.)-- (2.,-3.4641016151377544);
			\draw [line width=0.4 pt,dash pattern=on 1pt off 1pt] (0.,0.)-- (2.,3.4641016151377544);
			\draw [line width=0.4 pt,dash pattern=on 1pt off 1pt] (0.,0.)-- (-4.,0.);
			\draw [line width=0.4 pt,dash pattern=on 1pt off 1pt] (0.,0.)-- (4.,0.);
			\draw [line width=0.4 pt,dash pattern=on 1pt off 1pt] (0.,-0.02)-- (-2.,3.4441016151377544);
			\begin{scriptsize}
				\draw [fill=ffqqqq] (-2.,-3.4641016151377544) circle (1.5pt);
				\draw [fill=ffqqqq] (2.,-3.4641016151377544) circle (1.5pt);
				\draw [fill=ffqqqq] (4.,0.) circle (1.5pt);
				\draw [fill=ffqqqq] (2.,3.4641016151377544) circle (1.5pt);
				\draw [fill=ffqqqq] (-2.,3.4641016151377544) circle (1.5pt);
				\draw [fill=ffqqqq] (-4.,0.) circle (1.5pt);
				\draw [fill=ffqqqq] (3.4641, -2.) circle (1.5pt);
				\draw [fill=ffqqqq] (3.4641, 2.) circle (1.5pt);
				\draw [fill=ffqqqq] (0., 4) circle (1.5pt);
				\draw [fill=ffqqqq] (-3.4641, 2.) circle (1.5pt);
				\draw [fill=ffqqqq] (-3.4641, -2.) circle (1.5pt);
				\draw [fill=ffqqqq] (-0.724332, -3.93387) circle (1.5pt);
				\draw [fill=ffqqqq] (0.724332, -3.93387) circle (1.5pt);
			\end{scriptsize}
		\end{tikzpicture}
		\caption{In conditional constrained quantization with the constraint as the circumcircle, for a uniform distribution on a regular hexagon, the dots represent the elements in an optimal set of $n$-points for $6\leq n\leq 13$.} \label{Fig3}
	\end{figure}

	The results in the following examples are obtained using Methodology~\ref{method101} and Lemma~\ref{lemmaMe112} (see Figure~\ref{Fig3}). 
	\begin{exam} \label{exam}  
		Let $P$ be a uniform distribution on a regular hexagon. 
		To obtain the optimal set $\ga_7$ and the corresponding conditional constrained quantization error $V_7$,  we proceed as follows: By Lemma~\ref{lemmaMe112}, we can assume that $\te{card}(\ga_7\ii S_1)=3$ and $\te{card}(\ga_7\ii S_j)=2$ for $2\leq j\leq 6$. Then, using Methodology~\ref{method101}, we obtain
		\[\ga_7\ii S_1=\set{(-\frac 12, -\frac{\sqrt 3}{2}), (0, 1), (\frac 12, -\frac{\sqrt 3}{2})} \te{ and } V(P; \ga_7\ii S_1)=\frac{\sqrt{3}}{2}-\frac{31}{36}\approx 0.00492.\]
		Moreover, we have 
		\[\ga_7\ii S_j=T^j(\set{(-\frac 12, -\frac{\sqrt 3}{2}), (\frac 12, -\frac{\sqrt 3}{2})}) \te{ with } V(\ga_7\ii S_j)=\frac 1{72}  \te{ for } 2\leq j\leq 6. \]
		Hence, 
		\begin{align*} \ga_7& =\UU_{j=1}^6\ga_7\ii S_j=\gb\UU\set{(0, -1)} \te{ with } \\
			V_7&=V(P; \ga_7\ii S_1)+\sum_{j=2}^6 V(P; \ga_7\ii S_j)=\frac{\sqrt{3}}{2}-\frac{31}{36}+\frac 5{72}=\frac{\sqrt{3}}{2}-\frac{19}{24}\approx 0.07436. 
		\end{align*} 
	\end{exam} 
	
	\begin{exam}  Let $P$ be a uniform distribution on a regular hexagon. Proceeding analogously as Example~\ref{exam}, we obtain the following results (see Figure~\ref{Fig3}). 
		\begin{align*}
			\ga_8&=\gb\uu\set{(0, -1), T(0, -1)} \te{ with } V_8=2(\frac{\sqrt{3}}{2}-\frac{31}{36})+\frac 4{72}\approx 0.0653841; \\
			\ga_9&=\gb\uu\set{(0, -1), T(0, -1), T^2(0, -1)} \te{ with } V_9=3(\frac{\sqrt{3}}{2}-\frac{31}{36})+\frac 3{72}\approx 0.0564095 ; \\
			\ga_{10}&=\gb\uu\set{(0, -1), T(0, -1), T^2(0, -1), T^3(0, -1)} \te{ with } V_{10}=4(\frac{\sqrt{3}}{2}-\frac{31}{36})+\frac 2{72}\approx 0.0474349; \\
			\ga_{11}&=\gb\uu\set{(0, -1), T(0, -1), T^2(0, -1), T^3(0, -1), T^4(0, -1)}\\
			&\hspace{4cm} \te{ with } V_{11}=4(\frac{\sqrt{3}}{2}-\frac{31}{36})+\frac 1{72}\approx 0.0335461; \\
			\ga_{12}&=\gb\uu\set{(0, -1), T(0, -1), T^2(0, -1), T^3(0, -1), T^4(0, -1), T^5(0, -1)}\\
			&\hspace{4cm} \te{ with } V_{12}=5(\frac{\sqrt{3}}{2}-\frac{31}{36})\approx 0.0245715 .
		\end{align*} 
	\end{exam} 
	\begin{exam}  Let $P$ be a uniform distribution on a regular hexagon.  To obtain the constrained optimal set of $n$-points for $n=13$ (see Figure~\ref{Fig3}), we proceed as follows: Notice that $n=13=2\times 6+1$. Hence, by Lemma~\ref{lemmaMe112}, we can assume that 
		$\te{card}(\ga_{13}\ii S_1)=4, \te{card}(\ga_{13}\ii S_j)=3 \te{ for } 2\leq j\leq 6.$ Then, using Methodology~\ref{method101}, we obtain 
		$\ga_{13}\ii S_1$ and $\ga_{13}\ii S_j$ for $2\leq j\leq 6$ yielding 
		\begin{align*}
			\ga_{13}\ii S_1&=\set{(-\frac 12, -\frac{\sqrt 3}{2}), (-0.181083, -0.983468), (0.181083, -0.983468), (\frac 12, -\frac{\sqrt 3}{2})}, \\
			\ga_{13}\ii S_j&=T^j\Big(\set{(-\frac 12, -\frac{\sqrt 3}{2}), (0, 1), (\frac 12, -\frac{\sqrt 3}{2})}\Big) \te{ for } 2\leq j\leq 6. 
		\end{align*} 
		Hence, 
		\[\ga_{13}=\gb \uu \set{(-0.181083, -0.983468), (0.181083, -0.983468)}\uu \UU_{j=2}^6T^j\Big(\set{(-\frac 12, -\frac{\sqrt 3}{2}), (0, 1), (\frac 12, -\frac{\sqrt 3}{2})}\Big),\]
		and \[V_{13}=V(P; \ga_{13}\ii S_1)+\sum_{j=2}^6 V(P; \ga_{13}\ii S_j)= \frac{1}{72} \left(27\ 3^{2/3}-17-27 \sqrt[3]{3}\right)+5(\frac{\sqrt{3}}{2}-\frac{31}{36})\approx 0.11513.\]
	\end{exam} 
	
	\section{Conditional constrained quantization for the uniform distribution $P$ on a regular $k$-gon when the constraint is the incircle} \label{sec4}  
	
	In this section, for all $n\geq k$, we calculate the conditional constrained optimal sets $\ga_n$ of $n$-points and the corresponding $n$th conditional constrained quantization errors $V_n$ for the uniform distribution $P$ on a regular polygon with respect to the conditional set $\gb$ and the constraint $S$ which is the incircle of the polygon and is represented by the parametric equations as follows: 
	\[S:=\set{(r\cos\gq, r\sin \gq) : r=\cos \frac \pi k \te{ and } 0\leq \gq\leq 2\pi}.\] 
	Let $S_i$ be the circular arcs of the incircle lying over the sides $L_i$ subtending the same central angle of $\frac {2\pi} k$ for $1\leq i\leq k$ (see Figure~\ref{Fig4} when $k=6$). Let $\ga_n$ be an optimal set of $n$-points for $P$ for some $n\geq k$. Notice that if $n=k$, then $\ga_n=\gb$, where $\gb$ is the conditional set, i.e., $\ga_k$ does not contain any element from $S_i$ for any $1\leq i\leq k$. Hence, in the sequel we assume that $n\geq k+1$. 
	Let $\ga_n$ contain $n_i$ elements from $S_i\uu L_i$ including the two elements in the set $L_i\ii \gb$, i.e., $\te{card}(\ga_n\ii (S_i\uu L_i))=n_i$, where $n_i \geq 2$ for $1\leq i\leq k$. Notice that to obtain $\ga_n\ii (S_i\uu L_i)$ and the corresponding distortion error $V(P; \ga_n\ii (S_i\uu L_i))$ contributed by the elements in $\ga_n\ii(S_i\uu L_i)$ on $L_i$, where $\te{card}(\ga_n\ii (S_i\uu L_i))=n_i$, it is enough to obtain $\ga_n\ii (S_1\uu L_1)$ and the corresponding distortion error $V(P; \ga_n\ii (S_1\uu L_1))$ with $\te{card}(\ga_n\ii (S_1\uu L_1))=n_i$, and then to use the affine transformation $T$, as we know
	\[\ga_n\ii (S_i\uu L_i)=T^{i-1}(\ga_n\ii (S_1\uu L_1)), \te{ where } \te{card}(\ga_n\ii (S_1\uu L_1))=n_i, \te{ for } 2\leq i\leq k.\]
	Moreover, notice that $V(P; \ga_n\ii (S_i\uu L_i))=V(P; \ga_n\ii(S_1\uu L_1)), \te{ where } \te{card}(\ga_n\ii (S_1\uu L_1))=n_i$ and $V(P; \ga_n\ii (S_1\uu L_1))$ denotes the distortion error contributed by the $n_i$ elements on $S_1\uu L_1$. 
	
	Since $P$ is a uniform distribution and $S_i\uu L_i$ are rotationally symmetric, the following lemma, which is similar to Lemma~\ref{lemmaMe1}, is also true here.    
	\begin{lemma} \label{lemmaMe113} 
		Let $\ga_n$ be an optimal set of $n$-points with $\te{card}(\ga_n\ii (S_i\uu L_i))=n_i$ for $1\leq i\leq k$. Then, $|n_i-n_j|\in \set{0,1}$ for $1\leq i\neq j\leq k$. 
	\end{lemma}

	\begin{remark} \label{rem00}  
		Notice the dotted circles in Figure~\ref{Fig4}, when $k=6$, with centers $A_1(-\frac 12, -\frac{\sqrt 3}{2})$ and $A_2(\frac 12, -\frac{\sqrt 3}{2})$. Analogously, it is also true for any regular polygon with $k$ sides. Recall that $P$ is a uniform distribution, the support of which is a regular polygon with $k$ sides, and the elements in an optimal set $\ga_n$ of $n$-points are located in such a way so that the overall distortion error is minimum. Hence, from the geometry of Figure~\ref{Fig4}, without going into the details of calculations, we can assume that there cannot be any element in the optimal set $\ga_n$ lying on the lines $OA_1$, or $OA_2$, i.e., on any line joining the center $O$ of the circle and any of the vertices of the regular polygon, i.e., the Voronoi regions of any element in $\ga_n\ii S_i$ contains elements from $L_i$, and does not contain any element from $L_j$ for $i\neq j$. 
	\end{remark}

	\begin{method}\label{method11}
		To obtain the conditional constrained optimal sets $\ga_n$ of $n$-points and the $n$th conditional constrained quantization error for the uniform distribution $P$ on the regular polygon with $k$ sides with respect to the conditional set $\gb$ and the constraint $S$ for any $n\geq k+1$, assume that $\te{card}(\ga_n\ii S_1)=\ell$, i.e., 
		\[\ga_n\ii S_1=\set{(r\cos \gq_j, r\sin \gq_j) : 1\leq j\leq \ell}\]
		for $\ell\geq 1$ and $r=\cos \frac \pi k$. If $\ell=1$, then by Remark~\ref{rem00}, we can assume that $\ga_n\ii S_1$ contains the midpoint of the side $L_1$. We assume that $\ell\geq 2$. 
		Let the elements in $\ga_n\ii S_1$ be represented by the parameters $\gq_1<\gq_1<\gq_2<\dots<\gq_\ell$, where $\frac{3\pi}{2}-\frac{\pi}k<\gq_1<\cdots<\gq_\ell<\frac{3\pi}{2}+\frac{\pi}k$. 
		Let the boundary of the Voronoi regions of the element $A_1(\cos(\frac{3\pi}{2}-\frac{\pi}k), \sin(\frac{3\pi}{2}-\frac{\pi}k))$ and $\gq_1$ intersect the side $L_1$ at the element $(a_1, -\cos \frac \pi k )$,  the boundary of the Voronoi regions of the elements $\gq_j$ and $\gq_{j+1}$ intersect the side $L_1$ at the element $(a_{j+1}, -\cos \frac \pi k )$, where $1\leq j\leq \ell-1$, and the boundary of the Voronoi regions of the element $\gq_\ell$ and $A_2(\cos(\frac{3\pi}{2}+\frac{\pi}k), \sin(\frac{3\pi}{2}+\frac{\pi}k))$ intersect the side $L_1$ at the element $(a_{\ell+1}, -\cos \frac \pi k )$. 
		Then, the canonical equations are given by
		\begin{align*}
			&\rho((\cos(\frac{3\pi}{2}-\frac{\pi}k), \sin(\frac{3\pi}{2}-\frac{\pi}k)), (a_1, -\cos \frac \pi k ))-\rho((r\cos \gq_1, r\sin \gq_1), (a_1, -\cos \frac \pi k ))=0, \\
			&\rho((r\cos\gq_j, r\sin\gq_j), (a_{j+1}, -\cos \frac \pi k ))-\rho((r\cos\gq_{j+1}, r\sin\gq_{j+1}), (a_{j+1}, -\cos \frac \pi k ))=0 \te{ for } 1\leq j\leq \ell-1,  \te{ and }\\
			&\rho((\cos(\frac{3\pi}{2}+\frac{\pi}k), \sin(\frac{3\pi}{2}+\frac{\pi}k)), (a_{\ell+1}, -\cos \frac \pi k ))-\rho((r\cos \gq_\ell, r\sin \gq_\ell), (a_{\ell+1}, -\cos \frac \pi k ))=0,
		\end{align*}  
		where $r=\cos\frac{\pi}k$.
		Solving the canonical equations, we obtain the values of $a_1, a_2, \cdots, a_{\ell+1}$ in terms of $\gq_1, \gq_2, \cdots, \gq_\ell$ for $\ell\geq 2$. Hence, if $V(P; \ga_n\ii(S_1\uu L_1))$ is the distortion error contributed by the elements in $\ga_n\ii S_1$ on $L_1$, then we have 
		\begin{align}\label{Megha2} 
			&V(P; \ga_n\ii (S_1\uu L_1))\\
			&=\frac 1 k\Big(\int_{-\cos(\frac{3\pi}{2}-\frac{\pi}k)}^{a_1} \rho((x,-\cos\frac \pi k), (\cos(\frac{3\pi}{2}-\frac{\pi}k), \sin(\frac{3\pi}{2}-\frac{\pi}k))) \, dx \notag\\
			&+\sum_{j=1}^{k}\int_{a_j}^{a_{j+1}} \rho((x,-\cos\frac \pi k), (\cos\theta_{j} ,\sin \theta_{j})) \, dx\notag\\
			&\qquad +\int_{a_{k+1}}^{\frac 12} \rho((x,-\cos\frac \pi k), (\cos(\frac{3\pi}{2}+\frac{\pi}k), \sin(\frac{3\pi}{2}+\frac{\pi}k))) \, dx.
		\end{align} 
		Putting the values of $a_j$ we obtain the distortion error $V(P; \ga_n\ii (S_1\uu L_1))$, which is a function of $\gq_j$ for $1\leq j\leq \ell$. Now, using calculus, we obtain the values of $\gq_j$ for $1\leq j\leq \ell$ for which $V(P; \ga_n\ii (S_1\uu L_1))$ is minimum. Upon obtained the values of $\gq_j$, we can easily obtain
		$\ga_n\ii (S_1\uu L_1)$ and the corresponding distortion error $V(P; \ga_n\ii (S_1\uu L_1))$. Due to the affine transformation, we can also obtain $\ga_n\ii (S_i\uu L_i)$ and the corresponding distortion error $V(P; \ga_n\ii (S_i\uu L_i))$ for $2\leq i\leq k$. Thus, we have 
		\[\ga_n=\UU_{i=1}^k(\ga_n\ii (S_i\uu L_i)) \te{ with the $n$th constrained quantization error } V_n=\sum_{i=1}^k V(P; \ga_n\ii (S_i\uu L_i)).\]
	\end{method}
	
	\begin{figure}
		\vspace{-0.3 in} 
		\begin{tikzpicture}[line cap=round,line join=round,>=triangle 45,x=0.5 cm,y=0.5 cm]
			\clip(-6.8,-6.5) rectangle (6.5,6.5);
			\draw [line width=0.8 pt] (-2.,3.4641016151377544)-- (-4.,0.)-- (-2.,-3.4641016151377544)-- (2.,-3.4641016151377544)-- (4.,0.)-- (2.,3.4641016151377544);
			\draw [line width=0.8 pt] (-2.,3.4641016151377544)-- (2.,3.4641016151377544);
			\draw [line width=0.8 pt] (0.,0.) circle (3.4641);
			\draw [line width=0.2 pt,dash pattern=on 1pt off 1pt] (-2.,-3.4641016151377544) circle (0.5492747006848391);
			\draw [line width=0.2 pt,dash pattern=on 1pt off 1pt] (2.,-3.4641016151377544) circle (0.5268507093427136);
			\draw (-7.00,-3.10) node[anchor=north west] {$A_1 (-\frac 12, -\frac{\sqrt{3}}{2})$};
			\draw (2.10,-3.10) node[anchor=north west] {$A_2(\frac 12, -\frac{\sqrt{3}}{2})$};
			\draw (1.8,4.66) node[anchor=north west] {$(\frac 12, \frac{\sqrt{3}}{2})$};
			\draw (-5.4,4.66) node[anchor=north west] {$(-\frac 12, \frac{\sqrt{3}}{2})$};
			\draw (3.75,0.65) node[anchor=north west] {$(1, 0)$};
			\draw (-6.8,0.65) node[anchor=north west] {$(-1, 0)$};
			\draw (-0.591474726580256,-0.06591068017200247) node[anchor=north west] {$O$};
			\draw [line width=0.2 pt,dash pattern=on 1pt off 1pt] (0.,0.)-- (-2.,-3.4641016151377544);
			\draw [line width=0.2,dash pattern=on 1pt off 1pt] (0.,0.)-- (2.,-3.4641016151377544);
			\draw [line width=0.2,dash pattern=on 1pt off 1pt] (0.,0.)-- (2.,3.4641016151377544);
			\draw [line width=0.2,dash pattern=on 1pt off 1pt] (0.,0.)-- (-4.,0.);
			\draw [line width=0.2,dash pattern=on 1pt off 1pt] (0.,0.)-- (4.,0.);
			\draw [line width=0.2,dash pattern=on 1pt off 1pt] (0.,-0.02)-- (-2.,3.4441016151377544);
			\begin{scriptsize}
				\draw [fill=ffqqqq] (-2.,-3.4641016151377544) circle (1.0pt);
				\draw [fill=ffqqqq] (2.,-3.4641016151377544) circle (1.0pt);
				\draw [fill=ffqqqq] (4.,0.) circle (1.0pt);
				\draw [fill=ffqqqq] (2.,3.4641016151377544) circle (1.0pt);
				\draw [fill=ffqqqq] (-2.,3.4641016151377544) circle (1.0pt);
				\draw [fill=ffqqqq] (-4.,0.) circle (1.0pt);
				\draw [fill=ffqqqq] (0.,0.) circle (1.0pt);
			\end{scriptsize}
		\end{tikzpicture}
		\vspace{-0.5 in}
		\caption{The regular hexagon with a circle inscribed in it.} \label{Fig4}
	\end{figure}
	
	Using the above methodology, we obtain the results in the following examples. 
	
	\begin{exam}
		For a uniform distribution on a regular hexagon, i.e., when $k=6$,  
		let $\ga_n$ be a $n$th conditional constrained optimal set with the conditional constrained quantization error $V_n$. Then, we obtain the following results (see Figure~\ref{Fig5}): 
		\begin{align*}
			\ga_6&=\gb \te{ with } V_6=\frac 1 {12};\\
			\ga_7&=\gb\uu \set{(0, -\frac{\sqrt 3}2)} \te{ with } V_7=V(P; \ga_n\ii (S_1\uu L_1))+\sum_{j=2}^6 V(P; \ga_n\ii (S_j\uu L_j))=\frac{7}{96}\approx 0.0729;\\
			\ga_8&=\gb\uu \set{(0, -\frac{\sqrt 3}2), T(0, -\frac{\sqrt 3}2)} \te{ with } V_8=\frac{1}{16}\approx 0.0625;\\
			\ga_9&=\gb\uu \set{(0, -\frac{\sqrt 3}2), T(0, -\frac{\sqrt 3}2),  T^2(0, -\frac{\sqrt 3}2)} \te{ with } V_9=\frac{5}{96}\approx 0.0521;\\
			\ga_{10}&=\gb\uu \set{(0, -\frac{\sqrt 3}2), T(0, -\frac{\sqrt 3}2), T^2(0, -\frac{\sqrt 3}2), T^3(0, -\frac{\sqrt 3}2)} \te{ with } V_{10}=\frac{1}{24}\approx 0.0417;\\
			\ga_{11}&=\gb\uu \set{(0, -\frac{\sqrt 3}2), T(0, -\frac{\sqrt 3}2), T^2(0, -\frac{\sqrt 3}2), T^3(0, -\frac{\sqrt 3}2), T^4(0, -\frac{\sqrt 3}2)} \te{ with } V_{11}=\frac{1}{32}\approx 0.03125;\\
			\ga_{12}&=\gb\uu \set{(0, -\frac{\sqrt 3}2), T(0, -\frac{\sqrt 3}2), T^2(0, -\frac{\sqrt 3}2), T^3(0, -\frac{\sqrt 3}2), T^4(0, -\frac{\sqrt 3}2), T^5(0, -\frac{\sqrt 3}2)}\\
			&\hspace{3cm} \te{ with } V_{12}=\frac{1}{48}\approx0.0208. 
		\end{align*} 
	\end{exam}

	\begin{figure}
		\begin{tikzpicture}[line cap=round,line join=round,>=triangle 45,x=0.35 cm,y=0.35 cm]
			\clip(-4.472797830477084,-4.626156858833838) rectangle (4.25308689949056,4.274903060882107);
			\draw [line width=0.8 pt] (-2.,3.4641016151377544)-- (-4.,0.)-- (-2.,-3.4641016151377544)-- (2.,-3.4641016151377544)-- (4.,0.)-- (2.,3.4641016151377544);
			\draw [line width=0.8 pt] (-2.,3.4641016151377544)-- (2.,3.4641016151377544);
			\draw [line width=0.8 pt] (0.,0.) circle (3.4641);
			\draw [line width=0.2 pt,dash pattern=on 1pt off 1pt] (0.,0.)-- (-2.,-3.4641016151377544);
			\draw [line width=0.2,dash pattern=on 1pt off 1pt] (0.,0.)-- (2.,-3.4641016151377544);
			\draw [line width=0.2,dash pattern=on 1pt off 1pt] (0.,0.)-- (2.,3.4641016151377544);
			\draw [line width=0.2,dash pattern=on 1pt off 1pt] (0.,0.)-- (-4.,0.);
			\draw [line width=0.2,dash pattern=on 1pt off 1pt] (0.,0.)-- (4.,0.);
			\draw [line width=0.2,dash pattern=on 1pt off 1pt] (0.,-0.02)-- (-2.,3.4441016151377544);
			\begin{scriptsize}
				\draw [fill=ffqqqq] (-2.,-3.4641016151377544) circle (1.5pt);
				\draw [fill=ffqqqq] (2.,-3.4641016151377544) circle (1.5pt);
				\draw [fill=ffqqqq] (4.,0.) circle (1.5pt);
				\draw [fill=ffqqqq] (2.,3.4641016151377544) circle (1.5pt);
				\draw [fill=ffqqqq] (-2.,3.4641016151377544) circle (1.5pt);
				\draw [fill=ffqqqq] (-4.,0.) circle (1.5pt);
			\end{scriptsize}
		\end{tikzpicture}
		\begin{tikzpicture}[line cap=round,line join=round,>=triangle 45,x=0.35 cm,y=0.35 cm]
			\clip(-4.472797830477084,-4.626156858833838) rectangle (4.25308689949056,4.274903060882107);
			\draw [line width=0.8 pt] (-2.,3.4641016151377544)-- (-4.,0.)-- (-2.,-3.4641016151377544)-- (2.,-3.4641016151377544)-- (4.,0.)-- (2.,3.4641016151377544);
			\draw [line width=0.8 pt] (-2.,3.4641016151377544)-- (2.,3.4641016151377544);
			\draw [line width=0.8 pt] (0.,0.) circle (3.4641);
			\draw [line width=0.2 pt,dash pattern=on 1pt off 1pt] (0.,0.)-- (-2.,-3.4641016151377544);
			\draw [line width=0.2,dash pattern=on 1pt off 1pt] (0.,0.)-- (2.,-3.4641016151377544);
			\draw [line width=0.2,dash pattern=on 1pt off 1pt] (0.,0.)-- (2.,3.4641016151377544);
			\draw [line width=0.2,dash pattern=on 1pt off 1pt] (0.,0.)-- (-4.,0.);
			\draw [line width=0.2,dash pattern=on 1pt off 1pt] (0.,0.)-- (4.,0.);
			\draw [line width=0.2,dash pattern=on 1pt off 1pt] (0.,-0.02)-- (-2.,3.4441016151377544);
			\begin{scriptsize}
				\draw [fill=ffqqqq] (-2.,-3.4641016151377544) circle (1.5pt);
				\draw [fill=ffqqqq] (2.,-3.4641016151377544) circle (1.5pt);
				\draw [fill=ffqqqq] (4.,0.) circle (1.5pt);
				\draw [fill=ffqqqq] (2.,3.4641016151377544) circle (1.5pt);
				\draw [fill=ffqqqq] (-2.,3.4641016151377544) circle (1.5pt);
				\draw [fill=ffqqqq] (-4.,0.) circle (1.5pt);
				\draw [fill=ffqqqq] (0.,-3.4641) circle (1.5pt);
			\end{scriptsize}
		\end{tikzpicture}
		\begin{tikzpicture}[line cap=round,line join=round,>=triangle 45,x=0.35 cm,y=0.35 cm]
			\clip(-4.472797830477084,-4.626156858833838) rectangle (4.25308689949056,4.274903060882107);
			\draw [line width=0.8 pt] (-2.,3.4641016151377544)-- (-4.,0.)-- (-2.,-3.4641016151377544)-- (2.,-3.4641016151377544)-- (4.,0.)-- (2.,3.4641016151377544);
			\draw [line width=0.8 pt] (-2.,3.4641016151377544)-- (2.,3.4641016151377544);
			\draw [line width=0.8 pt] (0.,0.) circle (3.4641);
			\draw [line width=0.2 pt,dash pattern=on 1pt off 1pt] (0.,0.)-- (-2.,-3.4641016151377544);
			\draw [line width=0.2,dash pattern=on 1pt off 1pt] (0.,0.)-- (2.,-3.4641016151377544);
			\draw [line width=0.2,dash pattern=on 1pt off 1pt] (0.,0.)-- (2.,3.4641016151377544);
			\draw [line width=0.2,dash pattern=on 1pt off 1pt] (0.,0.)-- (-4.,0.);
			\draw [line width=0.2,dash pattern=on 1pt off 1pt] (0.,0.)-- (4.,0.);
			\draw [line width=0.2,dash pattern=on 1pt off 1pt] (0.,-0.02)-- (-2.,3.4441016151377544);
			\begin{scriptsize}
				\draw [fill=ffqqqq] (-2.,-3.4641016151377544) circle (1.5pt);
				\draw [fill=ffqqqq] (2.,-3.4641016151377544) circle (1.5pt);
				\draw [fill=ffqqqq] (4.,0.) circle (1.5pt);
				\draw [fill=ffqqqq] (2.,3.4641016151377544) circle (1.5pt);
				\draw [fill=ffqqqq] (-2.,3.4641016151377544) circle (1.5pt);
				\draw [fill=ffqqqq] (-4.,0.) circle (1.5pt);
				\draw [fill=ffqqqq] (0.,-3.4641) circle (1.5pt);
				\draw [fill=ffqqqq] (3., -1.73205) circle (1.5pt);
			\end{scriptsize}
		\end{tikzpicture}
		\begin{tikzpicture}[line cap=round,line join=round,>=triangle 45,x=0.35 cm,y=0.35 cm]
			\clip(-4.472797830477084,-4.626156858833838) rectangle (4.25308689949056,4.274903060882107);
			\draw [line width=0.8 pt] (-2.,3.4641016151377544)-- (-4.,0.)-- (-2.,-3.4641016151377544)-- (2.,-3.4641016151377544)-- (4.,0.)-- (2.,3.4641016151377544);
			\draw [line width=0.8 pt] (-2.,3.4641016151377544)-- (2.,3.4641016151377544);
			\draw [line width=0.8 pt] (0.,0.) circle (3.4641);
			\draw [line width=0.2 pt,dash pattern=on 1pt off 1pt] (0.,0.)-- (-2.,-3.4641016151377544);
			\draw [line width=0.2,dash pattern=on 1pt off 1pt] (0.,0.)-- (2.,-3.4641016151377544);
			\draw [line width=0.2,dash pattern=on 1pt off 1pt] (0.,0.)-- (2.,3.4641016151377544);
			\draw [line width=0.2,dash pattern=on 1pt off 1pt] (0.,0.)-- (-4.,0.);
			\draw [line width=0.2,dash pattern=on 1pt off 1pt] (0.,0.)-- (4.,0.);
			\draw [line width=0.2,dash pattern=on 1pt off 1pt] (0.,-0.02)-- (-2.,3.4441016151377544);
			\begin{scriptsize}
				\draw [fill=ffqqqq] (-2.,-3.4641016151377544) circle (1.5pt);
				\draw [fill=ffqqqq] (2.,-3.4641016151377544) circle (1.5pt);
				\draw [fill=ffqqqq] (4.,0.) circle (1.5pt);
				\draw [fill=ffqqqq] (2.,3.4641016151377544) circle (1.5pt);
				\draw [fill=ffqqqq] (-2.,3.4641016151377544) circle (1.5pt);
				\draw [fill=ffqqqq] (-4.,0.) circle (1.5pt);
				\draw [fill=ffqqqq] (0.,-3.4641) circle (1.5pt);
				\draw [fill=ffqqqq] (3., -1.73205) circle (1.5pt);
				\draw [fill=ffqqqq] (3., 1.73205) circle (1.5pt);
			\end{scriptsize}
		\end{tikzpicture}\\
		\begin{tikzpicture}[line cap=round,line join=round,>=triangle 45,x=0.35 cm,y=0.35 cm]
			\clip(-4.472797830477084,-4.626156858833838) rectangle (4.25308689949056,4.274903060882107);
			\draw [line width=0.8 pt] (-2.,3.4641016151377544)-- (-4.,0.)-- (-2.,-3.4641016151377544)-- (2.,-3.4641016151377544)-- (4.,0.)-- (2.,3.4641016151377544);
			\draw [line width=0.8 pt] (-2.,3.4641016151377544)-- (2.,3.4641016151377544);
			\draw [line width=0.8 pt] (0.,0.) circle (3.4641);
			\draw [line width=0.2 pt,dash pattern=on 1pt off 1pt] (0.,0.)-- (-2.,-3.4641016151377544);
			\draw [line width=0.2,dash pattern=on 1pt off 1pt] (0.,0.)-- (2.,-3.4641016151377544);
			\draw [line width=0.2,dash pattern=on 1pt off 1pt] (0.,0.)-- (2.,3.4641016151377544);
			\draw [line width=0.2,dash pattern=on 1pt off 1pt] (0.,0.)-- (-4.,0.);
			\draw [line width=0.2,dash pattern=on 1pt off 1pt] (0.,0.)-- (4.,0.);
			\draw [line width=0.2,dash pattern=on 1pt off 1pt] (0.,-0.02)-- (-2.,3.4441016151377544);
			\begin{scriptsize}
				\draw [fill=ffqqqq] (-2.,-3.4641016151377544) circle (1.5pt);
				\draw [fill=ffqqqq] (2.,-3.4641016151377544) circle (1.5pt);
				\draw [fill=ffqqqq] (4.,0.) circle (1.5pt);
				\draw [fill=ffqqqq] (2.,3.4641016151377544) circle (1.5pt);
				\draw [fill=ffqqqq] (-2.,3.4641016151377544) circle (1.5pt);
				\draw [fill=ffqqqq] (-4.,0.) circle (1.5pt);
				\draw [fill=ffqqqq] (0.,-3.4641) circle (1.5pt);
				\draw [fill=ffqqqq] (3., -1.73205) circle (1.5pt);
				\draw [fill=ffqqqq] (3., 1.73205) circle (1.5pt);
				\draw [fill=ffqqqq] (0.,3.4641) circle (1.5pt);
			\end{scriptsize}
		\end{tikzpicture}
		\begin{tikzpicture}[line cap=round,line join=round,>=triangle 45,x=0.35 cm,y=0.35 cm]
			\clip(-4.472797830477084,-4.626156858833838) rectangle (4.25308689949056,4.274903060882107);
			\draw [line width=0.8 pt] (-2.,3.4641016151377544)-- (-4.,0.)-- (-2.,-3.4641016151377544)-- (2.,-3.4641016151377544)-- (4.,0.)-- (2.,3.4641016151377544);
			\draw [line width=0.8 pt] (-2.,3.4641016151377544)-- (2.,3.4641016151377544);
			\draw [line width=0.8 pt] (0.,0.) circle (3.4641);
			\draw [line width=0.2 pt,dash pattern=on 1pt off 1pt] (0.,0.)-- (-2.,-3.4641016151377544);
			\draw [line width=0.2,dash pattern=on 1pt off 1pt] (0.,0.)-- (2.,-3.4641016151377544);
			\draw [line width=0.2,dash pattern=on 1pt off 1pt] (0.,0.)-- (2.,3.4641016151377544);
			\draw [line width=0.2,dash pattern=on 1pt off 1pt] (0.,0.)-- (-4.,0.);
			\draw [line width=0.2,dash pattern=on 1pt off 1pt] (0.,0.)-- (4.,0.);
			\draw [line width=0.2,dash pattern=on 1pt off 1pt] (0.,-0.02)-- (-2.,3.4441016151377544);
			\begin{scriptsize}
				\draw [fill=ffqqqq] (-2.,-3.4641016151377544) circle (1.5pt);
				\draw [fill=ffqqqq] (2.,-3.4641016151377544) circle (1.5pt);
				\draw [fill=ffqqqq] (4.,0.) circle (1.5pt);
				\draw [fill=ffqqqq] (2.,3.4641016151377544) circle (1.5pt);
				\draw [fill=ffqqqq] (-2.,3.4641016151377544) circle (1.5pt);
				\draw [fill=ffqqqq] (-4.,0.) circle (1.5pt);
				\draw [fill=ffqqqq] (0.,-3.4641) circle (1.5pt);
				\draw [fill=ffqqqq] (3., -1.73205) circle (1.5pt);
				\draw [fill=ffqqqq] (3., 1.73205) circle (1.5pt);
				\draw [fill=ffqqqq] (0.,3.4641) circle (1.5pt);
				\draw [fill=ffqqqq] (-3., 1.73205) circle (1.5pt);
			\end{scriptsize}
		\end{tikzpicture}
		\begin{tikzpicture}[line cap=round,line join=round,>=triangle 45,x=0.35 cm,y=0.35 cm]
			\clip(-4.472797830477084,-4.626156858833838) rectangle (4.25308689949056,4.274903060882107);
			\draw [line width=0.8 pt] (-2.,3.4641016151377544)-- (-4.,0.)-- (-2.,-3.4641016151377544)-- (2.,-3.4641016151377544)-- (4.,0.)-- (2.,3.4641016151377544);
			\draw [line width=0.8 pt] (-2.,3.4641016151377544)-- (2.,3.4641016151377544);
			\draw [line width=0.8 pt] (0.,0.) circle (3.4641);
			\draw [line width=0.2 pt,dash pattern=on 1pt off 1pt] (0.,0.)-- (-2.,-3.4641016151377544);
			\draw [line width=0.2,dash pattern=on 1pt off 1pt] (0.,0.)-- (2.,-3.4641016151377544);
			\draw [line width=0.2,dash pattern=on 1pt off 1pt] (0.,0.)-- (2.,3.4641016151377544);
			\draw [line width=0.2,dash pattern=on 1pt off 1pt] (0.,0.)-- (-4.,0.);
			\draw [line width=0.2,dash pattern=on 1pt off 1pt] (0.,0.)-- (4.,0.);
			\draw [line width=0.2,dash pattern=on 1pt off 1pt] (0.,-0.02)-- (-2.,3.4441016151377544);
			\begin{scriptsize}
				\draw [fill=ffqqqq] (-2.,-3.4641016151377544) circle (1.5pt);
				\draw [fill=ffqqqq] (2.,-3.4641016151377544) circle (1.5pt);
				\draw [fill=ffqqqq] (4.,0.) circle (1.5pt);
				\draw [fill=ffqqqq] (2.,3.4641016151377544) circle (1.5pt);
				\draw [fill=ffqqqq] (-2.,3.4641016151377544) circle (1.5pt);
				\draw [fill=ffqqqq] (-4.,0.) circle (1.5pt);
				\draw [fill=ffqqqq] (0.,-3.4641) circle (1.5pt);
				\draw [fill=ffqqqq] (3., -1.73205) circle (1.5pt);
				\draw [fill=ffqqqq] (3., 1.73205) circle (1.5pt);
				\draw [fill=ffqqqq] (0.,3.4641) circle (1.5pt);
				\draw [fill=ffqqqq] (-3., 1.73205) circle (1.5pt);
				\draw [fill=ffqqqq] (-3., -1.73205) circle (1.5pt);
			\end{scriptsize}
		\end{tikzpicture}
		\begin{tikzpicture}[line cap=round,line join=round,>=triangle 45,x=0.35 cm,y=0.35 cm]
			\clip(-4.472797830477084,-4.626156858833838) rectangle (4.25308689949056,4.274903060882107);
			\draw [line width=0.8 pt] (-2.,3.4641016151377544)-- (-4.,0.)-- (-2.,-3.4641016151377544)-- (2.,-3.4641016151377544)-- (4.,0.)-- (2.,3.4641016151377544);
			\draw [line width=0.8 pt] (-2.,3.4641016151377544)-- (2.,3.4641016151377544);
			\draw [line width=0.8 pt] (0.,0.) circle (3.4641);
			\draw [line width=0.2 pt,dash pattern=on 1pt off 1pt] (0.,0.)-- (-2.,-3.4641016151377544);
			\draw [line width=0.2,dash pattern=on 1pt off 1pt] (0.,0.)-- (2.,-3.4641016151377544);
			\draw [line width=0.2,dash pattern=on 1pt off 1pt] (0.,0.)-- (2.,3.4641016151377544);
			\draw [line width=0.2,dash pattern=on 1pt off 1pt] (0.,0.)-- (-4.,0.);
			\draw [line width=0.2,dash pattern=on 1pt off 1pt] (0.,0.)-- (4.,0.);
			\draw [line width=0.2,dash pattern=on 1pt off 1pt] (0.,-0.02)-- (-2.,3.4441016151377544);
			\begin{scriptsize}
				\draw [fill=ffqqqq] (-2.,-3.4641016151377544) circle (1.5pt);
				\draw [fill=ffqqqq] (2.,-3.4641016151377544) circle (1.5pt);
				\draw [fill=ffqqqq] (4.,0.) circle (1.5pt);
				\draw [fill=ffqqqq] (2.,3.4641016151377544) circle (1.5pt);
				\draw [fill=ffqqqq] (-2.,3.4641016151377544) circle (1.5pt);
				\draw [fill=ffqqqq] (-4.,0.) circle (1.5pt);
				\draw [fill=ffqqqq] (3., -1.73205) circle (1.5pt);
				\draw [fill=ffqqqq] (3., 1.73205) circle (1.5pt);
				\draw [fill=ffqqqq] (0.,3.4641) circle (1.5pt);
				\draw [fill=ffqqqq] (-3., 1.73205) circle (1.5pt);
				\draw [fill=ffqqqq] (-3., -1.73205) circle (1.5pt);
				\draw [fill=ffqqqq] (-0.650055, -3.40256) circle (1.5pt);
				\draw [fill=ffqqqq] (0.650055, -3.40256) circle (1.5pt);
			\end{scriptsize}
		\end{tikzpicture}
		\caption{In a regular hexagon, the conditional constrained quantization with the constraint as the incircle, the dots represent the elements in an optimal set of $n$-points for $6\leq n\leq 13$.} \label{Fig5}
	\end{figure}

	\begin{exam} \label{exam111}  
		For a uniform distribution on a regular hexagon, i.e., when $k=6$, to obtain the optimal set $\ga_{13}$ and the corresponding conditional constrained quantization error $V_{13}$,  we proceed as follows (see Figure~\ref{Fig5}): By Lemma~\ref{lemmaMe113}, we can assume that $\te{card}(\ga_{13}\ii (S_1\uu L_1))=4$ and $\te{card}(\ga_n\ii (S_i\uu L_i))=3$ for $2\leq i\leq 6$. Then, using Methodology~\ref{method11}, we obtain 
		\[\ga_{13}\ii (S_1\uu L_1)=\set{(-\frac 12, -\frac{\sqrt 3}{2}), (-0.162514, -0.85064), (0.162514, -0.85064), (\frac 12, -\frac{\sqrt 3}{2})}\]  and  $V(P; \ga_{13}\ii(S_1\uu L_1))=0.00157076.$
		Moreover, we have 
		\[\ga_{13}\ii (S_i\uu L_i)=T^i(\set{(-\frac 12, -\frac{\sqrt 3}{2}), (0, -\frac{\sqrt 3}2),(\frac 12, -\frac{\sqrt 3}{2})})\]  with $V(\ga_{13}\ii (S_i\uu L_i))=\frac 1{288}  \te{ for } 2\leq j\leq 6.$
		Hence, 
		\begin{align*} \ga_{13}&=\UU_{i=1}^6\ga_{13}\ii (S_i\uu L_i)=\gb\UU\set{(-0.162514, -0.85064), (0.162514, -0.85064)}\UU\UU_{i=2}^6 \set{T^i(0, -\frac {\sqrt 3}{2})} \te{ with } \\
			V_{13}&=V(P; \ga_{13}\ii (S_1\uu L_1))+\sum_{i=2}^6 V(P; \ga_{13}\ii (S_i\uu L_i))=0.00157076+\frac 5{288}=0.0189319. 
		\end{align*} 
	\end{exam}

	\section{Conditional constrained quantization for a uniform distribution $P$ on a regular polygon with $k$ sides when the constraint is a diagonal of the polygon} \label{sec5}  
	For definiteness sake in this section, we take $k=6$, i.e., we take $P$ as a uniform distribution on a regular hexagon.   
	Notice that the hexagon has nine diagonals: six of the diagonals are of equal lengths, and we call them the diagonals of the smaller length; the remaining three diagonals are of equal lengths, and we call
	them the diagonals of the larger length. Let $diag$ represent any of the nine diagonals of the hexagon. By $\te{Int}(diag)$, it is meant the interior of the diagonal $diag$. $\gb$ is the conditional set. Then, notice that $\gb\ii diag\neq \es$ with $\te{card}(\gb\ii diag)=2$, and $\gb\ii \te{Int}(diag)=\es$. Hence, we can assume that in a conditional constrained optimal set $\ga_n$ of $n$-points with the constraint a diagonal $diag$, the set $\ga_n\setminus \gb$ will contain $n-6$ elements from the $\te{Int}(diag)$, i.e., $\te{card}(\ga_n\ii \te{Int}(diag))=n-6$.
	
	\begin{defi} \label{defiMe}
		Let $\ga$ be a discrete set. Then, for a Borel probability measure $\mu$ and a set $A$, by $V(\mu; \ga, A)$, it is meant the distortion error for $\mu$ with respect to the set $\ga$ over the set $A$, i.e., 
		%\begin{equation}
		%			V(\mu; \set{\ga, A}):= \int_A \mathop{\min}\limits_{a\in\ga} \rho(x, a) \,d\mu(x).
		%		\end{equation}
	\begin{equation}
		V(\mu; \ga, A):= \sum_{a\in \ga} \int_{A\ii M(a|\ga)} \mathop{\min}\limits_{a\in\ga} \rho(x, a) \,d\mu(x).
	\end{equation}
\end{defi}
In the following two subsections, we give the main results of this section.
\subsection{When the constraint is a diagonal of the smaller length} \label{sub2}
As can be seen from Figure~\ref{Fig1}, there
are six diagonals of the smaller length. Without any loss of generality, we can choose the diagonal $A_6A_4$
of the smaller length as the constraint. The conditional constrained optimal set of six-points is the conditional set $\gb$.

\begin{lemma}\label{Me421}
	Let $\ga_n$ be a conditional constrained optimal set of $n$-points for any $n\geq 7$. Then, $\ga_n\ii \te{Int}(A_6A_4)$ does not contain any element from $L_1\uu L_2\uu L_3\uu L_6$. 
\end{lemma} 
\begin{proof}
	Let $\ga_n$ be a conditional constrained optimal set of $n$-points for $n \geq 7$. Then, $\te{card}(\ga_n\ii \te{Int}(A_6A_4))\geq 1$. Since  $L_3$ is perpendicular to $A_6A_4$, for any $(x, y) \in L_3$, we have
	\[\min_{a\in \ga_n} \rho((x, y), a)= \rho((x, y), (\frac 12, \frac {\sqrt 3}2)),\]
	i.e., the Voronoi region of any element in $\ga_n\ii \te{Int}(A_6A_4)$ does not contain any element from $L_3$. Due to symmetry,
	the Voronoi region of any element in $\ga_n\ii \te{Int}(A_6A_4)$ also does not contain any element from $L_6$.  Again, notice that for any $(x, y)\in L_1\uu L_2$, we have 
	\[\min_{a\in \ga_n\ii A_6A_4} \rho((x, y), a)>\min_{a\in \gb} \rho((x, y), a),\]
	i.e., the Voronoi region of any element in $\ga_n\ii \te{Int}(A_6A_4)$ does not contain any element from $L_1\uu L_2$. 
	Thus, the proof of the lemma is complete. 
\end{proof} 

\begin{lemma} \label{Me422}
	For any $n\geq 6$, a conditional constrained optimal set of $n$-points contains exactly $n$-elements. 
\end{lemma} 
\begin{proof}
	$\gb$ is a conditional constrained optimal set of six-points that contains exactly six elements. Let $n \geq 7$. For the sake of contradiction, assume that there is an optimal set $\ga_n$ that contains fewer than $n$ elements. 
	The equation of the diagonal $A_6A_4$ is $y=\frac 1 {\sqrt 3} (x+1)$. The side $L_5$ is
	given by
	\[L_5=\{(\frac{1}{2} (-t-1),-\frac{1}{2} \sqrt{3} (t-1)): 0\leq t\leq 1\}.\]
	Let us consider an element $(\frac{1}{2} (-t-1),-\frac{1}{2} \sqrt{3} (t-1))$ on $L_5$. Its distance from the line  $y=\frac 1 {\sqrt 3} (x+1)$ is given by $d_1:=\frac{ 1-t} 2$. On the other hand, the distance of the element $(\frac{1}{2} (-t-1),-\frac{1}{2} \sqrt{3} (t-1))$ from $(-1, 0)$ is $d_2:=1-t$, and from $(-\frac 12, \frac{\sqrt 3}{2})$ is $d_3:=t$. Notice that always $d_1<d_2$ for $0< t\leq 1$. On the other hand, $d_1< d_3$ for $\frac 13<t\leq 1$. Hence, the possible elements in $L_5$ that can fall within the Voronoi regions of the elements in $\ga_n\ii \te{Int}(A_6A_4)$ are for $\frac 13< t<1$. Similarly, the possible elements in $L_4$ that can fall within the Voronoi regions of the elements in $\ga_n\ii \te{Int}(A_6A_4)$ are for $0< t< \frac 23$. Notice that $t=\frac 13$ corresponds to the point $(-\frac{2}{3},\frac{1}{\sqrt{3}})$, and $t=1$ corresponds to the point $(-1, 0)$ on the side $L_5$. If $(u, v)$ be the foot of the perpendicular from the point $(-\frac{2}{3},\frac{1}{\sqrt{3}})$ on the diagonal $A_6A_4$, then we have
	\[(u, v)=(-\frac 12, \frac{1}{2 \sqrt{3}}).\]
	Thus, we can conclude that the elements of $\ga_n\ii \te{Int}(A_6A_4)$, the Voronoi regions of which cover the elements from $L_5$ will lie on the line $y=\frac 1 {\sqrt 3} (x+1)$ for $-1<x<-\frac 12$. Similarly, we can conclude that the elements of $\ga_n\ii \te{Int}(A_6A_4)$, the Voronoi regions of which cover the elements from $L_4$ will lie on the line $y=\frac 1 {\sqrt 3} (x+1)$ for $0<x<\frac{1}{2}$. Hence, if the optimal set $\ga_n$ contains fewer than $n$ elements, then by adding at least one element in $\ga_n$ which comes from $y=\frac 1 {\sqrt 3} (x+1)$ either for $-1<x<-\frac 12$, or for $0<x<\frac{1}{2}$, we can strictly reduce the quantization error, which contradicts the fact that $\ga_n$ is a conditional constrained optimal set of $n$-points. Thus, we conclude that a conditional constrained optimal set of $n$-points contains exactly $n$-elements, which completes the proof of the lemma. 
\end{proof} 
\begin{remark}
	Recall that $\ga_n$ is a conditional constrained optimal set of $n$-points for $n\geq 6$. Then, 
	\[\ga_n\ii A_6A_4=(\ga_n\ii \te{Int}(A_6A_4))\uu \set{(-1, 0), (\frac 12, \frac{\sqrt 3}{2})},\]
	As mentioned in the proof of Lemma~\ref{Me422}, we have 
	\begin{equation} \label{eq1111} \ga_n\ii A_6A_4 \ci \set{(x, y) \in A_6A_4 : -1\leq x\leq -\frac 12}\uu\set {(x, y)\in A_6A_4 : 0\leq x\leq \frac 12}.
	\end{equation} 
\end{remark} 
\begin{remark} 
	By the expression \eqref{eq1111}, we can conclude that for any $a\in \ga_n\ii A_6A_4$, if the Voronoi region of the element $a$ includes elements from $L_5$, then it will not include any element from $L_4$. We denote such elements in $\ga_n\ii A_6A_4$ by $(\ga_n\ii A_6A_4)^{(1)}$. Similarly, for any $a\in \ga_n\ii A_6A_4$, if the Voronoi region of the element $a$ includes elements from $L_4$, then it will not include any element from $L_5$. We denote such elements in $\ga_n\ii A_6A_4$ by $(\ga_n\ii A_6A_4)^{(2)}$.
	Notice that 
	\[\ga_n\ii A_6A_4=(\ga_n\ii A_6A_4)^{(1)}\UU (\ga_n\ii A_6A_4)^{(2)}, \te{ and } (\ga_n\ii A_6A_4)^{(1)}\II (\ga_n\ii A_6A_4)^{(2)}=\es.  \]
\end{remark}
Let us now state the following lemma. The proof is similar to Lemma~\ref{lemmaMe1}, in fact, $P$ being a uniform distribution, and the two sides $L_4$ and $L_5$ are symmetrically located with respect to the perpendicular bisector of the segment $A_6A_4$, it is not difficult to show the proof. 

\begin{lemma} \label{Megha1} 
	For $n\geq 7$, let $(\ga_n\ii A_6A_4)^{(1)}$ and $(\ga_n\ii A_6A_4)^{(2)}$ be two sets as defined before. Let $\te{card}((\ga_n\ii A_6A_4)^{(1)}))=n_1$ and $\te{card}((\ga_n\ii A_6A_4)^{(2)}))=n_2$. Then, $n_1, n_2\geq 1 $ with $|n_1-n_2|\in \set{0,1}$. 
\end{lemma} 
\subsection*{Isometry} Consider the isometry $U: \D R^2 \to \D R^2$ such that 
\[U(x, y)=(\frac{\sqrt{3}} 2x+\frac{y}{2}, -\frac{x}{2}+\frac{\sqrt{3}} 2y-\frac 12).\] Since $U$ is a bijective function, its inverse transformation is given by 
\[U^{-1}(x, y) : \D R^2 \to \D R^2 \te{ such that } U^{-1}(x, y)=(\frac{\sqrt{3}} 2 x-\frac{y}{2}-\frac 14, \frac x2  +\frac{\sqrt{3}} 2 y+ \frac {\sqrt 3}4). \] 
Let us now consider a triangle $\tri GHI$ with vertices $G(-\frac{\sqrt{3}}{2},0)$, $H(\frac{\sqrt{3}}{2},0)$, and $I(0, \frac 12)$. 
Notice that $U(-1, 0)=(-\frac{\sqrt{3}}{2},0), \, U(\frac{1}{2},\frac{\sqrt{3}}{2})=(\frac{\sqrt{3}}{2},0)$, and $U(-\frac 12, \frac {\sqrt 3}2)=(0,\frac{1}{2})$. Hence, we have 
\begin{equation} \label{eq54} U(\tri A_6A_4A_5)=\tri GHI \te{ and } U^{-1}(\tri GHI)=\tri A_6A_4A_5.\end{equation}

\begin{note} \label{note1} 
	To calculate the conditional constrained optimal sets of $n$-points for any $n\geq 7$, for the uniform distribution $P$, we will compare $\tri A_6A_4A_5$ with $\tri GHI$ as given by \eqref{eq54}.   Consequently,  $(\ga_n\ii A_6A_4 )^{(1)}$ and $(\ga_n\ii A_6A_4 )^{(2)}$ will be compared, respectively, with  $(U(\ga_n)\ii GH)^{(1)}$ and $(U(\ga_n)\ii GH)^{(2)}$. Thus, by Lemma~\ref{Megha1}, we see that if $\te{card}((U(\ga_n)\ii GH)^{(1)})=n_1$ and $\te{card}((U(\ga_n)\ii GH)^{(2)})=n_2$, then $n_1, n_2\geq 1 $ with $|n_1-n_2|\in \set{0,1}$. 
	Moreover, notice that after the isometry, the expression given by \eqref{eq1111}, reduces to 
	\begin{equation} \label{eq1112} U(\ga_n)\ii GH \ci \set{(x, y) \in GH : -\frac{\sqrt{3}}{2}\leq x\leq -\frac{1}{2 \sqrt{3}}}\uu\set {(x, y)\in GH : \frac{1}{2 \sqrt{3}}\leq x\leq \frac{\sqrt{3}}{2}}.
	\end{equation}
	From the expression~\eqref{eq1112}, we see that the Voronoi region of any element in $\set{(x, y) \in GH : -\frac{\sqrt{3}}{2}\leq x\leq -\frac{1}{2 \sqrt{3}}}$ does not contain any element from the Voronoi region of any element in  
	$\set {(x, y)\in GH : \frac{1}{2 \sqrt{3}}\leq x\leq \frac{\sqrt{3}}{2}}$, and vice versa. 
	Let $Q$ be the normalized image measure of $P$ under the isometry $U$, i.e., $Q=3 U(P)$ such that for any Borel subset $A\ci \D R^2$, we have $Q(A)=3 P(U^{-1}(A))$. Thus, $Q$ can be considered as a Borel probability measure on $\D R^2$ which is uniform on its support $GI\uu HI$. The parametric representations of the sides $GI$, $HI$ and $GH$ are given as follows:
	\begin{align*} GI&=\set{(-\frac{\sqrt{3}t }{2},\frac{1-t}{2}) : 0\leq t\leq 1},\\
		HI&=\set{(\frac{\sqrt{3} (1-t)}{2},\frac{t}{2}) : 0\leq t\leq 1}, \te{ and }\\
		GH&=\set{(t, 0) : -\frac{\sqrt 3} 2\leq t\leq \frac{\sqrt 3} 2}.
	\end{align*} 
	Let $s$ represent the distance of any point on $\tri GHI$ starting from the vertex $G$ tracing along the boundary of the triangle in the counterclockwise direction. Then, the points $G, H, I$ are, respectively, represented by $s=0, s=1, s=2$. 
	The probability density function (pdf) $g$ of the uniform distribution $Q$ is given by $g(s):=g(x, y)=\frac 1{2}$ for all $(x, y)\in GI\uu HI$, and zero otherwise. 
	Again, $dQ(s)=Q(ds)=g(x_1, x_2) ds$. Notice that on each of $GI$ and $HI$, we have $ds=\sqrt{(dx)^2+(dy)^2}=dt$.
	
	For $n\geq 6$, to obtain a conditional constrained optimal set $\ga_n$ of $n$-points for the uniform distribution $P$ under the conditional set $\gb$ and the constraint $A_6A_4 $, we will first calculate a conditional constrained optimal set $\gg_{n-3}$ of $(n-3)$-points for the uniform distribution $Q$ under the conditional set $\hat\gb:=\set{(-\frac {\sqrt 3}2, 0), (\frac {\sqrt 3}2, 0), (0,\frac 12)}$ and the constraint $GH$, and then will obtain them for the uniform distribution $P$ as described in Methodology~\ref{method1} given below.
\end{note}

Denote the sets $\set{(x, y) \in GH : -\frac{\sqrt{3}}{2}\leq x\leq -\frac{1}{2 \sqrt{3}}}$ and $\set{(x, y) \in GH : \frac{\sqrt{3}}{2}\leq x\leq \frac{1}{2 \sqrt{3}}}$, metioned in Note~\ref{note1}, respectively, by $GH^{(1)}$ and $GH^{(2)}$.

Let us now give the following proposition. 
\begin{lemma} \label{lemma78} 
	For $n\geq 6$, let $\gg_{n-3}$ be a conditional constrained optimal set of $(n-3)$-points for the uniform distribution $Q$ under the conditional set $\hat \gb$ and the constraint $GH$. Then, the conditional constrained optimal set $\ga_n$ of $n$-points for the uniform distribution $P$ under the conditional set $\gb$ and the constraint $A_6A_4 $ is given by
	\[\ga_n=U^{-1}(\gg_{n-3})\uu  \gb.\] 
	The conditional constrained quantization error is given by 
	\[V_n(P)=\frac 1 3 V_{n-3}(Q)+\frac 1 {18}.\] 
\end{lemma} 
\begin{proof} The proof that 
	$\ga_n=U^{-1}(\gg_{n-3})\uu  \gb$ follows from the definition of the isometry. We now prove the expression for quantization error. Let $\gg_{n-3}$ be a conditional constrained optimal set of $(n-3)$-points for the uniform distribution $Q$. Then, 
	\begin{align*}
		V_{n-3}(Q)&=\int_{GI\uu HI} \min_{a\in \gg_{n-3}}\rho((x, y), a)\,dQ(x)=3\int_{GI\uu HI} \min_{a\in \gg_{n-3}}\rho((x, y), a)\,d(UP)(x)\\
		&=\int_{GI\uu HI} \min_{a\in \gg_{n-3}}\rho((x, y), a)\,d(P\circ U^{-1})(x)=3\int_{U^{-1}(GI\uu HI)} \min_{a\in \gg_{n-3}}\rho(U(x, y), a)\,dP(x)\\
		&=3\int_{L_4\uu L_5} \min_{a\in U^{-1}(\gg_{n-3})}\rho(U(x, y), U(a))\,dP(x)=3\int_{L_4\uu L_5} \min_{a\in U^{-1}(\gg_{n-3})}\rho((x, y), a)\,dP(x)\\
		&=3 V(P; \set{U^{-1}(\gg_{n-3}), L_4\uu L_5}).  
	\end{align*} 
	Due to symmetry, we have 
	\begin{align*} &V(P; \set{\gb, L_1\uu L_2\uu L_3\uu L_6})=4 \int_{L_1}\min_{a\in \gb} \rho((x, y), a)\,dP\\
		&=\frac 8 6\int_{0}^{\frac 12} \rho((t-\frac{1}{2}, -\frac{\sqrt{3}}{2}), (-\frac{1}{2},-\frac{\sqrt{3}}{2}))\,dt=\frac{1}{18}.
	\end{align*} 
	Thus, we obtain
	\[V_n(P)=V(P; \set{U^{-1}(\gg_{n-3}), L_4\uu L_5})+V(P; \set{\gb, L_1\uu L_2\uu L_3\uu L_6})=\frac 13 V_{n-3}(Q)+\frac 1{18}.\]
	Thus, the lemma is yielded. 
\end{proof} 

\begin{remark} \label{remark22}
	Due to \eqref{eq1112}, we see that if $\gg_n$ is a conditional constrained optimal set of $n$-points for the uniform distribution $Q$ under the constraint $GI\uu HI$ and the conditional set $\hat\gb$, and $n=2m+1$ for some $m\in \D N$, then $V(Q, \set{\gg_n, GI})=V(Q, \set{\gg_n, HI})$ and hence, 
	\[V_n(Q)=2V(Q, \set{\gg_n, GI}).\]
\end{remark}  
\begin{lemma} \label{lemma79} 
	The conditional constrained optimal set of four-points for $Q$ is given by $\gg_4=\hat\gb\uu \set{(-\frac 12, 0)}$ with conditional constrained quantization error
	$V_4(Q)= \frac{1}{24} \left(5-2 \sqrt{3}\right)$.  
\end{lemma} 

\begin{proof}
	Let $\gg_4:=\hat \gb\uu \set{(a,0)}$ be a conditional constrained optimal set of four-points for some $-\frac{\sqrt 3} 2< a<-\frac{1}{2 \sqrt{3}}$. Let the boundary of the Voronoi regions of $(-\frac{\sqrt 3} 2, 0)$ and $(a, 0)$ intersect $GH$ at a point given by the parameter $t=d_1$, and let the boundary of the Voronoi regions of $(a, 0)$ and  $I(0, \frac 12)$ intersect $GI$ at a point given by the parameter $t=d_2$. Then, notice that $0<d_2<d_1<1$. Solving the canonical equations
	\begin{align*}
		\rho((-\frac{\sqrt 3 d_1}{2}, \frac{1-d_1}{2}), (-\frac{\sqrt 3} 2, 0))&=\rho((-\frac{\sqrt 3 d_1}{2}, \frac{1-d_1}{2}), (a, 0)) \te{ and }\\
		\rho((-\frac{\sqrt 3 d_2}{2}, \frac{1-d_2}{2}), (a, 0))&=\rho((-\frac{\sqrt 3 d_2}{2}, \frac{1-d_2}{2}), (0, \frac 12)
	\end{align*}
	we have 
	\[d_1=-\frac{-\frac{\sqrt{3}}{2}+a}{\sqrt{3}} \te{ and } d_2=\frac{4a^2+1}{2-4 \sqrt{3} a}.\]
	Now, calculate the distortion error 
	\begin{align*}
		V(Q; \set{\gg_4, GI})&=  \frac 12\Big(\int_{d_1}^1 \rho((-\frac{\sqrt 3 t}2, \frac{1-t}2), (-\frac {\sqrt 3} 2, 0))\,dt+ \int_{d_2}^{d_1} \rho((-\frac{\sqrt 3 t}2, \frac{1-t}2), (a, 0))\,dt\\
		&\qquad + \int_{0}^{d_2} \rho((-\frac{\sqrt 3 t}2, \frac{1-t}2), (0, \frac 12))\,dt+\Big)\\
		&=\frac{96 \sqrt{3} a^5+48 a^4-16 \sqrt{3} a^3-24 a^2+6 \sqrt{3} a-1}{24 \left(2 \sqrt{3} a-1\right)^3}.
	\end{align*}
	which is minimum if $a=-\frac 12$ and then $V(Q; \set{\gg_4, GI})=\frac{1}{12} \left(2-\sqrt{3}\right)$. Hence, $\gg_4=\hat \gb\uu \set{(-\frac 12,0)}$, and $V_4(Q)$ is given by 
	\[V_4(Q)=V(Q; \set{\gg_4, GI})+V(Q; \set{\gg_4, HI})=V(Q; \set{\gg_4, GI})+\frac{1}{24}= \frac{1}{24} \left(5-2 \sqrt{3}\right)=0.0639958.\]
	Thus, the lemma is yielded. 
\end{proof} 

\begin{remark} \label{remark23}
	Due to Remark~\ref{remark22}, the conditional constrained optimal set of five-points for $Q$ is given by $\gg_5=\hat\gb\uu \set{(-\frac 12, 0), (\frac 12, 0)}$ with conditional constrained quantization error
	$V_5(Q)= 2 V(Q; \set{\gg_5, GI})=\frac{1}{6} \left(2-\sqrt{3}\right).$
\end{remark}

\begin{prop}
	Let $\ga_n$ be a conditional constrained optimal sets of $n$-points for $P$ with quantization error $V_n$. Then, 
	\begin{align*} \ga_7&= \gb\uu \Big\{\Big(\frac{1}{4}(-\sqrt{3}-1),\frac{1}{4} (\sqrt{3}-1)\Big)\Big\} \te{ with } V_7= \frac{1}{72} \left(9-2 \sqrt{3}\right)  \\
		\ga_8&=\gb\uu \Big\{\Big(\frac{1}{4}(-\sqrt{3}-1),\frac{1}{4} (\sqrt{3}-1)\Big),\Big(\frac{1}{4} (\sqrt{3}-1),\frac{1}{4} (\sqrt{3}+1)\Big)\Big\} \te{ with } V_8=\frac{1}{18} (3-\sqrt{3}).
	\end{align*} 
\end{prop}

\begin{proof}
	The proof follows by Lemma~\ref{lemma78}, Lemma~\ref{lemma79}, and Remark~\ref{remark23} (see Figure~\ref{Fig6}). 
\end{proof} 

We now give a methodology to obtain the conditional constrained optimal sets of $n$-points for the uniform distribution $Q$ for any $n\geq 6$. 
\begin{method} \label{method1} 
	For $n\geq 6$, let $\gg_n$ be a conditional constrained optimal set of $n$-points such that $\te{card}(\gg_n\ii GH^{(1)})=n_1$ and $\te{card}(\gg_n\ii GH^{(2)})=n_2$. Then, we know that $|n_1-n_2|\in \set{0,1}$. Also, $n=n_1+n_2+1$. If $n=6$, without any loss of generality, we can assume that $n_1=3$ and $n_2=2$. If $n\geq 7$, then always $n_1, n_2\geq 3$.   To calculate $\gg_n\ii GH^{(1)}$ we proceed as follows. Let 
	\begin{align*}
		\gg_n\ii GH^{(1)}&=\set{(a(j), 0): 1\leq j\leq n_1} \te{ with } a(i)<a(j)  \te{ for }  1\leq i<j\leq n_1, 
	\end{align*} 
	where $a(1)=-\frac{\sqrt 3} 2$. 
	For $1\leq j<n_1-1$, let the boundary of the Voronoi regions of $(a(j), 0)$ and $(a(j+1), 0)$ intersect the side $GI$ at a point represented by the parameter $t=d(j)$. For $1\leq j\leq n_1-1$ solving the canonical equations 
	\[\rho((-\frac{\sqrt 3d(j)}{2}, \frac{1-d(j)} 2), (a(j),0))=\rho((-\frac{\sqrt 3d(j)}{2}, \frac{1-d(j)} 2), (a(j+1),0)),\]
	we have $d(j)=-\frac{a(j)+a(j+1)}{\sqrt{3}}$. Let the boundary of the Voronoi regions of $(a(n_1), 0)$ and $(0, \frac 12)$ intersect the side $GI$ at a point represented by the parameter $t=d(n_1)$. Then, solving the canonical equation  
	\[\rho((-\frac{\sqrt 3d(n_1)}{2}, \frac{1-d(n_1)} 2), (a(n_1),0))=\rho((-\frac{\sqrt 3d(n_1)}{2}, \frac{1-d(n_1)} 2), (0, \frac 12)),\]
	we have 
	$d(n_1)= \frac{4 a(n_1)^2+1}{2-4 \sqrt{3} a(n_1)}$. We now calculate the distortion error $V(Q; \set{\gg_n, GI})$ as follows: 
	\begin{align} \label{Me82} 
		&V(Q; \set{\gg_n, GI})=\frac 12\Big(\int_{d(1)}^1\rho((-\frac{\sqrt 3 t}2, \frac{1-t}2), (a(1),0))\,dt \notag\\
		&=\sum_{j=2}^{n_1-1}\int_{d(j)}^{d(j-1)}\rho((-\frac{\sqrt 3 t}2, \frac{1-t}2), (a(j),0)) \,dt+\int_{d(n_1)}^{d(n_1-1)}\rho((-\frac{\sqrt 3 t}2, \frac{1-t}2), (a(n_1),0)) \,dt \notag\\
		&\qquad +\int_{0}^{d(n_1)}\rho((-\frac{\sqrt 3 t}2, \frac{1-t}2), (0, \frac 12))\, dt\Big).
	\end{align}
	Recall that $a(1)=-\frac{\sqrt 3} 2$.  
	After putting the corresponding values of $d(j)$ for $1\leq j\leq n_1$, and then upon simplification, we obtain $V(Q; \set{\gg_n, GI})$ as a function of $a(j)$ for $2\leq j\leq n_1$. Then, after solving the partial derivatives 
	\begin{equation} \label{Me831} \frac{\pa}{\pa a(j)}(V(Q; \set{\gg_n, GI})=0\end{equation} 
	we obtain the values of $a(j)$ for $2\leq j\leq n_1$ for which $V(Q; \set{\gg_n, GI})$ is minimum. Thus, $\gg_n\ii GH^{(1)}$ and the corresponding distortion error $V(Q; \set{\gg_n, GI})$ are obtained for $n_1\geq 3$. If $n_2=2$, then by Lemma~\ref{lemma79}, we have $\gg_n\ii GH^{(2)}=\set{(\frac 12, 0), (\frac{\sqrt 3} 2, 0)}$. If $n_2\geq 3$, we replace $n_1$ by $n_2$ in the above procedure, and obtain the set $\gg_n\ii GH^{(1)}$ for $\te{card}(\gg_n\ii GH^{(1)})=n_2$, and then reflect it with respect to the $y$-axis to obtain the set $\gg_n\ii GH^{(2)}$ with $\te{card}(\gg_n\ii GH^{(2)})=n_2$. Likewise, we obtain the distortion error $V(Q; \set{\gg_n, HI})$ as  $V(Q; \set{\gg_n, HI})=V(Q; \set{\gg_n, GI})$ with $\te{card}(\gg_n\ii GH^{(1)})=\te{card}(\gg_n\ii GH^{(2)})=n_2$.
\end{method} 

Using Methodology~\ref{method1}, we obtain the following examples. 
\begin{exam} \label{Me83}
	To obtain $\gg_6(Q)$ and the corresponding conditional constrained quantization error $V_6(Q)$ for the uniform distribution $Q$, we proceed as follows. Let $n=6$ and choose $n_1=3$ and $n_2=2$. Putting $n_1=3$ in Expression~\eqref{Me82} and then upon simplification, we have 
	\begin{align*}
		V(Q; \set{\gg_6, GI})&=\frac{1}{48 (2 \sqrt{3} a(3)-1)^3}\Big(-2+15 \sqrt{3} a(3)-96 a(3)^2+40 \sqrt{3} a(3)^3+96 a(3)^4+48 \sqrt{3} a(3)^5\\
		&+a(2) \Big(4 a(3)^2-3\Big)\Big(-72 a(3)^3+36 \sqrt{3} a(3)^2-18 a(3)+\sqrt{3}\Big)\\
		&+2 a(2)^2 \Big(144 a(3)^4-72 a(3)^2+16 \sqrt{3} a(3)-3\Big)\Big).
	\end{align*}
	Now, solving the partial derivatives $\frac{\pa}{\pa a(j)}(V(Q; \set{\gg_n, GI})=0$ for $j=2, 3$, we have 
	$a(2)=-0.639488$ and $a(3)= -0.41295$. 
	Thus, we have 
	\begin{equation} \label{Me91}
		\gg_6\ii GH^{(1)}=\set{(-\frac{\sqrt 3}2, 0), (-0.639488, 0),  ( -0.41295, 0)} \te{ with } V(Q; \set{\gg_6, GI})=0.0199575.
	\end{equation} 
	Moreover, as $n_2=2$, using the properties of reflection by Lemma~\ref{lemma79} as mentioned in Methodology~\ref{method1}, we have 
	\[\gg_6\ii GH^{(2)}=\set{(\frac 12, 0), (\frac{\sqrt 3}2, 0)}\te{ with } V(Q; \set{\gg_6, HI})=\frac{1}{12} (2-\sqrt{3}).\]
	Thus, we have 
	\[\gg_6=\hat \gb\UU \set{(-0.639488, 0),  ( -0.41295, 0), (\frac 12, 0)},\]
	with $V_6(Q)=V(Q; \set{\gg_6, GI})+V(Q; \set{\gg_6, HI})=0.0199575+\frac{1}{12} (2-\sqrt{3})=0.0422866$. 
\end{exam} 

\begin{exam} \label{Me84}
	To obtain $\gg_7(Q)$ and the corresponding conditional constrained quantization error $V_7(Q)$ for the uniform distribution $Q$, we proceed as follows. Let $n=7$ and choose $n_1=3$ and $n_2=3$. Then, by \eqref{Me91}, we have 
	\begin{align*}
		\gg_7\ii GH^{(1)}&=\gg_6\ii GH^{(1)}=\set{(-\frac{\sqrt 3}2, 0), (-0.639488, 0),  ( -0.41295, 0)} \te{ with } \\
		V(Q; \set{\gg_7, GI})&= V(Q; \set{\gg_6, GI})=0.0199575.
	\end{align*} 
	Moreover, as $n_2=3$, using the properties of reflection as mentioned in Methodology~\ref{method1}, we have 
	\begin{align*}
		\gg_7\ii GH^{(2)}&=\set{( 0.41295, 0), (0.639488, 0), (\frac{\sqrt 3}2, 0)} \te{ with } \\
		V(Q; \set{\gg_7, HI})&= V(Q; \set{\gg_7, GI})=0.0199575.
	\end{align*}
	Thus, we have 
	\[\gg_7=\hat \gb \UU \set{(-0.639488, 0),  ( -0.41295, 0), ( 0.41295, 0), (0.639488, 0)}\]
	with $V_7(Q)=V(Q; \set{\gg_7, GI})+V(Q; \set{\gg_7, HI})=0.0199575+0.0199575=0.039915$. 
\end{exam}

\begin{figure}
	\vspace{-0.3 in} 
	\begin{tikzpicture}[line cap=round,line join=round,>=triangle 45,x=0.5 cm,y=0.5 cm]
		\clip(-7.5,-3.8) rectangle (5.5,5.5);
		\draw [line width=0.8 pt] (-2.,3.4641016151377544)-- (-4.,0.)-- (-2.,-3.4641016151377544)-- (2.,-3.4641016151377544)-- (4.,0.)-- (2.,3.4641016151377544);
		\draw [line width=0.8 pt] (-2.,3.4641016151377544)-- (2.,3.4641016151377544);
		\draw (1.8,4.66) node[anchor=north west] {$A_4(\frac 12, \frac{\sqrt{3}}{2})$};
		\draw (-6.4,4.66) node[anchor=north west] {$A_5(-\frac 12, \frac{\sqrt{3}}{2})$};
		\draw (-7.9,0.65) node[anchor=north west] {$A_6(-1, 0)$};
		\draw [line width=0.2 pt,dash pattern=on 1pt off 1pt] (0.,0.)-- (-2.,-3.4641016151377544);
		\draw [line width=0.2,dash pattern=on 1pt off 1pt] (0.,0.)-- (2.,-3.4641016151377544);
		\draw [line width=0.2,dash pattern=on 1pt off 1pt] (0.,0.)-- (2.,3.4641016151377544);
		\draw [line width=0.2,dash pattern=on 1pt off 1pt] (0.,0.)-- (-4.,0.);
		\draw [line width=0.2,dash pattern=on 1pt off 1pt] (0.,0.)-- (4.,0.);
		\draw [line width=0.2,dash pattern=on 1pt off 1pt] (0.,-0.02)-- (-2.,3.4441016151377544);
		\draw [line width=0.2 pt,dash pattern=on 1pt off 1pt] (-2., -3.4641)-- (4, 0);
		\draw [line width=0.2 pt,dash pattern=on 1pt off 1pt] (2., -3.4641)-- (2.,3.4641016151377544);
		\draw [line width=0.8 pt] (-4, 0)-- (2.,3.4641016151377544);
		\draw [line width=0.2 pt,dash pattern=on 1pt off 1pt] (-4, 0)-- (2.,-3.4641016151377544);
		\draw [line width=0.2 pt,dash pattern=on 1pt off 1pt] (4, 0)-- (-2.,3.4641016151377544);
		\draw [line width=0.2 pt,dash pattern=on 1pt off 1pt] (-2.,3.4641016151377544)-- (-2.,-3.4641016151377544);
		\begin{scriptsize}
			\draw [fill=ffqqqq] (-2.,-3.4641016151377544) circle (1.5pt);
			\draw [fill=ffqqqq] (2.,-3.4641016151377544) circle (1.5pt);
			\draw [fill=ffqqqq] (4.,0.) circle (1.5 pt);
			\draw [fill=ffqqqq] (2.,3.4641016151377544) circle (1.5 pt);
			\draw [fill=ffqqqq] (-2.,3.4641016151377544) circle (1.5 pt);
			\draw [fill=ffqqqq] (-4.,0.) circle (1.5pt);
			\draw [fill=ffqqqq] (-2.73205, 0.732051) circle (1.5 pt);
		\end{scriptsize}
	\end{tikzpicture}
	\begin{tikzpicture}[line cap=round,line join=round,>=triangle 45,x=0.5 cm,y=0.5 cm]
		\clip(-7.5,-3.8) rectangle (5.5,5.5);
		\draw [line width=0.8 pt] (-2.,3.4641016151377544)-- (-4.,0.)-- (-2.,-3.4641016151377544)-- (2.,-3.4641016151377544)-- (4.,0.)-- (2.,3.4641016151377544);
		\draw [line width=0.8 pt] (-2.,3.4641016151377544)-- (2.,3.4641016151377544);
		\draw (1.8,4.66) node[anchor=north west] {$A_4(\frac 12, \frac{\sqrt{3}}{2})$};
		\draw (-6.4,4.66) node[anchor=north west] {$A_5(-\frac 12, \frac{\sqrt{3}}{2})$};
		\draw (-7.9,0.65) node[anchor=north west] {$A_6(-1, 0)$};
		\draw [line width=0.2 pt,dash pattern=on 1pt off 1pt] (0.,0.)-- (-2.,-3.4641016151377544);
		\draw [line width=0.2,dash pattern=on 1pt off 1pt] (0.,0.)-- (2.,-3.4641016151377544);
		\draw [line width=0.2,dash pattern=on 1pt off 1pt] (0.,0.)-- (2.,3.4641016151377544);
		\draw [line width=0.2,dash pattern=on 1pt off 1pt] (0.,0.)-- (-4.,0.);
		\draw [line width=0.2,dash pattern=on 1pt off 1pt] (0.,0.)-- (4.,0.);
		\draw [line width=0.2,dash pattern=on 1pt off 1pt] (0.,-0.02)-- (-2.,3.4441016151377544);
		\draw [line width=0.2 pt,dash pattern=on 1pt off 1pt] (-2., -3.4641)-- (4, 0);
		\draw [line width=0.2 pt,dash pattern=on 1pt off 1pt] (2., -3.4641)-- (2.,3.4641016151377544);
		\draw [line width=0.8 pt] (-4, 0)-- (2.,3.4641016151377544);
		\draw [line width=0.2 pt,dash pattern=on 1pt off 1pt] (-4, 0)-- (2.,-3.4641016151377544);
		\draw [line width=0.2 pt,dash pattern=on 1pt off 1pt] (4, 0)-- (-2.,3.4641016151377544);
		\draw [line width=0.2 pt,dash pattern=on 1pt off 1pt] (-2.,3.4641016151377544)-- (-2.,-3.4641016151377544);
		\begin{scriptsize}
			\draw [fill=ffqqqq] (-2.,-3.4641016151377544) circle (1.5pt);
			\draw [fill=ffqqqq] (2.,-3.4641016151377544) circle (1.5pt);
			\draw [fill=ffqqqq] (4.,0.) circle (1.5 pt);
			\draw [fill=ffqqqq] (2.,3.4641016151377544) circle (1.5 pt);
			\draw [fill=ffqqqq] (-2.,3.4641016151377544) circle (1.5 pt);
			\draw [fill=ffqqqq] (-4.,0.) circle (1.5pt);
			\draw [fill=ffqqqq] (-2.73205, 0.732051) circle (1.5 pt);
			\draw [fill=ffqqqq] (0.732051, 2.73205) circle (1.5 pt);
		\end{scriptsize}
	\end{tikzpicture}\\
	\begin{tikzpicture}[line cap=round,line join=round,>=triangle 45,x=0.5 cm,y=0.5 cm]
		\clip(-7.5,-3.8) rectangle (5.5,5.5);
		\draw [line width=0.8 pt] (-2.,3.4641016151377544)-- (-4.,0.)-- (-2.,-3.4641016151377544)-- (2.,-3.4641016151377544)-- (4.,0.)-- (2.,3.4641016151377544);
		\draw [line width=0.8 pt] (-2.,3.4641016151377544)-- (2.,3.4641016151377544);
		\draw (1.8,4.66) node[anchor=north west] {$A_4(\frac 12, \frac{\sqrt{3}}{2})$};
		\draw (-6.4,4.66) node[anchor=north west] {$A_5(-\frac 12, \frac{\sqrt{3}}{2})$};
		\draw (-7.9,0.65) node[anchor=north west] {$A_6(-1, 0)$};
		\draw [line width=0.2 pt,dash pattern=on 1pt off 1pt] (0.,0.)-- (-2.,-3.4641016151377544);
		\draw [line width=0.2,dash pattern=on 1pt off 1pt] (0.,0.)-- (2.,-3.4641016151377544);
		\draw [line width=0.2,dash pattern=on 1pt off 1pt] (0.,0.)-- (2.,3.4641016151377544);
		\draw [line width=0.2,dash pattern=on 1pt off 1pt] (0.,0.)-- (-4.,0.);
		\draw [line width=0.2,dash pattern=on 1pt off 1pt] (0.,0.)-- (4.,0.);
		\draw [line width=0.2,dash pattern=on 1pt off 1pt] (0.,-0.02)-- (-2.,3.4441016151377544);
		\draw [line width=0.2 pt,dash pattern=on 1pt off 1pt] (-2., -3.4641)-- (4, 0);
		\draw [line width=0.2 pt,dash pattern=on 1pt off 1pt] (2., -3.4641)-- (2.,3.4641016151377544);
		\draw [line width=0.8 pt] (-4, 0)-- (2.,3.4641016151377544);
		\draw [line width=0.2 pt,dash pattern=on 1pt off 1pt] (-4, 0)-- (2.,-3.4641016151377544);
		\draw [line width=0.2 pt,dash pattern=on 1pt off 1pt] (4, 0)-- (-2.,3.4641016151377544);
		\draw [line width=0.2 pt,dash pattern=on 1pt off 1pt] (-2.,3.4641016151377544)-- (-2.,-3.4641016151377544);
		\begin{scriptsize}
			\draw [fill=ffqqqq] (-2.,-3.4641016151377544) circle (1.5pt);
			\draw [fill=ffqqqq] (2.,-3.4641016151377544) circle (1.5pt);
			\draw [fill=ffqqqq] (4.,0.) circle (1.5 pt);
			\draw [fill=ffqqqq] (2.,3.4641016151377544) circle (1.5 pt);
			\draw [fill=ffqqqq] (-2.,3.4641016151377544) circle (1.5 pt);
			\draw [fill=ffqqqq] (-4.,0.) circle (1.5pt);
			\draw [fill=ffqqqq] (-3.21525, 0.453075) circle (1.5 pt);
			\draw [fill=ffqqqq] (-2.4305, 0.906151) circle (1.5 pt);
			\draw [fill=ffqqqq] (0.732051, 2.73205) circle (1.5 pt);
		\end{scriptsize}
	\end{tikzpicture}
	\begin{tikzpicture}[line cap=round,line join=round,>=triangle 45,x=0.5 cm,y=0.5 cm]
		\clip(-7.5,-3.8) rectangle (5.5,5.5);
		\draw [line width=0.8 pt] (-2.,3.4641016151377544)-- (-4.,0.)-- (-2.,-3.4641016151377544)-- (2.,-3.4641016151377544)-- (4.,0.)-- (2.,3.4641016151377544);
		\draw [line width=0.8 pt] (-2.,3.4641016151377544)-- (2.,3.4641016151377544);
		\draw (1.8,4.66) node[anchor=north west] {$A_4(\frac 12, \frac{\sqrt{3}}{2})$};
		\draw (-6.4,4.66) node[anchor=north west] {$A_5(-\frac 12, \frac{\sqrt{3}}{2})$};
		\draw (-7.9,0.65) node[anchor=north west] {$A_6(-1, 0)$};
		\draw [line width=0.2 pt,dash pattern=on 1pt off 1pt] (0.,0.)-- (-2.,-3.4641016151377544);
		\draw [line width=0.2,dash pattern=on 1pt off 1pt] (0.,0.)-- (2.,-3.4641016151377544);
		\draw [line width=0.2,dash pattern=on 1pt off 1pt] (0.,0.)-- (2.,3.4641016151377544);
		\draw [line width=0.2,dash pattern=on 1pt off 1pt] (0.,0.)-- (-4.,0.);
		\draw [line width=0.2,dash pattern=on 1pt off 1pt] (0.,0.)-- (4.,0.);
		\draw [line width=0.2,dash pattern=on 1pt off 1pt] (0.,-0.02)-- (-2.,3.4441016151377544);
		\draw [line width=0.2 pt,dash pattern=on 1pt off 1pt] (-2., -3.4641)-- (4, 0);
		\draw [line width=0.2 pt,dash pattern=on 1pt off 1pt] (2., -3.4641)-- (2.,3.4641016151377544);
		\draw [line width=0.8 pt] (-4, 0)-- (2.,3.4641016151377544);
		\draw [line width=0.2 pt,dash pattern=on 1pt off 1pt] (-4, 0)-- (2.,-3.4641016151377544);
		\draw [line width=0.2 pt,dash pattern=on 1pt off 1pt] (4, 0)-- (-2.,3.4641016151377544);
		\draw [line width=0.2 pt,dash pattern=on 1pt off 1pt] (-2.,3.4641016151377544)-- (-2.,-3.4641016151377544);
		\begin{scriptsize}
			\draw [fill=ffqqqq] (-2.,-3.4641016151377544) circle (1.5pt);
			\draw [fill=ffqqqq] (2.,-3.4641016151377544) circle (1.5pt);
			\draw [fill=ffqqqq] (4.,0.) circle (1.5 pt);
			\draw [fill=ffqqqq] (2.,3.4641016151377544) circle (1.5 pt);
			\draw [fill=ffqqqq] (-2.,3.4641016151377544) circle (1.5 pt);
			\draw [fill=ffqqqq] (-4.,0.) circle (1.5pt);
			\draw [fill=ffqqqq] (-3.21525, 0.453075) circle (1.5 pt);
			\draw [fill=ffqqqq] (-2.4305, 0.906151) circle (1.5 pt);
			\draw [fill=ffqqqq] (1.21525, 3.01103) circle (1.5 pt);
			\draw [fill=ffqqqq] (0.430501, 2.55795) circle (1.5 pt);
		\end{scriptsize}
	\end{tikzpicture}
	\caption{In conditional constrained quantization with the constraint as the diagonal $A_6A_4 $ the dots represent the elements in an optimal set of $n$-points for $7\leq n\leq 10$.} \label{Fig6}
\end{figure}

\begin{prop}
	Let $\ga_n$ be a conditional constrained optimal set of $n$-points for $P$ with conditional constrained quantization error $V_n$. Then, 
	\begin{align*} \ga_9&= \gb\uu \Big\{(-0.803813,0.113269), (-0.607625,0.226538), (0.183013, 0.683013)\Big\} \\
		&\qquad \te{ with } V_9=\frac{0.0422866}{3}+\frac{1}{18}=0.0696511,  \\
		\ga_{10}&=\gb\uu \Big\{(-0.803813, 0.113269), (-0.607625,0.226538), (0.107625, 0.639488), (0.303813, 0.752757)\Big\} \\
		&\qquad \te{ with } V_{10}=\frac{0.039915}{3}+\frac{1}{18}=0.0688606
	\end{align*} 
\end{prop}

\begin{proof}
	The proof follows by Lemma~\ref{lemma78}, Example~\ref{Me83}, and Example~\ref{Me84} (see Figure~\ref{Fig6}). 
\end{proof}

\begin{main}
	Recall that the main result of this subsection is to obtain the conditional constrained optimal sets $\ga_n:=\ga_n(P)$ of $n$-points and the corresponding conditional constrained quantization errors $V_n(P)$  for all $n\geq 6$ for the uniform distribution $P$ under the constraint as the smaller diagonal $A_6A_4 $ and the conditional set $\gb$. As shown in Example~\ref{Me83} and Example~\ref{Me84}, using Methodology~\ref{method1}, we obtain $\gg_{n-3}(Q)$ and $V_{n-3}(Q)$ for the uniform distribution $Q$. Then, using Lemma~\ref{lemma78}, we obtain the sets $\ga_n$ and the corresponding conditional constrained quantization errors $V_n(P)$ for all $n\geq 6$.
\end{main}

\subsection{When the constraint is a diagonal of the larger length} \label{sub1}
As depicted in Figure~\ref{Fig1}, there
are three diagonals of the larger length. Without any loss of generality, we can choose the diagonal $A_1A_4$
of the larger length as the constraint. We now give the following lemma. 

\begin{lemma}\label{0Me421}
	Let $\ga_n$ be a conditional constrained optimal set of $n$-points for any $n\geq 7$. Then, the Voronoi region of any element in $\ga_n\ii \te{Int}(A_1A_4)$ does not contain any element from $L_2\uu L_5$. 
\end{lemma} 
\begin{proof}
	Let $\ga_n$ be a conditional constrained optimal set of $n$-points for $n \geq 7$.
	For any $(x, y) \in L_2\uu L_5$, we see that 
	\[\min_{a\in \ga_n\ii A_1A_4}\rho((x, y), a)>\min_{a\in \gb}\rho((x, y), a).\]
	Hence, the Voronoi region of any element in $\ga_n\ii \te{Int}(A_1A_4)$ does not contain any element from $L_2\uu L_5$.
\end{proof}
\begin{lemma} \label{Me0422}
	For any $n\geq 6$, a conditional constrained optimal set of $n$-points contains exactly $n$-elements. 
\end{lemma} 
\begin{proof}
	$\gb$ is a conditional constrained optimal set of six-points that contains exactly six elements. Let $n \geq 7$. 
	The equation of the diagonal $A_1A_4$ is $y=\sqrt 3 x$. The sides $L_1$ and $L_6$ are, respectively,  
	given by
	\[L_1=\{(t-\frac{1}{2},-\frac{\sqrt{3}}{2}): 0\leq t\leq 1\} \te{ and } L_6=\{(\frac{1}{2} (-t-1),\frac{1}{2} \sqrt{3} (t-1)): 0\leq t\leq 1\}.\]
	Let us consider an element $(t-\frac{1}{2},-\frac{\sqrt{3}}{2})$ on $L_1$. Its distance from the line  $y=\sqrt 3 x$ is given by $d_1:=\frac{\sqrt 3} 2t$. On the other hand, the distance of the element $(t-\frac{1}{2},-\frac{\sqrt{3}}{2})$ from $A_1(-\frac 12, -\frac{\sqrt 3}2)$ is $d_2:=t$, and from $A_2(\frac 12, -\frac{\sqrt 3}{2})$ is $d_3:=1-t$. Notice that $d_1< d_2$ for $0< t\leq 1$. On the other hand, $d_1<d_3$ for $0\leq t< \frac{2}{\sqrt{3}+2}$. 
	Hence, the possible elements in $L_1$ that can fall within the Voronoi regions of the elements in $\ga_n\ii \te{Int}(A_1A_4)$ are for $0< t< \frac{2}{\sqrt{3}+2}$. Similarly, the possible elements in $L_6$ that can fall within the Voronoi regions of the elements in $\ga_n\ii \te{Int}(A_1A_4)$ are for $0< t< \frac{2}{\sqrt{3}+2}$. Notice that $t=0$ corresponds to the point $(-\frac 12, -\frac {\sqrt 3}2)$, and $t=\frac{2}{\sqrt{3}+2}$ corresponds to the point $(\frac{7}{2}-2 \sqrt{3},-\frac{\sqrt{3}}{2})$ on the side $L_1$. If $(u, v)$ be the foot of the perpendicular from the point $(\frac{7}{2}-2 \sqrt{3},-\frac{\sqrt{3}}{2})$ on the diagonal $A_1A_4$, then we have
	\[(u, v)=\Big(\frac{1}{2}  (1-\sqrt 3 ), \frac{\sqrt 3}{2}(1-\sqrt 3 )\Big).\]
	Thus, we can conclude that the elements of $\ga_n\ii \te{Int}(A_1A_4)$, the Voronoi regions of which cover the elements from $L_1\uu L_6$ will lie on the line $y=\sqrt 3 x$ for $-\frac 12<x<\frac 12 (1-\sqrt 3 )$. Similarly, we can conclude that the elements of $\ga_n\ii \te{Int}(A_1A_4)$, the Voronoi regions of which cover the elements from $L_3\uu L_4$ will lie on the line $y=\sqrt 3 x$ for $\frac 12(\sqrt 3-1)<x<\frac{1}{2}$. Hence, if the optimal set $\ga_n$ contains fewer than $n$ elements, then by adding at least one element in $\ga_n$ which comes from $y=\sqrt 3 x$ either for $-\frac 12<x<\frac 12 (1-\sqrt 3 )$, or for $\frac 12(\sqrt 3-1)<x<\frac{1}{2}$, we can strictly reduce the quantization error, which contradicts the fact that $\ga_n$ is a conditional constrained optimal set of $n$-points. Thus, we conclude that a conditional constrained optimal set of $n$-points contains exactly $n$-elements, which completes the proof of the lemma. 
\end{proof} 

\begin{remark}
	Recall that $\ga_n$ is a conditional constrained optimal set of $n$-points for $n\geq 6$. Then, 
	\[\ga_n\ii A_1A_4=(\ga_n\ii \te{Int}(A_1A_4))\uu \set{(-\frac 12, -\frac{\sqrt 3}{2}), (\frac 12, \frac{\sqrt 3}{2})}.\]
	As mentioned in the proof of Lemma~\ref{Me0422}, we have 
	\begin{equation} \label{eq01111} \ga_n\ii A_1A_4 \ci \set{(x, y) \in A_1A_4 : -\frac 12\leq x\leq \frac 12 (1-\sqrt 3 )}\uu\set {(x, y)\in A_1A_4 : \frac 12(\sqrt 3-1)\leq x\leq \frac{1}{2}}.
	\end{equation}
\end{remark} 

\begin{remark} 
	By the expression in \eqref{eq01111}, we can conclude that for any $a\in \ga_n\ii A_1A_4$, if the Voronoi region of the element $a$ includes elements from $L_1\uu L_6$, then it will not include any element from $L_3\uu L_4$. We denote such elements in $\ga_n\ii A_1A_4$ by $(\ga_n\ii A_1A_4)^{(1)}$. Similarly, for any $a\in \ga_n\ii A_1A_4$, if the Voronoi region of the element $a$ includes elements from $L_3\uu L_4$, then it will not include any element from $L_1\uu L_6$. We denote such elements in $\ga_n\ii A_1A_4$ by $(\ga_n\ii A_1A_4)^{(2)}$.
	Notice that 
	\[\ga_n\ii A_1A_4=(\ga_n\ii A_1A_4)^{(1)}\UU (\ga_n\ii A_1A_4)^{(2)}, \te{ and } (\ga_n\ii A_1A_4)^{(1)}\II (\ga_n\ii A_1A_4)^{(2)}=\es.\]
\end{remark}

Let us now state the following lemma. The proof is similar to Lemma~\ref{lemmaMe1}, in fact, $P$ being a uniform distribution, and $L_1\uu L_6$ and $L_3\uu L_4$ are symmetrically located with respect to the perpendicular bisector of the segment $A_1A_4$, it is not difficult to show the proof. 

\begin{lemma} \label{Megha2} 
	For $n\geq 7$, let $(\ga_n\ii A_1A_4)^{(1)}$ and $(\ga_n\ii A_1A_4)^{(2)}$ be two sets as defined before. Let $\te{card}((\ga_n\ii A_1A_4)^{(1)}))=n_1$ and $\te{card}((\ga_n\ii A_1A_4)^{(2)}))=n_2$. Then, $n_1, n_2\geq 1 $ with $|n_1-n_2|\in \set{0,1}$. 
\end{lemma} 

\begin{note}
	Notice that due to symmetry if $\te{card}((\ga_n\ii A_1A_4)^{(1)})=\te{card}((\ga_n\ii A_1A_4)^{(2)})$, then $F((\ga_n\ii A_1A_4)^{(1)})=(\ga_n\ii A_1A_4)^{(2)}$, where $F: R^2\to \D R^2$ is the reflection mapping through the origin, i.e., the mapping $F$ is defined by 
	\begin{equation*} \label{eqMegha1}  F : \D R^2\to \D R^2 \te{ such that } F(x, y)=(-x, -y). 
	\end{equation*} 
	For $m\in \D N$, by $(\ga_n\ii A_1A_4)^{(i)}(m)$, it is meant the set $(\ga_n\ii A_1A_4)^{(i)}$ which contains $m$ elements, where $i=1, 2$. Hence, for any $n\in \D N$, to calculate $(\ga_n\ii A_1A_4)^{(2)}(m)$, first we will calculate $(\ga_n\ii A_1A_4)^{(1)}(m)$ and then will use the reflection property, i.e., 
	\begin{equation} \label{eqMegha2}  (\ga_n\ii A_1A_4)^{(2)}(m)=F((\ga_n\ii A_1A_4)^{(1)}(m)).
	\end{equation} 
	Recall Definition~\ref{defiMe}. Due to symmetry, we get the following error relationship: 
	\begin{align} \label{eqMegha3} 
		\left\{\begin{array} {ccc} 
			V(P; (\ga_n\ii A_1A_4)^{(1)}(m), L_1\uu L_6)=V(P; (\ga_n\ii A_1A_4)^{(2)}(m), L_3\uu L_4), \\
			V(P; (\ga_n\ii A_1A_4)^{(1)}(m), L_1)=V(P; (\ga_n\ii A_1A_4)^{(1)}(m), L_6), \\
			V(P; (\ga_n\ii A_1A_4)^{(2)}(m), L_3)=V(P; (\ga_n\ii A_1A_4)^{(2)}(m), L_4).
		\end{array}
		\right.
	\end{align} 
	For $n\geq 6$, let $V_n(P)$ be the $n$th constrained quantization error with respect to the constraint $A_1A_4$ for the probability distribution $P$. Notice that $\te{card}(\ga_n\ii A_1A_4)=n-4$. Let $(\ga_n\ii A_1A_4)^{(1)}=q$, and then $(\ga_n\ii A_1A_4)^{(2)}=n-4-q$, where $q\geq 1$.  Then,  
	\begin{align*}
		V_n(P)=V(P: (\ga_n\setminus \ga_n\ii A_1A_4), L)+V(P; \ga_n\ii A_1A_4, L_1\uu L_6\uu L_3\uu L_4),
	\end{align*}
	where due to symmetry, we have
	\begin{align*}
		&V(P; \ga_n\setminus (\ga_n\ii A_1A_4), L)=V(P; \set{A_2, A_3, A_5, A_6}, L)\\
		&=\frac 4 6 \Big(\int_{t=0}^{\frac 12}\rho((t-\frac{1}{2},-\frac{\sqrt{3}}{2}), (-\frac{1}{2},-\frac{\sqrt{3}}{2}))\, dt+\int_{t=\frac{2}{\sqrt{3}+2}}^{1}\rho((t-\frac{1}{2},-\frac{\sqrt{3}}{2}), (\frac{1}{2},-\frac{\sqrt{3}}{2}))\, dt\Big)\\
		&=\frac{13}{36}  (48 \sqrt{3}-83),
	\end{align*}
	and 
	\begin{align*}
		&V(P; \ga_n\ii A_1A_4, L_1\uu L_6\uu L_3\uu L_4)\\
		&=V(P; (\ga_n\ii A_1A_4)^{(1)}(q), L_1\uu L_6)+V(P; (\ga_n\ii A_1A_4)^{(2)}(n-4-q), L_3\uu L_4)\\
		&=2\Big(V(P; (\ga_n\ii A_1A_4)^{(1)}(q), L_1)+V(P; (\ga_n\ii A_1A_4)^{(2)}(n-4-q), L_3)\Big)\\
		&=2\Big(V(P; (\ga_n\ii A_1A_4)^{(1)}(q), L_1)+V(P; (\ga_n\ii A_1A_4)^{(1)}(n-4-q), L_1)\Big).
	\end{align*}
	Thus, we see that 
	\begin{equation} \label{eqMegha4} 
		V_n(P)=\frac{13}{36}  (48 \sqrt{3}-83)+2\Big(V(P; (\ga_n\ii A_1A_4)^{(1)}(q), L_1)+V(P; (\ga_n\ii A_1A_4)^{(1)}(n-4-q), L_1)\Big).
	\end{equation} 
\end{note}

\begin{main} \label{mainMegha} 
	Let $\te{card}(\ga_n\ii A_1A_4)^{(1)}=q$. Then, by Lemma~\ref{Megha2},  
	$|q-(n-4-q)|\in \set{0,1}$. Without any loss of generality, we can assume that $q\geq (n-4-q).$ Then, if $n-4$ is an even positive integer, then $q=\frac{n-4}2$, and if $n-4$ is an odd positive integer then, $q=\frac {n-3}2$. Notice that if $n=6$, then $\ga_n=\gb$ which is the conditional set itself with conditional constrained quantization error
	\[V_n(P)=V_6(P)=6 V(P; \gb \ii L_1, L_1)=\frac {12} 6  \int_{t=0}^{\frac 12} \rho((t-\frac{1}{2},-\frac{\sqrt{3}}{2}), (-\frac{1}{2},-\frac{\sqrt{3}}{2}))\, dt=\frac 1{12}.\]
	By the expressions given by \eqref{eqMegha2} and \eqref{eqMegha4}, we can say that to find a conditional constrained optimal set $\ga_n$ of $n$-points and the corresponding $n$th conditional constrained quantization error $V_n(P)$ for any $n\geq 7$, we need to know $(\ga_n\ii A_1A_4)^{(1)}(q)$ and $V(P; (\ga_n\ii A_1A_4)^{(1)}(q), L_1)$ for any positive integer $q\geq 2$, which can easily be obtained by the following proposition Proposition~\ref{PropMe}, while for $n-4-q=1$, we have $F((\ga_n\ii A_1A_2)^{(1)}(n-4-q))=(\frac 12, \frac{\sqrt 3}2)$. Then, the conditional constrained optimal set $\ga_n(P)$ is given by 
	\[\ga_n(P)=\gb\uu (\ga_n\ii A_1A_2)^{(1)}(q)\uu F((\ga_n\ii A_1A_2)^{(1)}(n-4-q)),\]
	and the corresponding conditional constrained quantization error $V_n(P)$ is obtained by using the formula given by \eqref{eqMegha4}.
\end{main} 
\begin{prop}\label{PropMe}
	For any positive integer $q\geq 2$, the set $(\ga_n\ii A_1A_4)^{(1)}(q)$ is given by 
	\[(\ga_n\ii A_1A_4)^{(1)}(q)=\set{(a_i, \sqrt 3 a_i) : 1\leq i\leq q},\]
	where  $a_i=(i-1)\frac{2-\sqrt{3}}{2 q-1}-\frac 12$ for $1\leq i\leq q$, and 
	\begin{align*}
		& V(P; (\ga_n\ii A_1A_4)^{(1)}(q), L_1)=\frac{4 (26-15 \sqrt{3}) (3 (q-1) q+1)}{9 (1-2 q)^2}.
	\end{align*}
\end{prop}

\begin{proof}
	For $q\geq 2$, let $(\ga_n\ii A_1A_4)^{(1)}(q):=\set{(a_i, \sqrt 3 a_i) :  1\leq i\leq q}$ such that $-\frac 12= a_1<a_2<\cdots<a_{q-1}<a_q\leq \frac{1}{2}  (1-\sqrt 3 ) $. Notice that the boundary of the Voronoi region of the element $(a_1, \sqrt 3 a_1)$ intersects $L_1$ at the element $(2(a_1+a_2)+\frac 32, -\frac {\sqrt 3}{2})$, the boundary of the Voronoi region of $(a_q, \sqrt 3 a_q)$ intersects $L_1$ at the element
	$(2(a_{q-1}+a_q)+\frac 32, -\frac {\sqrt 3}{2})$. On the other hand, the boundaries of the Voronoi regions of $(a_i,\sqrt 3 a_i)$ for $2\leq i \leq q-1$ intersect $L_1$ at the elements $(2(a_{i-1}+a_i)+\frac 32,  -\frac {\sqrt 3}{2})$ and $(2(a_{i}+a_{i+1})+\frac 32,  -\frac {\sqrt 3}{2})$. Also, recall that the union of the Voronoi regions of the elements in $(\ga_n\ii A_1A_4)^{(1)}(q)$ cover the segment $\set{(x, - \frac{\sqrt 3} 2) : -\frac 12\leq x\leq \frac{7}{2}-2 \sqrt{3}}$ from $L_1$.   
	Let us consider the following two cases:
	
	\tit{Case 1. $q=2$.}
	
	In this case, the distortion error is given by
	\begin{align*}
		&V(P; (\ga_n\ii A_1A_4)^{(1)}(q), L_1)\\
		&=\frac{1}{6}\Big(\int_{a_1}^{2(a_{1}+a_2) +\frac 32} \rho((x, -\frac{\sqrt{3}}{2}), (a_1, \sqrt{3} a_1)) \, dx+ \int_{2(a_{1}+a_2) +\frac 32}^{\frac{7}{2}-2 \sqrt{3}} \rho((x, -\frac{\sqrt{3}}{2}), (a_2, \sqrt 3 a_2)) \, dx\Big)\\
		&=\frac{1}{6}  (-4 a_2^3-8 \sqrt{3} a_2^2+10 a_2^2+8 \sqrt{3} a_2-15 a_2-34 \sqrt{3}+\frac{353}{6}),
	\end{align*}
	the minimum value of which is  $-\frac{28}{81}   (15 \sqrt{3}-26)$ and it occurs when $a_2=\frac{1}{6}-\frac{1}{\sqrt{3}}$.
	
	\tit{Case~2: $q\geq 3$.}
	
	In this case, the distortion error due to the set $\ga_n$ is given by
	\begin{align*}
		& V(P; (\ga_n\ii A_1A_4)^{(1)}(q), L_1) \\
		&=\frac{1}{6}\Big(\int_{a_1}^{2(a_{1}+a_2)+\frac 32} \rho((x, -\frac{\sqrt{3}}{2}), (a_1,\sqrt 3 a_1)) \, dx+\sum _{i=2}^{q-1} \int_{2(a_{i-1}+a_i)+\frac 32}^{2(a_{i}+a_{i+1})+\frac 32} \rho((x, -\frac{\sqrt{3}}{2}), (a_i, \sqrt 3 a_i)) \, dx \\
		&\qquad \qquad +\int_{2(a_{q-1}+a_q)+\frac 32}^{\frac{7}{2}-2 \sqrt{3}} \rho((x, -\frac{\sqrt{3}}{2}), (a_q, \sqrt 3 a_q)) \, dx\Big).
	\end{align*}
	Since $V(P; (\ga_n\ii A_1A_4)^{(1)}(q), L_1)$ gives the optimal error and is always differentiable with respect to $a_i$ for $2\leq i\leq q-1$, we have $\frac{\pa}{\pa a_i} V(P; (\ga_n\ii A_1A_4)^{(1)}(q), L_1)=0$
	yielding
	\[a_{i+1}-a_i=a_i-a_{i-1} \te{ for } 2\leq i\leq q-1\]
	implying
	\begin{equation*} 
		a_2-a_1=a_3-a_2=\cdots =a_q-a_{q-1}=k
	\end{equation*}
	for some real $k$. 
	Thus,  we have 
	\begin{equation} \label{eq200} a_2=k+a_1, \, a_3=2k+a_2, \, a_4=3k+a_3, \cdots, a_q=(q-1)k+a_1.
	\end{equation} 
	Now, putting the values of $a_i$ for $2\leq i\leq q$ in terms of $a_1$ and $k$, and as $a_1=-\frac 12$,  upon simplification, we have 
	\begin{align*}
		& V(P; (\ga_n\ii A_1A_4)^{(1)}(q), L_1) \\
		&=\frac{1}{9} \Big(2 k (q-1) (k^2 (-4 (q-2) q-3)-6 (\sqrt{3}-2) k (q-1)+3 (4 \sqrt{3}-7))-60 \sqrt{3}+104\Big).
	\end{align*}
	The above expression is minimum if $k=\frac{2-\sqrt{3}}{2 q-1}$. Thus, we have $a_i=(i-1)\frac{2-\sqrt{3}}{2 q-1}-\frac 12$ for $2\leq i\leq q$, and 
	\begin{align*}
		& V(P; (\ga_n\ii A_1A_4)^{(1)}(q), L_1)=\frac{4 (26-15 \sqrt{3}) (3 (q-1) q+1)}{9 (1-2 q)^2}
	\end{align*}
	Thus, the proof of the proposition is complete.
\end{proof}

\begin{figure}
	\begin{tikzpicture}[line cap=round,line join=round,>=triangle 45,x=0.5 cm,y=0.5 cm]
		\clip(-7.5,-5.026156858833838) rectangle (7.25308689949056,5.274903060882107);
		\draw [line width=0.8 pt] (-2.,3.4641016151377544)-- (-4.,0.)-- (-2.,-3.4641016151377544)-- (2.,-3.4641016151377544)-- (4.,0.)-- (2.,3.4641016151377544);
		\draw [line width=0.8 pt] (-2.,3.4641016151377544)-- (2.,3.4641016151377544);
		\draw (1.8,4.66) node[anchor=north west] {$A_4(\frac 12, \frac{\sqrt{3}}{2})$};
		\draw (1.60,-3.05) node[anchor=north west] {$A_2(\frac 12, -\frac{\sqrt{3}}{2})$};
		\draw (-6.66,-3.10) node[anchor=north west] {$A_1(-\frac 12, -\frac{\sqrt{3}}{2})$};
		\draw (3.8,0.65) node[anchor=north west] {$A_3(1, 0)$};
		\draw (-6.4,4.66) node[anchor=north west] {$A_5(-\frac 12, \frac{\sqrt{3}}{2})$};
		\draw (-7.9,0.65) node[anchor=north west] {$A_6(-1, 0)$};
		\draw [line width=0.2 pt,dash pattern=on 1pt off 1pt] (0.,0.)-- (-2.,-3.4641016151377544);
		\draw [line width=0.2,dash pattern=on 1pt off 1pt] (0.,0.)-- (2.,-3.4641016151377544);
		\draw [line width=0.2,dash pattern=on 1pt off 1pt] (0.,0.)-- (2.,3.4641016151377544);
		\draw [line width=0.2,dash pattern=on 1pt off 1pt] (0.,0.)-- (-4.,0.);
		\draw [line width=0.2,dash pattern=on 1pt off 1pt] (0.,0.)-- (4.,0.);
		\draw [line width=0.2,dash pattern=on 1pt off 1pt] (0.,-0.02)-- (-2.,3.4441016151377544);
		\draw [line width=0.2 pt,dash pattern=on 1pt off 1pt] (-2., -3.4641)-- (4, 0);
		\draw [line width=0.8 pt]  (-2., -3.4641)-- (2.,3.4641016151377544);
		\draw [line width=0.2 pt,dash pattern=on 1pt off 1pt]  (-4, 0)-- (2.,3.4641016151377544);
		\draw [line width=0.2 pt,dash pattern=on 1pt off 1pt] (-4, 0)-- (2.,-3.4641016151377544);
		\draw [line width=0.2 pt,dash pattern=on 1pt off 1pt] (4, 0)-- (-2.,3.4641016151377544);
		\draw [line width=0.2 pt,dash pattern=on 1pt off 1pt] (-2.,3.4641016151377544)-- (-2.,-3.4641016151377544);
		\begin{scriptsize}
			\draw [fill=ffqqqq] (-2.,-3.4641016151377544) circle (1.5pt);
			\draw [fill=ffqqqq] (2.,-3.4641016151377544) circle (1.5pt);
			\draw [fill=ffqqqq] (4.,0.) circle (1.5 pt);
			\draw [fill=ffqqqq] (2.,3.4641016151377544) circle (1.5 pt);
			\draw [fill=ffqqqq] (-2.,3.4641016151377544) circle (1.5 pt);
			\draw [fill=ffqqqq] (-4.,0.) circle (1.5pt);
			\draw [fill=ffqqqq] (-1.64273, -2.8453) circle (1.5 pt);
		\end{scriptsize}
	\end{tikzpicture}
	\begin{tikzpicture}[line cap=round,line join=round,>=triangle 45,x=0.5 cm,y=0.5 cm]
		\clip(-7.5,-5.026156858833838) rectangle (7.25308689949056,5.274903060882107);
		\draw [line width=0.8 pt] (-2.,3.4641016151377544)-- (-4.,0.)-- (-2.,-3.4641016151377544)-- (2.,-3.4641016151377544)-- (4.,0.)-- (2.,3.4641016151377544);
		\draw [line width=0.8 pt] (-2.,3.4641016151377544)-- (2.,3.4641016151377544);
		\draw (1.8,4.66) node[anchor=north west] {$A_4(\frac 12, \frac{\sqrt{3}}{2})$};
		\draw (1.60,-3.05) node[anchor=north west] {$A_2(\frac 12, -\frac{\sqrt{3}}{2})$};
		\draw (3.8,0.65) node[anchor=north west] {$A_3(1, 0)$};
		\draw (-6.66,-3.10) node[anchor=north west] {$A_1(-\frac 12, -\frac{\sqrt{3}}{2})$};
		\draw (-6.4,4.66) node[anchor=north west] {$A_5(-\frac 12, \frac{\sqrt{3}}{2})$};
		\draw (-7.9,0.65) node[anchor=north west] {$A_6(-1, 0)$};
		\draw [line width=0.2 pt,dash pattern=on 1pt off 1pt] (0.,0.)-- (-2.,-3.4641016151377544);
		\draw [line width=0.2,dash pattern=on 1pt off 1pt] (0.,0.)-- (2.,-3.4641016151377544);
		\draw [line width=0.2,dash pattern=on 1pt off 1pt] (0.,0.)-- (2.,3.4641016151377544);
		\draw [line width=0.2,dash pattern=on 1pt off 1pt] (0.,0.)-- (-4.,0.);
		\draw [line width=0.2,dash pattern=on 1pt off 1pt] (0.,0.)-- (4.,0.);
		\draw [line width=0.2,dash pattern=on 1pt off 1pt] (0.,-0.02)-- (-2.,3.4441016151377544);
		\draw [line width=0.2 pt,dash pattern=on 1pt off 1pt] (-2., -3.4641)-- (4, 0);
		\draw [line width=0.2 pt,dash pattern=on 1pt off 1pt] (2., -3.4641)-- (2.,3.4641016151377544);
		\draw [line width=0.8 pt]  (-2., -3.4641)-- (2.,3.4641016151377544);
		\draw [line width=0.2 pt,dash pattern=on 1pt off 1pt]  (-4, 0)-- (2.,3.4641016151377544);
		\draw [line width=0.2 pt,dash pattern=on 1pt off 1pt] (-4, 0)-- (2.,-3.4641016151377544);
		\draw [line width=0.2 pt,dash pattern=on 1pt off 1pt] (4, 0)-- (-2.,3.4641016151377544);
		\draw [line width=0.2 pt,dash pattern=on 1pt off 1pt] (-2.,3.4641016151377544)-- (-2.,-3.4641016151377544);
		\begin{scriptsize}
			\draw [fill=ffqqqq] (-2.,-3.4641016151377544) circle (1.5pt);
			\draw [fill=ffqqqq] (2.,-3.4641016151377544) circle (1.5pt);
			\draw [fill=ffqqqq] (4.,0.) circle (1.5 pt);
			\draw [fill=ffqqqq] (2.,3.4641016151377544) circle (1.5 pt);
			\draw [fill=ffqqqq] (-2.,3.4641016151377544) circle (1.5 pt);
			\draw [fill=ffqqqq] (-4.,0.) circle (1.5pt);
			\draw [fill=ffqqqq] (-1.64273, -2.8453) circle (1.5 pt);
			\draw [fill=ffqqqq] (1.64273, 2.8453) circle (1.5 pt);
		\end{scriptsize}
	\end{tikzpicture}
	\caption{In conditional constrained quantization with the constraint as the diagonal $A_1A_4 $ the dots represent the elements in an optimal set of $n$-points for $7\leq n\leq 8$.} \label{Fig8}
\end{figure}
Let us now give the following two examples, where $\ga_n(P)$ represents a conditional constrained optimal set of $n$ points with the corresponding $n$th conditional constrained quantization error $V_n(P)$ for $n\geq 6$. 
\begin{exam}For $n=7$, we have  
	\begin{align*}
		\ga_7(P)&=\gb\uu \set{(a, \sqrt 3 a) : a=\frac{1-2\sqrt 3} 6}
	\end{align*}
	(see Figure~\ref{Fig8}),  and by \eqref{eqMegha4},  
	\begin{align*}  
		V_7(P)=\frac{13}{36}  (48 \sqrt{3}-83)+2\Big(V(P; (\ga_7\ii A_1A_4)^{(1)}(2), L_1)+V(P; (\ga_n\ii A_1A_4)^{(1)}(1), L_1)\Big).
	\end{align*} 
	Notice that 
	\[V(P; (\ga_n\ii A_1A_4)^{(1)}(1), L_1)=\frac 16(\int_{t=0}^{\frac 12}\rho((t-\frac{1}{2},-\frac{\sqrt{3}}{2}), (-\frac{1}{2},-\frac{\sqrt{3}}{2}))\, dt=\frac{1}{144},\]
	and by Proposition~\ref{PropMe}, we have 
	$V(P; (\ga_7\ii A_1A_4)^{(1)}(2), L_1)=-\frac{28}{81} \left(15 \sqrt{3}-26\right)$. 
	Hence, 
	\[V_7(P)=\frac{13}{36}  (48 \sqrt{3}-83)+2\Big(-\frac{28}{81} (15 \sqrt{3}-26)+\frac{1}{144}\Big)=\frac{1}{648}(4512 \sqrt{3}-7765)\approx 0.0772.\]
\end{exam} 

\begin{exam}Recall Proposition~\ref{PropMe}. For $n=8$, we have  
	\begin{align*}
		\ga_8(P)&=\gb\UU \Big\{\Big ((i-1)\frac{2-\sqrt{3}}{2 q-1}-\frac 12, \sqrt 3((i-1)\frac{2-\sqrt{3}}{2 q-1}-\frac 12)\Big) : i=1, 2\Big\}\\
		&\qquad \qquad \UU F\Big(\Big\{\Big ((i-1)\frac{2-\sqrt{3}}{2 q-1}-\frac 12, \sqrt 3((i-1)\frac{2-\sqrt{3}}{2 q-1}-\frac 12)\Big) : i=1, 2\Big\}\Big)\\
		&=\gb\UU \Big\{\Big(\frac{1}{6} (1-2 \sqrt{3}), \frac{\sqrt 3}{6}(1-2 \sqrt{3})\Big), \Big(-\frac{1}{6} (1-2 \sqrt{3}),- \frac{\sqrt 3}{6}(1-2 \sqrt{3})\Big)\Big\}
	\end{align*}
	(see Figure~\ref{Fig8}),  and by \eqref{eqMegha4},  
	\begin{align*}  
		V_8(P)&=\frac{13}{36}  (48 \sqrt{3}-83)+2\Big(2 V(P; (\ga_8\ii A_1A_4)^{(1)}(2), L_1)\Big)\\
		&=\frac{13}{36}  (48 \sqrt{3}-83)+4\Big(-\frac{28}{81} (15 \sqrt{3}-26) \Big)\\
		&=\frac{1937}{324}-\frac{92}{9 \sqrt{3}}\approx 0.0766.
	\end{align*} 
\end{exam}
\begin{remark}
	Proceeding in the similar way as the above two examples, one can easily calculate the conditional constrained optimal set $\ga_n(P)$ and the corresponding $n$th conditional constrained quantization error $V_n(P)$ taking the constraint as the diagonal $A_1A_4$ for any positive integer $n\geq 9$. Similarly, they can also be calculated taking any of the other two diagonals of the larger length. 
\end{remark} 

\bibliographystyle{plain}
\bibliography{References}

\end{document}